\documentclass{gtart}
\usepackage{amsfonts}
\usepackage{amsmath}
\usepackage{amsthm}
\usepackage{pinlabel}
\usepackage{color}

\newtheorem{theorem}{Theorem}
\newtheorem{lemma}[theorem]{Lemma}
\newtheorem{proposition}[theorem]{Proposition}

\theoremstyle{definition}
\newtheorem{definition}[theorem]{Definition}

\newtheorem{remark}[theorem]{Remark}

\numberwithin{theorem}{section}

\numberwithin{equation}{section}

\def\Z{\mathbb Z}

\def\R{\mathbb R}

\newcommand{\Diff}{\mathop{\rm Diff}\nolimits}

\begin{document}
%%%%

\title{Indefinite Morse $2$--functions; broken fibrations and generalizations}
\shorttitle{Indefinite Morse $2$--functions}

\authors{David T. Gay, Robion Kirby
\footnote{This work was partially supported by a grant from the Simons Foundation (\#210381 to David Gay), by South African National Research Foundation Focus Area grant FA2007042500033 and United States National Science Foundation grants EMSW21-RTG and DMS-1207721. } 
% \\
% Robion Kirby
% \footnote{Supported in part by}
}
\address{Euclid Lab, 160 Milledge Terrace, Athens, GA 30606\\
Department of Mathematics, University of Georgia, Athens, GA 30602}
\secondaddress{University of California, Berkeley, CA 94720} 
\email{d.gay@euclidlab.org}
\secondemail{kirby@math.berkeley.edu}
 
\begin{abstract}

A Morse $2$--function is a generic smooth map from a smooth manifold to a surface. In the absence of definite folds (in which case we say that the Morse $2$--function is indefinite), these are natural generalizations of broken (Lefschetz) fibrations. We prove existence and uniqueness results for indefinite Morse $2$--functions mapping to arbitrary compact, oriented surfaces. ``Uniqueness'' means there is a set of moves which are sufficient to go between two homotopic indefinite Morse $2$--functions while remaining indefinite throughout.  We extend the existence and uniqueness results to indefinite, Morse $2$--functions with connected fibers.
\end{abstract}


\sloppy

\maketitlepage

\section{Introduction}

A Morse $2$--function on a smooth $n$--manifold $X$ is a generic smooth map from $X$ to a $2$--manifold, just as an ordinary Morse function is a generic smooth map to a $1$--manifold. The singularities are {\em folds} and {\em cusps}. Folds look locally like $(t,x_1,\ldots,x_{n-1}) \mapsto (t,f(x_1,\ldots,x_{n-1}))$ for a standard Morse singularity $f$, and cusps look locally like $(t,x_1,\ldots,x_{n-1}) \mapsto (t,f_t(x_1,\ldots,x_{n-1}))$ for a standard birth $f_t$ of a cancelling pair of Morse singularities. 

We develop techniques for working with Morse $2$--functions and generic homotopies between them, paying particular attention to (1) avoiding {\em definite folds}, in which the modelling function $f$ is a definite Morse singularity, i.e. a local extremum, and (2) guaranteeing connected fibers.
When definite folds are avoided, we say that the Morse $2$--function (or generic homotopy) is {\em indefinite}. When fibers are connected, we say that the function or homotopy is {\em fiber-connected}. The same adjectives also describe Morse functions and their homotopies, when definite singularities (local extrema) are avoided and when level sets are connected.

Let $X$ be a compact, connected, oriented, smooth ($C^\infty$) $n$--manifold and let $\Sigma$ be a compact, connected, oriented surface (with possibly empty, possibly disconnected, boundaries). We leave the non-oriented case for others to think about.
\begin{theorem}[Existence] \label{T:Existence}
 Let $g : \partial X \to \partial \Sigma$ be an indefinite, surjective, Morse function which extends to a map $G' : X \to \Sigma$. If $n>2$ and $G'_*(\pi_1(X))$ has finite index in $\pi_1(\Sigma)$, then $G'$ is homotopic rel. boundary to an indefinite Morse $2$--function $G : X \to \Sigma$. When $n>3$, if $g$ is fiber-connected and $G'_*(\pi_1(X)) = \pi_1(\Sigma)$ then we can arrange that $G$ is fiber-connected.
\end{theorem}

\begin{theorem}[Uniqueness] \label{T:Uniqueness}
 Let $G_0, G_1 : X \to \Sigma$ be indefinite Morse $2$--functions which agree on $\partial X$ and are homotopic rel. boundary. If $n>3$ then $G_0$ and $G_1$ are homotopic through an indefinite generic homotopy $G_s$. If in addition $G_0$ and $G_1$ are fiber-connected then we can also arrange that $G_s$ is fiber-connected.
\end{theorem}

These results are analogs of the following facts in ordinary Morse theory:

\begin{theorem}[Existence] \label{T:1Existence}
 A compact, connected $m$--dimensional cobordism $M$ between $F_0 \neq \emptyset$ and $F_1 \neq \emptyset$ supports an indefinite Morse function $g : M \to I=[0,1]=B^1$. If $m>2$ and $F_0$ and $F_1$ are connected then we can also arrange that $g$ is fiber-connected. Any homotopically nontrival map $g'$ from a closed connected $m$--manifold $M$ to $S^1$ is homotopic to an indefinite Morse function $g : M \to S^1$. If $m > 2$ and  $g'_*(\pi_1(M)) = \pi_1(S^1)$ then we can also arrange that $g$ is fiber-connected.
\end{theorem}

\begin{theorem}[Uniqueness] \label{T:1Uniqueness}
 In both cases above, two homotopic (rel. $\partial$) indefinite Morse functions $g_0,g_1 : M \to N^1$, where $N=B^1$ or $N=S^1$, are homotopic through an indefinite generic homotopy $g_s$, and if $m > 2$ and $g_0$ and $g_1$ are fiber-connected then we can arrange that $g_s$ is fiber-connected for all $s$.
\end{theorem}

\begin{remark}
 Because our proofs of the above theorems begin with the case of maps to $I \times I$ or $I$ and are completed with Thom-Pontrjagin type arguments, all the homotopies constructed can be taken to be homotopic to given homotopies.
\end{remark}

The motivation for generalizing Morse functions and Cerf theory~\cite{Cerf} from dimension one to dimension two comes originally from
the importance of Lefschetz fibrations in complex and symplectic
geometry, and the new ideas around broken Lefschetz fibrations.  After
LeBrun~\cite{LeBrun} and Honda~\cite{HondaNearSymplecticExistence, HondaLocalProperties} showed that a smooth $4$-manifold $X^4$
with $b^+_2 > 0$ has a {\it near symplectic form} (a closed $2$-form
$\omega$ with $\omega \wedge \omega \geq 0$ and zero only on a smooth
embedded $1$-manifold $Z$ in $X$), Auroux, Donaldson and Katzarkov~\cite{ADK} proved that such $4$-manifolds are Lefschetz pencils in the
complement of $Z$, where $Z$ mapped onto latitudes of $S^2$.
Gay and Kirby~\cite{GayKirbyBLFs}, Lekili~\cite{LekiliWrinkled}, Baykur~\cite{BaykurExist} and Akbulut and Karakurt~\cite{AkbulutKarakurt}
extended this result to all smooth, oriented, compact $4$-manifolds as
broken fibrations (not just pencils).

One aim was to define invariants by counting pseudo-holomorphic curves
in $X^4$ which limit on $Z$, as had been done in the symplectic case
by Taubes~\cite{TaubesSWGR} and Usher~\cite{Usher}.  To do this, Perutz~\cite{PerutzI, PerutzII} defined his
Lagrangian matching invariants for a broken Lefschetz fibration, but
to get invariants of the underlying smooth $4$-manifold, one needs
{\it moves} between broken Lefschetz fibrations, preserving connectedness
of fibers, under which the Lagrangian matching invariants are preserved.
Our uniqueness theorem is intended to provide these moves, and thus
provide purely topological definitions of invariants.

Note that maps $X^4 \to S^2$ are partitioned according to their
homotopy class into the elements of the cohomotopy set $\pi^2 (X^4)$,
calculated homotopically by Taylor~\cite{Taylor} and geometrically by
Kirby, Melvin and Teichner~\cite{KMT}.  It is not clear how this
partitioning relates to known invariants, but all elements are
realized by indefinite, fiber connected, Morse $2$-functions.

Theorem~\ref{T:Existence}, in the case of $\Sigma = S^2$ and without
fiber-connectedness, is originally due to Saeki \cite{Saeki}, who also
pointed out that the finiteness of the index $[G_*^{\prime}
  (\pi_1(X)):\pi_1(\Sigma)]$ is a necessary condition.  A short proof
of existence for closed $X$ to $S^2$ is sketched in~\cite{GayKirbyPNAS}.  A
significant step forward in the uniqueness case was provided by Lekili~\cite{LekiliWrinkled}
when he reintroduced singularity
theory into the subject and showed how to go back and forth between
Lefschetz singularities and cusps on fold curves.

Of course there is an extensive history behind this paper in the world of singularity theory, which is too long to present, and an extensive history in complex algebraic geometry in the study of honest Lefschetz fibrations. A purely topological precedent lies in the study of round handles; see~\cite{Asimov} and~\cite{BaykurSunukjian}, for example.

Theorem~\ref{T:Uniqueness}, when $n=4$ and without fiber-connectedness, was originally proved by Williams~\cite{Williams}. Theorem~\ref{T:1Existence} is standard, with some of it proved in~\cite{Saeki}. The $B^1$--valued (cobordism) case of Theorem~\ref{T:1Uniqueness} is an essential ingredient in developing the calculus of framed links for $3$--manifolds and thus appears in~\cite{Kirby}. It seems that the fiber-connected assertion in the $S^1$--valued case of Theorem~\ref{T:1Uniqueness} is a new result, and was originally posed to us as a question by Katrin Wehrheim and Chris Woodward.

If we remove the adjectives ``indefinite'' and ``fiber-connected'' from the above theorems then the theorems become simply the facts that Morse functions, Morse $2$--functions, and generic homotopies between them are, in fact, generic. Although in the above discussion we simply stated that Morse functions, Morse $2$--functions and the homotopies we are calling ``generic'' are actually generic, in fact the definitions we prefer are in terms of local models, and the fact that maps and homotopies with these local models are generic (and in fact stable) is a standard result in singularity theory. 

To prove Theorems~\ref{T:Existence} and~\ref{T:Uniqueness}, we spend most of our time on the case where $\Sigma=B^2$, the disk, and in fact think of $B^2$ as the square $I^2 = I \times I$. Here the natural structure on the $n$--dimensional domain $X$ of a Morse $2$--function $G : X \to I^2$ is that of a cobordism with sides from $M_0$ to $M_1$, where $M_0$ is an $(n-1)$--dimensional cobordism from $F_{00}$ to $F_{01}$ and $M_1$ is an $(n-1)$--dimensional cobordism from $F_{10}$ to $F_{11}$, with $F_{00} \cong F_{10}$ and $F_{01} \cong F_{11}$. We ask that this cobordism structure should be mapped to the cobordism structure on $I^2$ as a cobordism from $I$ to $I$, and the boundary data comes in the form of $I$--valued Morse functions on $M_0$ and $M_1$. See Figure~\ref{F:Intro}.

\begin{figure}[ht!]
\labellist
\small\hair 2pt
\pinlabel $t$ [l] at 199 3
\pinlabel $z$ [b] at 8 130
\endlabellist
\begin{center}
\includegraphics{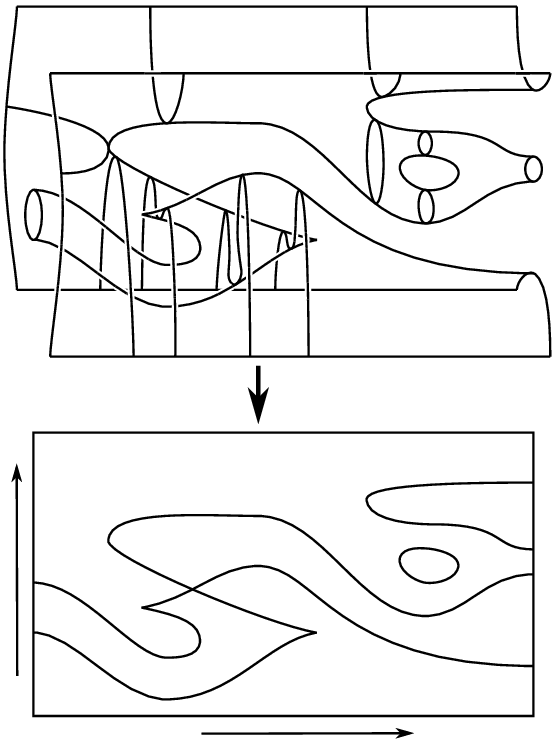}
\caption{A Morse $2$--function on a surface, mapping to the square $I \times I$.} \label{F:Intro}
\end{center}
\end{figure}

Consider the special case where $X = [0,1] \times M$ and $G(t,p) = (t,g_t(p))$, a generic homotopy between Morse functions $g_0, g_1 : M \to [0,1]$. Then removing definite folds from $G$ is the same as removing definite critical points from $g_t$, and this is done in Section~4 and also in~\cite{Kirby}. Fibers remain connected, and therefore existence is done in this special case. For uniqueness, suppose we have a generic $2$--parameter family $g_{s,t}$ between $g_{0,t}$ and $g_{1,t}$, giving a generic homotopy $G_s(t,p) = g_{s,t}(p)$. The $2$--dimensional definite folds are shown to be removable in Section~4 by use of singularity theory, in particular, use of the codimension two singularities called the {\it butterfly} and the {\it elliptic umbilic}.

For existence in the general case of a cobordism $(X,M_0,M_1) \to (I \times I, \{0\} \times I, \{1\} \times I)$, choose a Morse function $\tau : X \to I$ with no definite critical points. We will construct the indefinite Morse $2$--function $G$ so that $t \circ G = \tau$, where $(t,z)$ are coordinates on $I \times I$. Choose times $t_a$ and $t_b$ just before and after a critical value of $\tau$. At the critical point there is a standard function in local coordinates giving the descending sphere. Choose a $z$--valued Morse function $\zeta_a$ on the slice $\tau^{-1}(t_a)$ such that the descending sphere lies in a level set $\zeta_a^{-1}(z_a)$, which will mean that the descending sphere lies in the fiber of $G$ over $(t_a,z_a)$. 
Away from the local coordinates, we essentially have a product between $\tau^{-1}(t_a)$ and $\tau^{-1}(t_b)$ and the Morse function on $\tau^{-1}(t_a)$ determines a Morse function on $\tau^{-1}(t_b)$. 

We finish the existence outline by filling in the gaps between the strips around critical values by choosing Cerf graphics without definite folds; i.e. appealing to the existence of generic homotopies between ordinary Morse functions without definite critical points. Thus we get images of the folds in $I \times I$ as in Figure~\ref{F:IntroSimple}.
\begin{figure}[ht!]
\labellist
\small\hair 2pt
\pinlabel $t$ [l] at 195 5
\pinlabel $z$ [b] at 4 132
\pinlabel $t_a$ [b] at 89 151
\pinlabel $t_b$ [b] at 96 151
\pinlabel $t_a'$ [b] at 168 151
\pinlabel $t_b'$ [b] at 176 151
\endlabellist
\begin{center}
\includegraphics{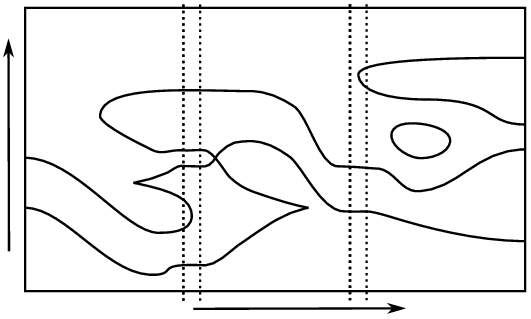}
\caption{An example illustrating a Cerf graphic in between two critical values of $\tau = t \circ G$.} \label{F:IntroSimple}
\end{center}
\end{figure}
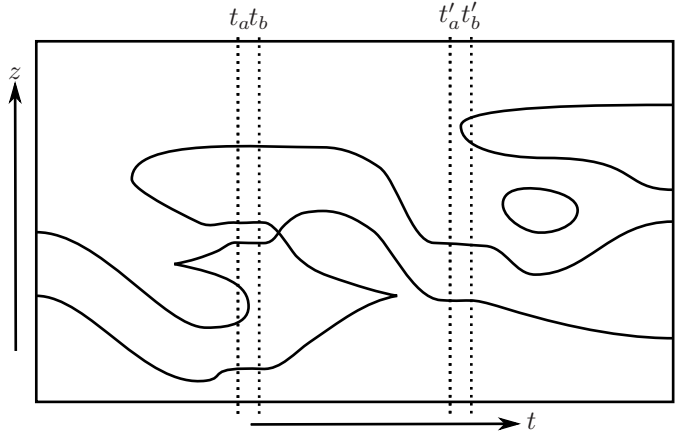

In our proof, we would produce the Morse function $\zeta_a$ mentioned above and the Cerf graphic that connects it to an earlier $\zeta$ via a very general argument which, in any particular application, should be replaced by something more explicit. An interesting and illustrative case is that of a (horizontal) Morse
function $\tau:X^4 \to R$, a critical point $p$ of index $2$, a given Morse
function $\zeta: M_a^3 \to R$ where $M_a$ is the level set of $\tau$ at $t_a$ just below
$\tau(p)$, and the attaching circle $C$ of the descending disk from $p$ lying in
$M_a$, but not in a level set of $\zeta$.  But we wish it to lie in a level
set, for that will be a fiber of the eventual Morse $2$-function. In other words, before attaching the handle associated to this critical point, we need to isotope $C$ and construct a Cerf graphic from the given $\zeta$ to a new Morse function $\zeta_a$ such that $C$ lies in a level set.  See
Figure ~\ref{F:ResolvingCrossings}.

\begin{figure}
\labellist
\small\hair 2pt
\pinlabel $z$ [r] at 8 72
\pinlabel $t$ [t] at 99 8
\pinlabel $M_a$ [b] at 44 136
\pinlabel $t_a$ [br] at 43 9
\pinlabel $t_b$ [bl] at 61 9
\pinlabel $p$ [r] at 52 72
\pinlabel $1$ [bl] at 69 32
\pinlabel $2$ [bl] at 69 57
\pinlabel $1$ [bl] at 69 90
\pinlabel $2$ [bl] at 69 115
\pinlabel $C$ [tl] at 145 56
\pinlabel $C$ [tr] at 175 56
\pinlabel $C$ [tl] at 201 33
\pinlabel $C$ [tr] at 277 33
\pinlabel {cusp $2$--handle} [b] at 239 116
\pinlabel {cusp $1$--handle} [l] at 264 51
\pinlabel {dual $1$--handle} [l] at 249 71
\endlabellist
\begin{center}
\includegraphics{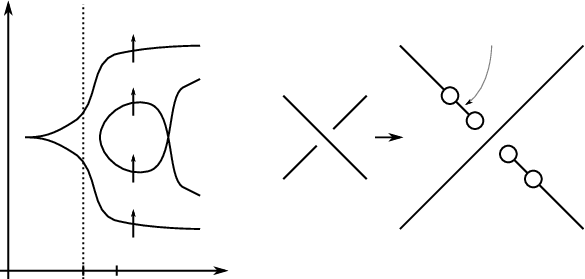}
\caption{Resolving crossings.} \label{F:ResolvingCrossings}
\end{center}
\end{figure}
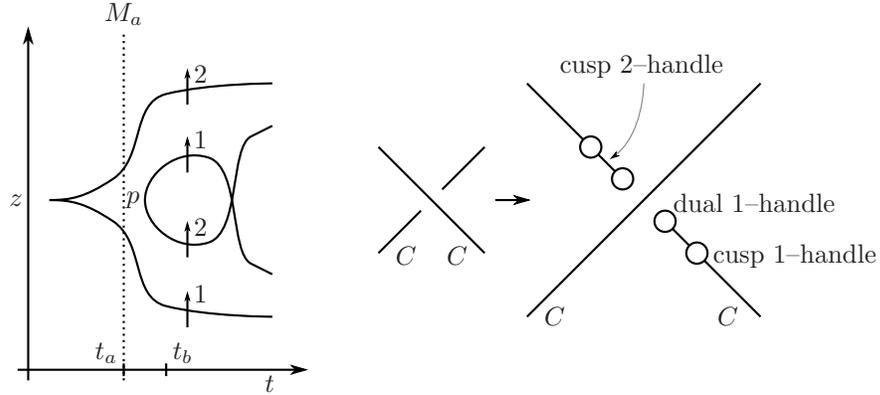

First we can isotope $C$ into a region $[-1,1] \times F$ where $F$ is a level set of $M_a$ with respect to $\zeta$. After generically projecting $C$ into $0 \times F = F$, we
get crossings, and these have to be resolved somehow.  For each crossing a
cusp consisting of a cancelling $1$-$2$--pair of critical points must be created
at a lower level of $\tau$, that is, before trying to embed $C$ in $F$, so
that they are available to remove the crossing. 

Then to construct $\zeta$ on $M_b = \tau^{-1}(t_b)$ near the crossing, we modify the $\zeta$ on $M_a$ as follows. The $1$-handle is
attached first, on either side of the arc of $C$ which is the under crossing.
Then a $2$--handle is attached along $C$, with the crossing having been resolved by sending one strand of $C$ over the cusp $1$-handle. Next we must add the dual $1$-handle of $C$.
This is attached to the $0$-sphere bundle (which is the boundary of the
normal line bundle to $C$ in $F$), and this $0$-sphere can be placed where
convenient.  In our case it is as drawn in Figure~\ref{F:ResolvingCrossings}.  Finally,
the $2$-handle of the cusp, which must cancel its $1$-handle, goes over
each of the $1$-handles once, as drawn.  This describes how $\zeta$ changes locally at each crossing while shifting from $M_a$ to $M_b$.   At each crossing of $C$ in $F$, a cusp is added and the genus of the fiber $F$ is raised by one.

Returning to the main outline, the same ideas are used for uniqueness in Section~5. We are given two maps $G_0, G_1 : (X,M_0,M_1) \to (I \times I, \{0\} \times I, \{1\} \times I)$ and $\tau_0 = t \circ G_0$ and $\tau_1 = t \circ G_1$ are homotopic indefinite Morse functions. Appealing again to the existence of indefinite homotopies between Morse functions, we get a generic homotopy $\tau_s$ from $\tau_0$ to $\tau_1$ with no definite critical points. This homotopy is a sequence of indefinite births, followed by changes of heights of critical values, followed by deaths. We first construct the homotopy $G_s$ for $s \in [0,1/4] \cup [3/4,1]$ so as to arrange that $t \circ G_s$ realizes the appropriate births and height changes for $s \in [0,1/4]$ and the appropriate deaths for $s \in [3/4,1]$. For example, a birth is achieved by the introduction of an eye followed by a kink, as in Figure~\ref{F:SquareDefFromTauBirth} (see Section~5). Then we have arranged that $t \circ G_{1/4} = t \circ G_{3/4}$ and we construct $G_s$ for $s \in [1/4,3/4]
$ keeping $t \circ G_s$ fixed. Here we end up appealing again to the existence of indefinite $2$--parameter homotopies between ordinary Morse functions, as in Section~4, where now the Morse functions are of the form $z \circ G|_{G^{-1}(\{t\} \times I)}$.

Once the case of image equal to $I \times I$ is done, it is not hard
using a {\it zig-zag} argument to extend the
theorems to other surfaces (see Section 6).  Extra care with
Thom-Pontrjagin type arguments is needed to keep fibers connected.

Although our motivation comes from the $n=4$ case, the arguments all
generalize in a straightforward manner to all dimensions, save
occasionally troubles with dimensions $\leq 3$.  It is hoped that in
other dimensions, even $3$, the theorems will be useful.

\begin{remark}
 In this paper we have assumed $X$ is $0$--connected and we have removed definite folds ($0$-- and $(n-1)$--folds) when $n \geq 3$, and removed them in $1$--parameter families when $n \geq 4$. One could speculate that,  if $X$ is $1$--connected, we could remove $0$--, $1$--, $(n-2)$-- and $(n-1)$--folds for $n \geq 5$ (existence) or $n \geq 6$ (uniqueness). In particular, simply connected $5$--manifolds would have only $2$--folds and no cusps. This speculation would further generalize to $k$--connected $n$--manifolds with $n \geq 2k+3$ (existence) or $n \geq 2k+4$ (uniqueness).
\end{remark}

We would like to thank Denis Auroux, {\.I}nan{\c{c}} Baykur, Michael Freedman, Yank{\i} Lekili, Kevin Walker, Katrin Wehrheim and Jonathan Williams for helpful discussions during the preparation of this paper. We are also extremely grateful for the care and time which our anonymous referee put into the paper, which has led to greatly improved exposition and eliminated some significant mistakes.

\section{Definitions and basic results}

We begin with the usual definition of Morse functions in terms of local models,
as a warm-up to the succeeding definitions, and also add a few slightly less
standard terms to this setting.

\begin{definition} \label{D:MorseModelFcn}
 The {\em standard index $k$ Morse model in dimension $m$} is the function $\mu_k^m(x_1, \ldots, x_m) = -x_1^2 - \ldots -x_k^2 + x_{k+1}^2 + \ldots + x_m^2$. When the ambient dimension $m$ is understood we will write $\mu_k$ instead of $\mu_k^m$. We will also abbreviate $\mu_k(x_1, \ldots, x_m)$ as $\mu_k(\mathbf{x})$.
\end{definition}

\begin{definition} \label{D:Morse}
Given an $m$--manifold $M$ and an oriented $1$--manifold $N$, a smooth function $g : M
\to N$ is {\em locally Morse} if there exist coordinates in a neighborhood of
each critical point $p$ together with coordinates in a neighborhood of $g(p)$
with respect to which $g(x_1, \ldots, x_m) =  \mu_k(\mathbf{x})$, where $k$ is the {\em index} of $p$. A {\em Morse function}
is a proper map $g : M \to N$ which is locally Morse with the additional
property that distinct critical points map to distinct critical values. When $g$
is a Morse function from $M$ to $I=[0,1]$ we imply that $M$ is given as a
cobordism from $F_0$ to $F_1$ and that $g^{-1}(0) = F_0$ and $g^{-1}(1) = F_1$. 
\end{definition}

It is a standard fact that Morse functions are stable and generic. Next
we will discuss homotopies and homotopies of homotopies between Morse
functions, and also make similar statements that homotopies satisfying certain
properties are stable and generic. These facts are only slightly less standard, and are discussed in many different references on singularity theory and Cerf theory. Probably the most comprehensive reference for the facts we mention is~\cite{HatcherWagoner}. To see these results in the more general context of singularity theory, look at~\cite{Wassermann}. For a modern exposition explicitly in a low dimensional setting, which also explains much of the motivation for this paper, we recommend~\cite{LekiliWrinkled}.

We want to discuss homotopies $g_t : M \to N$ between Morse functions which
are not necessarily Morse at intermediate times, in which case it is useful to discuss also the associated function $G : I \times M \to I \times N$ defined by $G(t,p) = (t,g_t(p))$, and its singular locus $Z_G$, the trajectory of the critical points of $g_t$. Given an $m$--manifold $M$ and a $1$--manifold $N$ with two Morse functions $g_0,g_1 : M \to N$, we are interested in homotopies $g_t : M \to N$ satisfying the following properties: The functions $g_t$ should be Morse for all but finitely many values of $t$ and, at those values $t_*$ when $g_{t_*}$ is not Morse exactly one of the following events should occur, possibly with the $t$ parameter reversed (Figure~\ref{F:CerfGraphic} illustrates these by drawing the Cerf graphic $G(Z_G)$ for a typical generic homotopy):
\begin{figure}[ht!]
\labellist
\small\hair 2pt
\pinlabel $3$ [r] at 28 117
\pinlabel $2$ [r] at 28 93
\pinlabel $1$ [r] at 28 69
\pinlabel $1$ [r] at 28 37
\pinlabel $3$ [l] at 195 113
\pinlabel $2$ [l] at 195 97
\pinlabel $1$ [l] at 195 73
\pinlabel $1$ [l] at 195 49
\pinlabel $g_t$ [r] at 6 75
\pinlabel $t$ [t] at 97 12
\pinlabel $2$ [t] at 134 84
\endlabellist
\begin{center}
\includegraphics{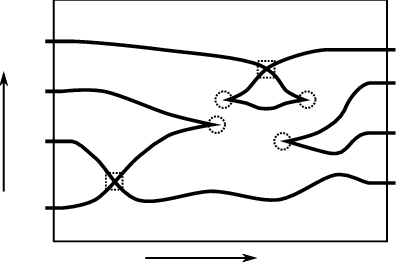}
\caption{Example of the Cerf graphic for a generic homotopy
between Morse functions, indices of critical points labelled with integers,
birth and death cusps indicated by dotted circles and critical value crossings
indicated by dotted squares.}\label{F:CerfGraphic}
\end{center}
\end{figure}
\begin{enumerate}
 \item Two critical values cross at $t_*$: More precisely, $g_{t_*}$ is locally Morse but not Morse, and $Z_G \cap ([t_* - \epsilon, t_* + \epsilon] \times M)$ is a collection of arcs on which $G$ is an embedding except for exactly one transverse double point where the images of two arcs cross. For future reference we call this event a {\em $1$--parameter crossing}, or just a {\em crossing}.
 \item A pair of cancelling critical points are born: For all $t \in [t_*-\epsilon,t_*+\epsilon]$, $g_t$ is Morse outside a ball, and inside that ball there are coordinates on domain and range (possibly varying with $t$) with respect to which $g_t(x_1, \ldots, x_m) = - x_1^2 - \ldots - x_k^2 + x_{k+1}^3 - (t-t_*) x_{k+1} + x_{k+2}^2 + \ldots + x_m^2$, with no other critical values near $0$. Thus for $t \neq t_*$, $g_t$ is Morse, but for $t < t_*$ there are no critical points in this ball, and for $t > t_*$ there are two critical points of index $k$ and $k+1$ in this ball. Note that here $G$ is injective on $Z_G \cap ([t_*-\epsilon,t_*+\epsilon] \times M)$, and $Z_G \cap ([t_*-\epsilon,t_*+\epsilon] \times M)$ is a collection of arcs all but one of which have end points at $t_*-\epsilon$ and $t_* + \epsilon$ and are smoothly embedded via $G$, and one of which has both end points at $t_* + \epsilon$ and is mapped via $G$ to a semicubical cusp in $[t_*-\epsilon,t_*+\epsilon] \times N$. For future reference we 
call this a {\em $1$--parameter birth singularity} (or {\em death singularity} when $t$ is reversed).
\end{enumerate}
It is a standard fact that homotopies satisfying these properties are
stable and generic, so for this reason:

\begin{definition} \label{D:MorseHomotopy}
 We call a homotopy $g_t : M \to N$, with $g_0$ and $g_1$ Morse, a {\em
generic homotopy between Morse functions} if $g_t$ satisfies the properties
listed in the preceding paragraph. 
\end{definition}

We distinguish the above from the following:
\begin{definition} \label{L:ArcOfMorse}
 An {\em arc of Morse functions} is a homotopy $g_t$ which is Morse for all $t$.
\end{definition}

Next we discuss homotopies $g_{s,t} : M \to N$ between generic homotopies
between Morse functions, which are not necessarily generic homotopies for
certain fixed values of $s$. In this case it is useful to consider the
associated functions $G_s : I \times M \to I \times N$ defined by $(t,p)
\mapsto (t,g_{s,t}(p))$ and $\mathcal{G} : I \times I \times M \to I \times I
\times N$ defined by $(s,t,p) \mapsto (s,t,g_{s,t}(p))$, and their singular loci
$Z_{G_s} \subset I \times M$ and $Z_{\mathcal{G}} \subset I \times I \times M$.
Given an $m$--manifold $M$ and a $1$--manifold $N$, with one generic
homotopy $g_{0,t} : M \to N$ between Morse functions $g_{0,0}$ and $g_{0,1}$
and another generic homotopy $g_{1,t} : M \to N$ between Morse functions
$g_{1,0}$ and $g_{1,1}$, we are interested in connecting these through a
$2$--parameter family $g_{s,t} : M \to N$, with $s,t \in I$, satisfying the
following conditions:
\begin{enumerate}
 \item $g_{s,0}$ is an arc of Morse functions from $g_{0,0}$ to $g_{1,0}$ and
$g_{s,1}$ is an arc of Morse functions from $g_{0,1}$ to $g_{1,1}$.
 \item For all but finitely many fixed values of $s$, $g_{s,t}$ is, in the
parameter $t$, a generic homotopy between the Morse functions $g_{s,0}$ and
$g_{s,1}$. 
 \item At those values $s_*$ when $g_{s_*,t}$ is not a generic homotopy
there is a single value $t_*$ such that $g_{s_*,t}$ is a generic homotopy for $t
\in [0,t_*)$ and for $t \in (t_*,1]$. 
 \item At each of these points $(s_*,t_*) \in I \times I$ exactly one of the following events occurs, possibly with either the $s$ or $t$ parameter reversed (some of which are illustrated in figures below by drawing
sequences of Cerf graphics $G_s(Z_{G_s})$): 
\begin{enumerate}
 \item (This event is not particularly important to us but we list it for completeness.) The function $g_{s_*,t_*}$ is locally Morse (or has a birth or death in the parameter $t$) but the $1$--parameter family $g_{s_*,t}$ does not meet the requirements to be a generic homotopy because exactly two of the events listed in Definition~\ref{D:MorseHomotopy} occur simultaneously at $t=t_*$. For example, a birth singularity may happen at the same time $t_*$ as a crossing. This phenomenon should be transverse in the obvious sense; for example, for $s<s_*$, the birth might happen before the crossing, and for $s > s_*$, the birth would then happen after the crossing. We call this event a {\em $2$--parameter coincidence}.
 \item The function $g_{s_*,t_*}$ is locally Morse but the $1$--parameter family
$g_{s_*,t}$ does not meet the requirements to be a generic homotopy
because the singular locus $Z_{G_{s_*}} \cap ([t_*-\epsilon,t_*+\epsilon]
\times M)$ is mapped into $I \times N$ via $G_{s_*}$ with a single
non-transverse quadratic double point at $t=t_*$. However, we require here that
the singular locus $Z_{\mathcal{G}} \cap ([t_*-\epsilon,t_*+\epsilon] \times
[s_*-\epsilon,s_*+\epsilon] \times M)$ is a collection of disjoint squares and
is mapped into $I \times I \times N$ via $\mathcal{G}$ with a single arc of
transverse double points. In other words, the image of $Z_{G_s}$ in $I \times N$
changes via a Reidemeister-II type move at $s=s_*$. See Figure~\ref{F:ReidII}; we call this event a {\em Reidemeister-II fold crossing}.
\begin{figure}[ht!]
\labellist
\small\hair 2pt
\pinlabel $t$ [t] at 45 1
\pinlabel $N$ [r] at 4 41
\pinlabel $s$ [b] at 99 41
\endlabellist
\begin{center}
\includegraphics{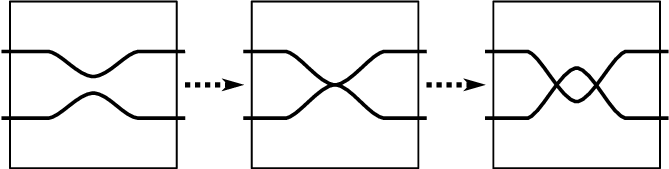}
\caption{Non-transverse double point in the singular locus for a generic
homotopy between generic homotopies between Morse functions. Note that
in general the indices of the two critical points involved can
be arbitrary.}\label{F:ReidII}
\end{center}
\end{figure}
\item The function $g_{s_*,t_*}$ is locally Morse but the $1$--parameter family
$g_{s_*,t}$ does not meet the requirements to be a generic homotopy because the
singular locus $Z_{G_{s_*}} \cap ([t_*-\epsilon,t_*+\epsilon] \times M)$ is
mapped into $I \times N$ via $G_{s_*}$ with a single transverse triple point.
However, we require here that the singular locus $Z_{\mathcal{G}} \cap
([t_*-\epsilon,t_*+\epsilon] \times [s_*-\epsilon,s_*+\epsilon] \times M)$ is a
collection of disjoint squares and is mapped  into $I \times I \times N$ via
$\mathcal{G}$ with three arcs of double points which meet transversely at the
triple point. In other words, the image of $Z_{G_s}$ in $I \times N$ is modified
via a Reidemeister-III type move. See Figure~\ref{F:ReidIII}; we call this event a {\em Reidemeister-III fold crossing}.
\begin{figure}[ht!]
\labellist
\small\hair 2pt
\pinlabel $t$ [t] at 45 1
\pinlabel $N$ [r] at 4 41
\pinlabel $s$ [b] at 103 41
\endlabellist
\begin{center}
\includegraphics{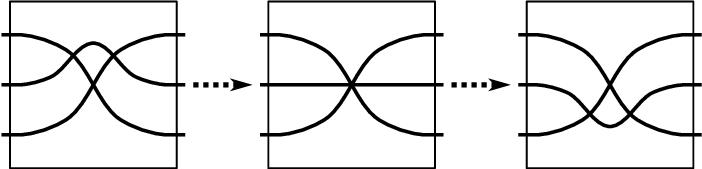}
\caption{Transverse triple point in the singular locus for a generic homotopy
between generic homotopies between Morse functions. Again, the indices
involved can be arbitrary.}\label{F:ReidIII}
\end{center}
\end{figure}
 \item The $1$--parameter family $g_{s_*,t}$ fails to be a generic homotopy
because a birth (or death) occurs at time $t_*$ at a point $p \in M$
at the same value as another Morse critical point $q$; $g_{s_*,t_*}(p) =
g_{s_*,t_*}(q)$. In other words, $G_{s_*}$ maps $Z_{G_{s_*}}$ into $I \times
N$ in such a way that a non-transverse double point occurs between a cusp and a
non-cusp point. However, here we require that the $1$--dimensional cusp locus
$C_{\mathcal{G}} \subset I \times I \times M$ and the $2$--dimensional singular
locus $Z_{\mathcal{G}} \subset I \times I \times M$ are mapped into $I \times I
\times N$ via $\mathcal{G}$ with a transverse intersection at this point. See
Figure~\ref{F:CuspCrossing}; we call this event a {\em cusp-fold crossing}. 
\begin{figure}[ht!]
\labellist
\small\hair 2pt
\pinlabel $t$ [t] at 45 1
\pinlabel $N$ [r] at 4 41
\pinlabel $s$ [b] at 103 41
\endlabellist
\begin{center}
\includegraphics{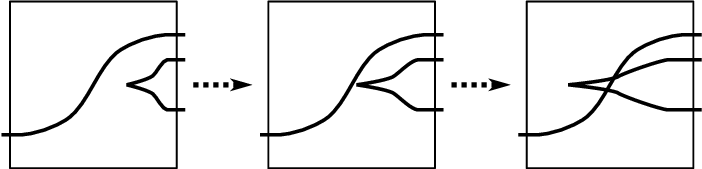}
\caption{Non-transverse double point involving a cusp, occurring in a generic
homotopy between generic homotopies between Morse functions. The only constraint
on indices is that coming from the cusp, namely that the two critical points
born at the cusp are of successive index.}
\label{F:CuspCrossing}
\end{center}
\end{figure}
 \item The function $g_{s_*,t_*}$ is Morse away from a point $p \in M$, and in
neighborhoods of $p$ and $g_{s_*,t_*}(p)$ we have coordinates with respect to
which, for $|s-s_*| < \delta$ and $|t-t_*| < \epsilon$, $g_{s,t}$ is given by
$g_{s,t}(x_1, \ldots, x_m) = -x_1^2 - \ldots -x_k^2 + x_{k+1}^3 + (t-t_*)^2
x_{k+1} - (s-s_*) x_{k+1} + x_{k+2}^2 + \ldots + x_m^2$. Furthermore, for these
$(s,t)$ there are no other singularities of $g_{s,t}$ in the inverse image of a
small neighborhood of $g_{s_*,t_*}(p)$. Geometrically, this is the birth of a
pair of cusps joined in an ``eye'' shape, involving a birth and a death of a
pair of cancelling critical points. See Figure~\ref{F:Eye}; we call this event an {\em eye birth singularity} (or {\em death} when $s$ is reversed).
\begin{figure}[ht!]
\labellist
\small\hair 2pt
\pinlabel $t$ [t] at 42 1
\pinlabel $N$ [r] at 1 41
\pinlabel $s$ [b] at 98 41
\pinlabel {no critical values} at 40 40
\pinlabel $k$ [t] at 273 30
\pinlabel $k+1$ [b] at 273 50
\endlabellist
\begin{center}
\includegraphics{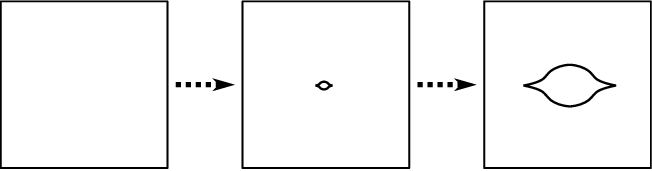}
\caption{The birth of a birth--death pair (eye birth) involving a pair of cancelling
critical points of successive index.}
\label{F:Eye}
\end{center}
\end{figure}
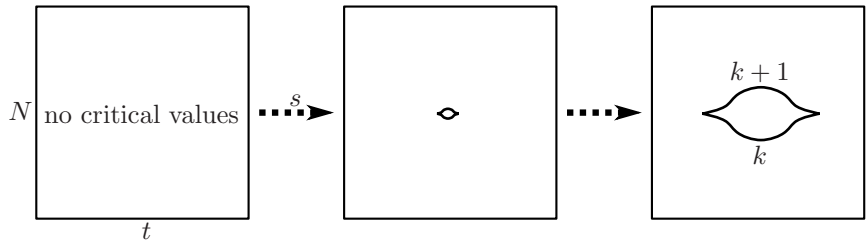
 \item The function $g_{s_*,t_*}$ is Morse away from a point $p \in M$, and in
neighborhoods of $p$ and $g_{s_*,t_*}(p)$ we have coordinates with respect to
which, for $|s-s_*| < \delta$ and $|t-t_*| < \epsilon$, $g_{s,t}$ is given by
$g_{s,t}(x_1, \ldots, x_m) = -x_1^2 - \ldots -x_k^2 + x_{k+1}^3 - (t-t_*)^2
x_{k+1} - (s-s_*) x_{k+1} + x_{k+2}^2 + \ldots + x_m^2$. Furthermore, for these
$(s,t)$ there are no other singularities of $g_{s,t}$ in the inverse image of a
small neighborhood of $g_{s_*,t_*}(p)$. Here a death and a birth of a
cancelling pair merge together, so that afterwards there is no cancellation. 
See Figure~\ref{F:Merge}; we call this event a {\em merge singularity} (or {\em unmerge} when $s$ is reversed).
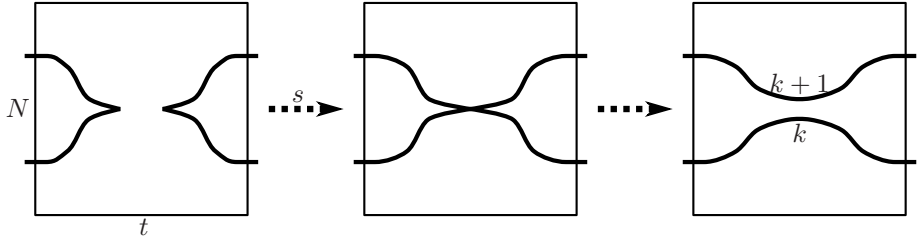
\begin{figure}[ht!]
\labellist
\small\hair 2pt
\pinlabel $t$ [t] at 45 1
\pinlabel $N$ [r] at 4 41
\pinlabel $s$ [b] at 103 41
\pinlabel $k$ [t] at 292 37
\pinlabel $k+1$ [b] at 292 44
\endlabellist
\begin{center}
\includegraphics{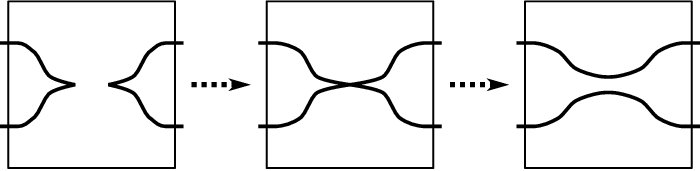}
\caption{The merge of a death--birth pair involving a pair of cancelling
critical points of successive index.}
\label{F:Merge}
\end{center}
\end{figure}
 \item The function $g_{s_*,t_*}$ is Morse away from a point $p \in M$, and in
neighborhoods of $p$ and $g_{s_*,t_*}(p)$ we have coordinates with respect to
which, for $|s-s_*| < \delta$ and $|t-t_*| < \epsilon$, $g_{s,t}$ is given by
$g_{s,t}(x_1, \ldots, x_m) = -x_1^2 - \ldots -x_k^2 + x_{k+1}^4 - (s-s_*)
x_{k+1}^2 + (t-t_*) x_{k+1} + x_{k+2}^2 + \ldots + x_m^2$. Furthermore, for
these $(s,t)$ there are no other singularities of $g_{s,t}$ in the inverse image
of a small neighborhood of $g_{s_*,t_*}(p)$. This singularity is known as a
swallowtail. See Figure~\ref{F:Swallowtail}; we call this event a {\em swallowtail birth singularity} (or {\em death} when $s$ is reversed).
\begin{figure}[ht!]
\labellist
\small\hair 2pt
\pinlabel $t$ [t] at 45 2
\pinlabel $N$ [r] at 4 42
\pinlabel $s$ [b] at 103 42
\pinlabel $k$ [t] at 293 38
\pinlabel $k+1$ [b] at 293 45
\endlabellist
\begin{center}
\scalebox{1}{\includegraphics{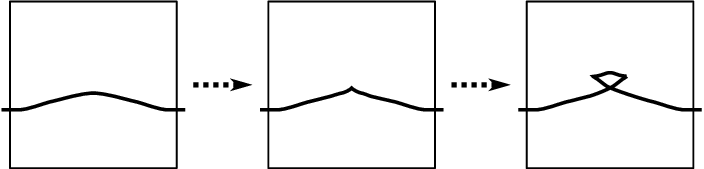}}
\caption{Birth of a swallowtail; an upside down version also occurs.}
\label{F:Swallowtail}
\end{center}
\end{figure}
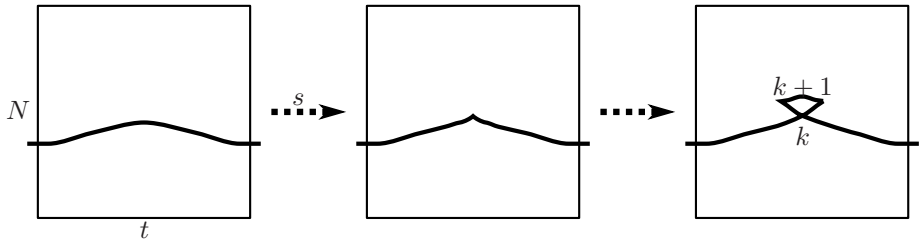
\end{enumerate}
\end{enumerate}

Note that, besides the coincidence event, we have two types of events: {\em $2$--parameter crossings} (Reidemeister-II's, Reidemeister-III's and cusp-fold's) and {\em $2$--parameter singularities} (eye births and deaths, merges and unmerges, and swallowtail births and deaths).  (As a technical point, note also that, in the definitions of the $2$--parameter singularities, the coordinates in which the homotopy of homotopies takes on the standard models may vary with $s$ and $t$, and also the parametrization of $t$ may depend on $s$.)

It is also standard that such homotopies of homotopies are generic and stable,
and so for this reason:

\begin{definition}
 A homotopy $g_{s,t}$ between generic homotopies $g_{0,t}$ and
$g_{1,t}$ is a {\em generic homotopy of homotopies} if it satisfies the
properties described above. If $g_{0,0}=g_{1,0}$ and
$g_{0,1} = g_{1,1}$, we can also ask that $g_{s,0} = g_{0,0}$ and
$g_{s,1}=g_{0,1}$ for all $s$, in which case we say that $g_{s,t}$ is a
generic homotopy {\em with fixed endpoints}. 
\end{definition}

Again, we distinguish this from the following:
\begin{definition} \label{D:ArcOfHomotopies}
 An {\em arc of generic homotopies} is a homotopy of homotopies $g_{s,t}$ which, for each fixed value of $s$, is a generic homotopy in the parameter $t$.
\end{definition}

\begin{definition}
 Given an $n$--manifold $X$ and a $2$--manifold $\Sigma$, a smooth proper map
$G : X \to \Sigma$ is a {\em Morse $2$--function} if for each $q \in \Sigma$
there is a compact neighborhood $S$ of $q$ with a diffeomorphism $\psi : S \to
I \times I$ and a diffeomorphism $\phi : G^{-1}(S) \to I \times M$, for an
$(n-1)$--manifold $M$, such that $\psi \circ G \circ \phi^{-1} : I \times M \to I
\times I$ is of the form $(t,p) \mapsto (t,g_t(p))$ for some generic homotopy
between Morse functions $g_t : M \to I$. A singular point for $G$ is called a {\em fold point} if the homotopy used to model $G$ at that point can actually be taken to be Morse, and is called a {\em cusp point} if the homotopy has a birth or death at that point. An arc of fold points is called a {\em fold}. When $\Sigma$ is given as a cobordism
between $1$--manifolds $N_0$ and $N_1$ then $X$ should be given as a cobordism
between $(n-1)$--manifolds $M_0$ and $M_1$, with $G^{-1}(N_i)=M_i$ and with
$G|_{M_i} : M_i \to N_i$ a Morse function. When $\Sigma$ is given as a
cobordism between cobordisms (in particular, when $\Sigma = I^2$, a cobordism
from $I$ to $I$, with $I$ being a cobordism from $\{0\}$ to $\{1\}$), then $X$
should also be given as such a relative cobordism, with all the cobordism
structure preserved by $G$. For us the structure of a relative cobordism
includes an explicit product structure on the sides, and this should also be
respected by $G$. In particular, there should be no critical points along the
side of the cobordism. 
\end{definition}

\begin{remark}
The important thing to understand here is that Morse $2$--functions look locally like generic homotopies between Morse functions, but that there is no global time direction. Note that the {\em index} of a fold is not well defined, but that if we choose a transverse direction to the fold, and consider local models $(t,p) \mapsto (t,g_t(p))$ in which the second coordinate in the range is given by this transverse direction, then we do have a well defined index. In figures, we will indicate this by drawing a small arrow transverse to the fold and labelling it with the index. If, however, we are drawing a Cerf graphic, then it is understood that the transverse direction is up, and we will label folds (arcs of critical points) with indices without indicating the arrow. We illustrate these conventions in Figure~\ref{F:Morse2FcnVsCerf}, which show the images of the singular locus for, on the left, a hypothetical Morse $2$--function mapping to a genus $2$ surface and, on the right, a generic homotopy between Morse 
functions.
\begin{figure}[ht!]
\labellist
\small\hair 2pt
\pinlabel $1$ [r] at 21 83 
\pinlabel $1$ [tl] at 39 29
\pinlabel $1$ [t] at 121 20
\pinlabel $2$ [tr] at 160 69
\pinlabel $0$ [r] at 92 44
\pinlabel $3$ [bl] at 90 69
\pinlabel $1$ [b] at 131 74
\pinlabel $1$ [t] at 273 20
\pinlabel $1$ [t] at 231 20
\pinlabel $2$ [b] at 257 78
\pinlabel $3$ [bl] at 221 90
\endlabellist
\begin{center}
\includegraphics{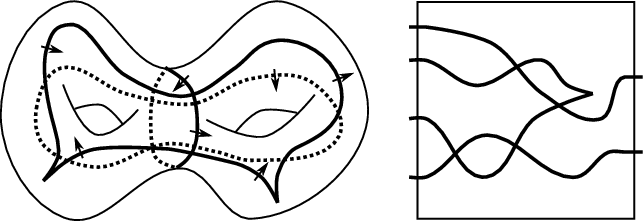}
\caption{Morse $2$--functions versus Cerf graphics, and index labelling conventions. On the left, we are mapping from a $4$--manifold to a genus $2$ surface. On the right, we are illustrating a generic homotopy between $I$--valued Morse functions on a $3$--manifold. }
\label{F:Morse2FcnVsCerf}
\end{center}
\end{figure}
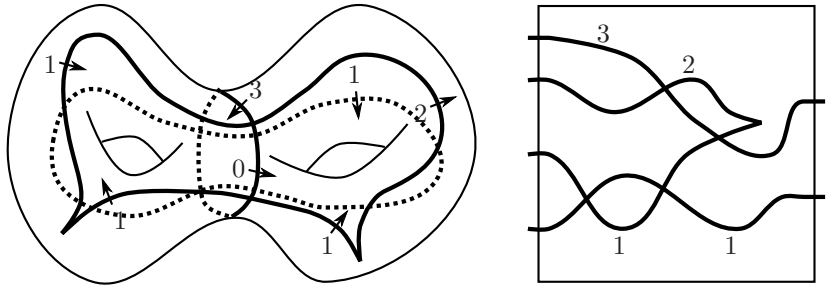
\end{remark}

\begin{definition}
A $1$--parameter family $G_s : X \to \Sigma$ is a {\em generic homotopy
between Morse $2$--functions} if, for each $q \in \Sigma$ and each $s_* \in I$
there is an $\epsilon>0$ and a compact neighborhood $S$ of $q$ with a
diffeomorphism $\psi : S \to I \times I$ and a $1$--parameter family of
diffeomorphisms $\phi_s : G_s^{-1}(S) \to I \times M$, for an
$(n-1)$--manifold $M$ and for $|s-s_*| < \epsilon$, such that $\psi \circ G_s \circ \phi_s^{-1} : I \times M \to I \times I$ is of the form $(t,p) \mapsto (t,g_{s,t}(p))$ for some generic homotopy of homotopies $g_{s,t} : M \to I$. Generic homotopies of Morse $2$--functions $G_s : X \to \Sigma$ are expected to be constant (independent of $s$) on $\partial X$.
\end{definition}

Again, although our terminology is not standard, it is a standard fact that
Morse $2$--functions and generic homotopies of Morse
$2$--functions are stable and generic. This is mostly explained in Section~4 and Appendix~A of~\cite{LekiliWrinkled}.

We will be interested, for most of this paper, in the special case of Morse
$2$--functions mapping to $I^2$, seen as a cobordism from $\{0\} \times I$ to
$\{1\} \times I$. We use coordinates $(t,z)$ on $I^2$, i.e. $t$ is the
horizontal axis. Here it is useful to impose one extra genericity condition:

\begin{definition} \label{D:SquareM2F}
 Suppose $X^n$ is a cobordism from $M_0$ to $M_1$, with each $M_i$ a cobordism
from $F_{i0}$ to $F_{i1}$. A {\em square Morse $2$--function} on $X$ is a Morse
$2$--function $G : X \to I^2$, respecting the cobordism structure, with no critical values in $I \times \{0,1\}$, such that
the projection onto the horizontal axis, $t \circ G : X \to I$, is itself a
Morse function. In particular, there is a parametrization of the sides $G^{-1}(I \times \{0,1\})$ as $I \times (F_{00} \amalg F_{10})$ with respect to which the horizontal projection $t \circ G$ restricts as projection to $I$. Homotopies between square Morse $2$--functions are assumed to maintain this last condition.
\end{definition}

It is not hard to see that, amongst Morse $2$--functions mapping to $I^2$
respecting the cobordism structures on domain and range, square Morse
$2$--functions are generic and stable.

\begin{definition}
 A Morse function is {\em indefinite} if there are no critical points of
minimal or maximal index, i.e. no critical points with the local model $(x_1, \ldots, x_m) \mapsto \pm (x_1^2 + \ldots + x_m^2)$. A generic homotopy, or generic homotopy of homotopies, of Morse functions is {\em indefinite} if it is indefinite at all parameter values at which it is Morse. A Morse $2$--function, resp. generic homotopy of Morse $2$--functions, is {\em indefinite} if it can always be locally modelled, as in the definition, by an indefinite generic homotopy, resp. generic homotopy of homotopies.
\end{definition}

The following definition will be useful when we want to make assertions about the connectedness of fibers:
\begin{definition}
 A Morse function $g : M \to I$ is {\em ordered} if, given two critical points $p, q \in M$ with indices $i,j$, respectively, if $i < j$ then $g(p) < g(q)$. A generic homotopy or generic homotopy of homotopies is ordered if it is ordered at all parameter values at which it is Morse. A Morse function (or homotopy or homotopy of homotopies) is {\em almost ordered} if, whenever $i < j-1$, we have $g(p) < g(q)$. 
\end{definition}
(Note that it is not clear how to generalize this definition to Morse
$2$--functions.) We leave the proof of the following observation to the reader: 

\begin{lemma}
Consider a Morse function $g : M^m \to I$, with $M$ a cobordism from $F_0$ to $F_1$, and with $F_0$ and $F_1$ both connected. If $m \geq 3$ and $g$ is indefinite and ordered then all level sets of $g$ will be connected. If $m \geq 4$ and $g$ is indefinite and almost ordered then the level sets will all be connected.
\end{lemma}

\section{An extended example}

An important example in dimension $4$, first described in Section~8.2 of~\cite{ADK}, concerns Morse $2$--functions over $S^2$ with some fibers being torus fibers. Because $\Diff (S^1 \times S^1)$ is not simply connected, a neighborhood $B^2 \times S^1 \times S^1$ of a torus fiber can be removed and glued back in via a nontrivial loop in $\Diff (S^1 \times S^1)$, i.e. by performing a logarithmic transform on the fiber. (See also~\cite{BaykurSunukjian}.) Thus we may change the $4$--manifold without changing the data of folds, fibers and attaching maps on $S^2$. The example in~\cite{ADK} involves, in particular, an indefinite Morse $2$--function $S^4 \to S^2$ which is homotopically trivial and can be obtained from a definite Morse $2$--function $S^4 \to B^4 \to B^2 \hookrightarrow S^2$ by flipping a circle of $0$--folds to a circle of $1$--folds as illustrated in Figure~\ref{F:ZeroToOneSingularLoci}. In this section, we show in detail how this happens in dimension $3$, in which case the fiber in question is 
$S^0 \times S^1$, we show how the nontrivial loop arises, and explain that the example generalizes to arbitrary dimensions $n \geq 3$, with fiber $S^{n-3} \times S^1$. This example also illustrates the important ideas, used throughout the paper, associated with thinking of a disk $B^2$ in the base as $I \times I$.

We begin with the following simple example of a Morse $2$--function $G : S^1 \times \R^2 \to \R^2$: Using cartesian coordinates on $\R^2$ in $S^1 \times \R^2$ and polar coordinates on the range $\R^2$, $G$ is defined by $G(\theta,x_1,x_2) = (1/2+(x_1^2+x_2^2)/2,\theta)$. The singular set is a single circle of definite folds at $S^1 \times \{(0,0)\}$ and is embedded via $G$ into $\R^2$ as the circle of radius $1/2$. Figure~\ref{F:ZeroToOneStart} illustrates this map by showing the image of the fold locus as a dark circle, with paraboloid fibers over rays emanating from the origin, showing clearly that the total space is $S^1 \times \R^2$. 
\begin{figure}
\labellist
\small\hair 2pt

\endlabellist
\begin{center}
\includegraphics{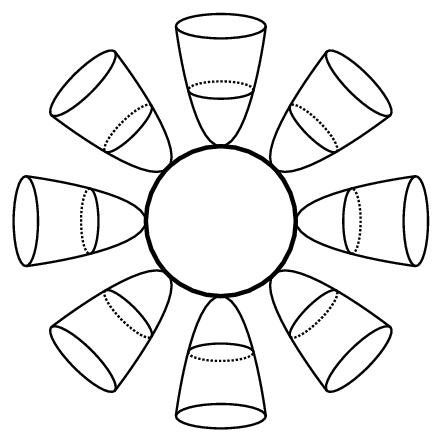}
\caption{A Morse $2$--function $G : S^1 \times \R^2 \to \R^2$ with a single definite fold circle.} 
\label{F:ZeroToOneStart}
\end{center}
\end{figure}

Now let $B$ be the square $[-1,1] \times [-1,1] \subset \R^2$ and let $X = G^{-1}(B)$. Then $X$ is a solid torus seen as a cobordism from $G^{-1}(\{-1\} \times [-1,1])$ to $G^{-1}(\{1\} \times [-1,1])$, both of which are diffeomorphic to $[-1,1] \times S^1$, and $G$ is a Morse $2$--function from $X$ to $B$. The first row in Figure~\ref{F:ZeroToOneA} illustrates ``vertical slices'' of this map, i.e. the inverse images of vertical line segments $\{t\} \times [-1,1]$; the reader should take a moment to reconcile this with Figure~\ref{F:ZeroToOneStart}, which shows the inverse images of rays from the origin.
Figures~\ref{F:ZeroToOneA} and~\ref{F:ZeroToOneB} then illustrate a homotopy $G_s$ beginning with the map $G_0=G$ described above. Each row of surfaces illustrates $G_s$ for a fixed $s$, beginning with $G_0$ in the top row of Figure~\ref{F:ZeroToOneA}. We have chosen six representative values of $s$, hence six rows. (Figure~\ref{F:ZeroToOneZoom} enlarges two regions in Figure~\ref{F:ZeroToOneB} just to illustrate the detail there carefully.) In each row (corresponding to a fixed value of $s$), the surfaces illustrated are each of the form $M_{s,t} = G_s^{-1}(\{t\} \times [-1,1])$, for nine representative values of $t$ ranging from $-1$ to $1$. Each surface $M_{s,t}$ is drawn embedded in $\R^3$; each embedding is such that the function $G_s |_{M_{s,t}} : M_{s,t} \to \{t\} \times [-1,1]$ is the height function, projection to the $z$--axis. Thus the critical locus of each $G_s$ can be seen as the trace of the critical points of each $G_s|{M_{s,t}}$ as $t$ ranges from $-1$ to $1$.

\begin{figure}%[ht!]
\labellist
\small\hair 2pt
\pinlabel $X$ [t] at 180 393
\pinlabel $Y$ [b] at 180 336
\pinlabel $X$ [tl] at 70 187
\pinlabel $Y$ [br] at 45 247
\pinlabel $a$ [r] at 102 200
\pinlabel $b$ [l] at 94 234
\pinlabel $c$ [r] at 265 200
\pinlabel $d$ [l] at 257 234
\endlabellist
\begin{center}
\includegraphics{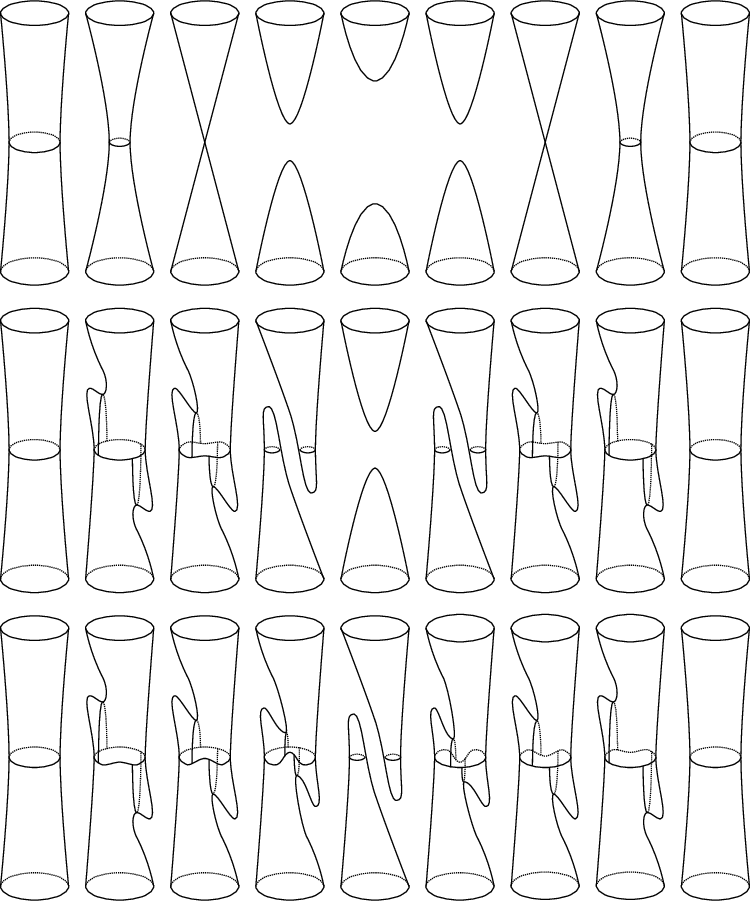}
\caption{The first half of a $1$--parameter family of Morse $2$--functions on $S^1 \times B^2$. Critical points are labelled for coordination with following figures. From the second row to the third row, we have broken the symmetry of the mid-level circle to prepare for the next move, going to the first row of Figure~\ref{F:ZeroToOneB}. Looking at the three middle surfaces in the bottom row, we see, from left to right, a $3$--dimensional $2$--handle attached, surgering the mid-level circle to two circles, followed by a $3$--dimensional $1$--handle which reattaches the two circles. 
}
\label{F:ZeroToOneA}
\end{center}
\end{figure}

\begin{figure}%[ht!]
\labellist
\small\hair 2pt
\endlabellist
\begin{center}
\includegraphics{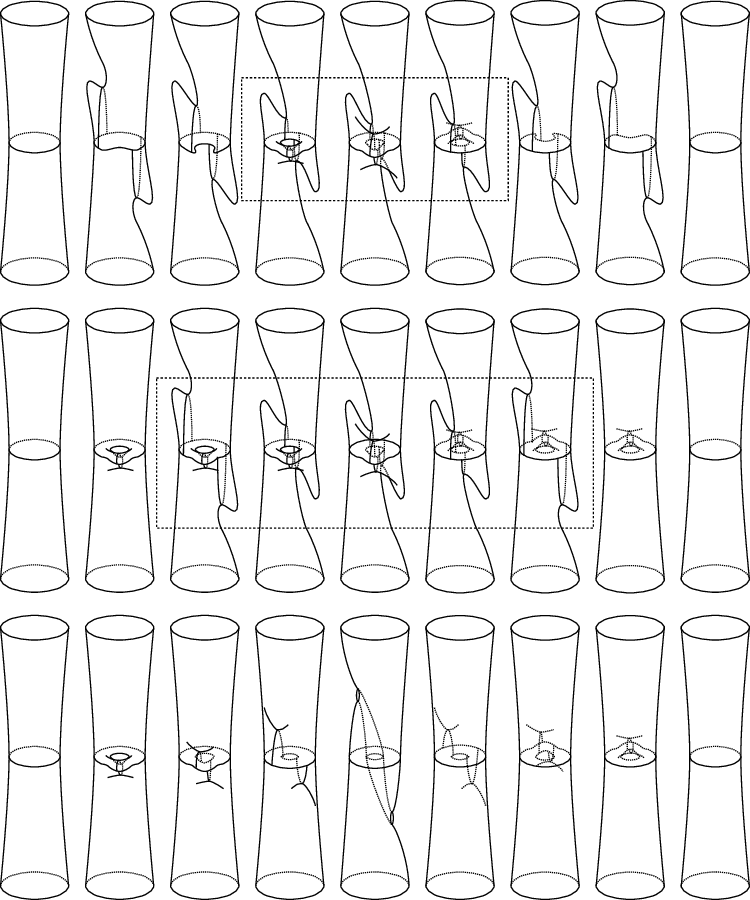}
\caption{The second half of a $1$--parameter family of Morse $2$--functions on $S^1 \times B^2$, with definite folds at the beginning and indefinite folds at the end. The two regions enclosed in boxes are shown enlarged in Figure~\ref{F:ZeroToOneZoom}. The $3$--dimensional $1$--handle mentioned in the caption for Figure~\ref{F:ZeroToOneA} is seen in the first row between the third and fourth surfaces, while the $2$--handle is seen in the first row between the sixth and seventh surface.}
\label{F:ZeroToOneB}
\end{center}
\end{figure}

\begin{figure}%[ht!]
\labellist
\small\hair 2pt
\pinlabel $X$ [tl] at 107 166
\pinlabel $Y$ [br] at 31 278
\pinlabel $a$ [l] at 79 201
\pinlabel $b$ [l] at 64 251
\pinlabel $c$ [tr] at 70 201
\pinlabel $d$ [b] at 67 218
\pinlabel $Y$ [br] at 17 134
\pinlabel $b$ [l] at 29 130
\pinlabel $d$ [b] at 42 72
\pinlabel $X$ [tl] at 70 18
\pinlabel $a$ [r] at 57 25
\pinlabel $c$ [t] at 43 59
\endlabellist
\begin{center}
\includegraphics{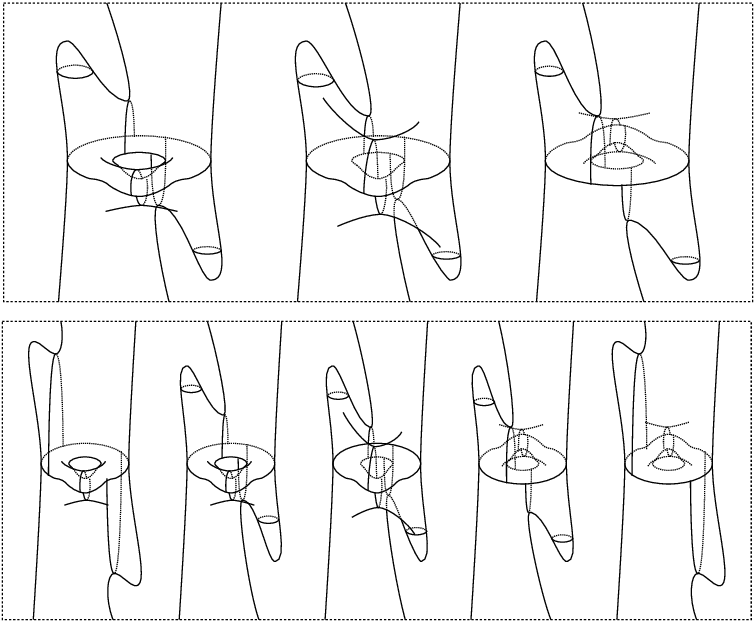}
\caption{Zooming in on two regions of Figure~\ref{F:ZeroToOneB}. Critical points are labelled corresponding to labels in Figure~\ref{F:ZeroToOneA} and Figure~\ref{F:ZeroToOneSingularLoci}. An important point to note here is that, in the second row, going from the first surface to the second, we see two handle slides: $b$ slides over $d$ and $a$ slides under $c$. And, of course, we have the symmetric slides at the other end of that row.}
\label{F:ZeroToOneZoom}
\end{center}
\end{figure}

We can extract from these figures two related sequences of diagrams. The first, Figure~\ref{F:ZeroToOneSingularLoci}, illustrates the images of the singular loci in the base $[-1,1] \times [-1,1]$ for $G_s$ for each of the six values of $s$ in Figures~\ref{F:ZeroToOneA} and ~\ref{F:ZeroToOneB}. The second, Figure~\ref{F:ZeroToOneHandles}, indicates how the respective ascending and descending manifolds of the critical points, for the vertical height functions $G_s|_{M_{s,t}}$, intersect the middle-level $1$--manifold in each surface $M_{s,t}$ from Figures~\ref{F:ZeroToOneA} and~\ref{F:ZeroToOneB}. Here we have only drawn the diagrams for the final five values of $s$ and for the middle five values of $t$.

\begin{figure}%[ht!]
\labellist
\small\hair 2pt
\pinlabel $Y$ [b] at 81 280
\pinlabel $X$ [t] at 81 315
\pinlabel $Y$ [b] at 109 194
\pinlabel $X$ [t] at 109 184
\pinlabel $a$ [b] at 38 233
\pinlabel $b$ [t] at 40 254
\pinlabel $c$ [b] at 169 181
\pinlabel $d$ [t] at 175 200
\pinlabel $a=c$ [b] at 108 12
\pinlabel $b=d$ [t] at 108 43
\endlabellist
\begin{center}
\includegraphics{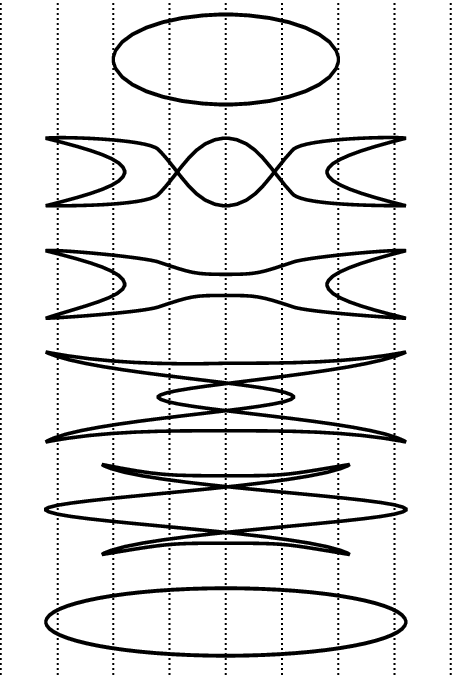}
\caption{The singular loci of $G_s$ for each of the six values of $s$ corresponding to the rows of Figures~\ref{F:ZeroToOneA} and~\ref{F:ZeroToOneB}. The nine values of $t$ corresponding to the columns are indicated by vertical dotted lines. The correspondence with the critical points in Figures~\ref{F:ZeroToOneA} and~\ref{F:ZeroToOneB} is indicated by the letter labels. Capital letters $X$ and $Y$ indicate definite folds ($2$--dimensional $0$-- and $2$--handles) while lower-case letters $a$, $b$, $c$ and $d$ indicate indefinite folds ($2$--dimensional $1$--handles). Note that by the end of the homotopy, $a$ and $c$ have become the same fold and $b$ and $d$ have become the same fold; this arises due to the cancellation of the two ``swallowtails'' at the top and bottom of the preceding singular locus.}
\label{F:ZeroToOneSingularLoci}
\end{center}
\end{figure}

\begin{figure}%[ht!]
\labellist
\small\hair 2pt
\pinlabel $b$ [tr] at 31 335
\pinlabel $b$ [br] at 31 313
\pinlabel $a$ [tl] at 41 335
\pinlabel $a$ [bl] at 41 313
\pinlabel $Y$ [r] at 11 325
\pinlabel $X$ [l] at 62 325
\pinlabel $d$ [tr] at 319 335
\pinlabel $d$ [br] at 319 313
\pinlabel $c$ [tl] at 329 335
\pinlabel $c$ [bl] at 329 313
\pinlabel $Y$ [r] at 298 325
\pinlabel $X$ [l] at 349 325
\pinlabel $Y$ [b] at 90 334
\pinlabel $X$ [b] at 126 334
\pinlabel $Y$ [b] at 233 334
\pinlabel $X$ [b] at 269 334
\pinlabel $b$ [r] at 25 180
\pinlabel $a$ [l] at 46 180
\pinlabel $d$ [r] at 106 167
\pinlabel $c$ [l] at 111 167
\pinlabel $b$ [r] at 103 187
\pinlabel $a$ [l] at 114 187
\pinlabel $d$ [r] at 311 177
\pinlabel $c$ [l] at 334 177
\pinlabel $d$ [r] at 32 97
\pinlabel $c$ [l] at 40 97
\pinlabel $b$ [l] at 20 112
\pinlabel $a$ [r] at 54 112
\pinlabel $Y$ [r] at 14 111
\pinlabel $X$ [l] at 56 111
\endlabellist
\begin{center}
\includegraphics{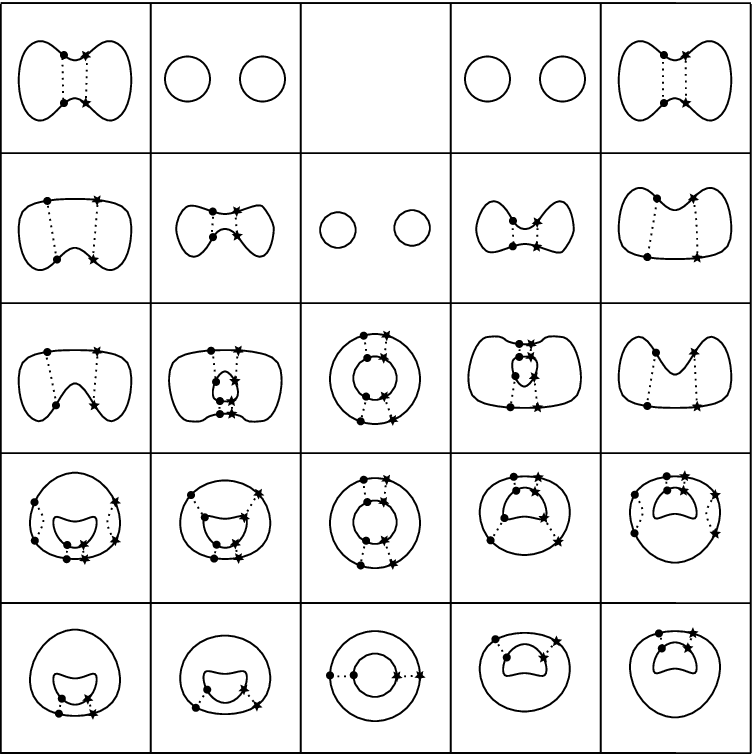}
\caption{Handle attachment data for a $5 \times 5$ block from Figures~\ref{F:ZeroToOneA} and~\ref{F:ZeroToOneB}. In each case we have drawn the middle-level $1$--manifold for the surface, and shown where the ascending and descending manifolds of the index $1$ critical points (for the vertical height function) intersect this $1$--manifold. The circles indicate descending manifolds and the stars indicate ascending manifolds. We have also indicated the ascending manifold for the $0$--handle $X$ and the descending manifold for the $2$--handle $Y$. Note again the handle slides in the fourth row. By the last row, the $0$-- and $2$--handles have been cancelled, the $1$--handles $b$ and $d$ can no longer be distinguished, and the $1$--handles $a$ and $c$ can no longer be distinguished.}
\label{F:ZeroToOneHandles}
\end{center}
\end{figure}

There are several key features to note here:
\begin{enumerate}
 \item For each Morse $2$--function $G_s$, we can consider the function $t \circ G_s$, where $t : \R^2 \to \R$ is projection to the horizontal axis. In our examples, this ``horizontal'' function is in fact (locally) an ordinary Morse function, which we call the ``horizontal Morse function'' associated to $G_s$. The critical points of $t \circ G_s$ occur at precisely the vertical tangencies of the singular loci, as illustrated in Figure~\ref{F:ZeroToOneSingularLoci}. These critical points should not be confused with the ($2$--dimensional) critical points of the vertical Morse function $z \circ G_s|_{M_{s,t}}$ on each surface $M_{s,t}$ in Figures~\ref{F:ZeroToOneA} and~\ref{F:ZeroToOneB}. Looking at how $t \circ G_s$ varies with $s$, in the beginning, we have a ($3$--dimensional) critical point of index $2$ on the left and a ($3$--dimensional) critical point of index $1$ on the right. By the fourth row, the index $1$ critical point has moved to the left and the index $2$ critical point to the right.
 \item The first Morse $2$--function $G_0$ has only definite folds, the intermediate functions have both definite and indefinite folds, and the final function $G_1$ has only indefinite folds.
 \item At $G_1$, the fiber over points inside the circle of indefinite folds is $S^0 \times S^1$ and, considering rays going out from the center point of the circle, we see, over each such ray, a $2$--dimensional $1$--handle attached along $S^0 \times \{p\}$, where the point $p$ moves once around the $S^1$ as the ray rotates once around the center point.
 \item This entire example generalizes to a generic homotopy $G_s : X^n = S^1 \times B^{n-1} \to B^2$. The starting point $G_0$ is easy to describe in coordinates, exactly as we have done here for the case $n=3$. The final map $G_1$ is harder to see in higher dimensions but this example makes it clear that, in the end, we get a circle of indefinite folds with the fiber over points inside the circle being $S^{n-3} \times S^1$. Furthermore, over rays going out from the center we see $(n-1)$--dimensional $(n-2)$--handles attached along $S^{n-3} \times \{p\}$, where $p$ rotates once around $S^1$ as the ray rotates once around the center point. When $n=4$, this example was already seen in~\cite{ADK} in their description of a broken fibration of $S^4$ over $S^2$. This example also highlights a subtlety involved in reading off information about the total space of a Morse $2$--function from data on the base; this subtlety is discussed in more detail in~\cite{GayKirbyPNAS}.
 \item Such an example can be placed anywhere in a Morse $2$--function by adding a cancelling $0$--$1$ round handle pair along any loop in $X^n$ which maps to an embedded circle bounding a disk in the base, so that the image of the $0$--fold ends up on the inside of the $1$-fold. After that, the round $0$--handle can be traded for a round $(n-2)$--handle as we have seen here. Adding this cancelling $0$--$1$ round handle pair and then trading the round $0$--handle is a homotopy that starts and ends without definite folds but which passes through definite folds during the homotopy. A worthwhile example for the reader to consider, when following the proofs in this paper, is how to carry out this homotopy without definite folds. (The authors have not done this.)
\end{enumerate}

\section{Theorems about $I$--valued Morse functions on cobordisms}

In this section we will prove Theorems~\ref{T:1Existence} and~\ref{T:1Uniqueness} in the case where the base is the interval $I$, and we will prove Theorems~\ref{T:Existence} and~\ref{T:Uniqueness} in the case where the total space is $X^n = I \times M^{n-1}$, the base is $I \times I$, and the Morse $2$--functions are of the form $G(t,p) = (t,g_t(p))$. In preparation for general Morse $2$--functions over $I \times I$ (see Section~5), we will need versions of Theorem~\ref{T:1Existence} and~\ref{T:1Uniqueness} in which attaching maps for $n$--dimensional handles lie in level sets of $I$--valued Morse functions on $(n-1)$--manifolds.

Throughout this section, we are given the following data:
\begin{enumerate}
 \item A connected $m$--dimensional cobordism $M$ from $F_0 \neq \emptyset$ to $F_1 \neq \emptyset$, where $F_0$ and $F_1$ are compact $(m-1)$--manifolds, possibly with boundary. We also assume $m \geq 2$ to avoid very-low-dimensional confusion.
 \item A collection $L_1, \ldots, L_p$ (possibly empty) of closed manifolds with $\dim(L_i)=l_i < m/2$ and with mutually disjoint embeddings $\phi_i : [-\epsilon,\epsilon] \times B^{m-1-l_i} \times L_i \hookrightarrow (M \setminus \partial M)$, for some small $\epsilon > 0$. Note that if $l_i < m/2$ then $l_i < m-1$. Assume the order is such that $l_1 \leq \ldots \leq l_p$.
 \item A collection of values $z_1 < \ldots < z_p \in (0,1)$.
\end{enumerate}
In the results that follow we will say that a Morse function $g : M \to I$ {\em is standard with respect to $\phi_i$ at height $z_i$} if $g \circ \phi_i : [-\epsilon,\epsilon] \times B^{m-1-l_i} \times L_i \to I$ is of the form $(z,x,p) \mapsto z + z_i$ on some neighborhood of $\{0\} \times \{0\} \times L_i$. 

The most illuminating example to bear in mind is when $m=3$ and each $l_i=1$, and we think of $L_1 \cup \ldots \cup L_p$ as a link in $M$ and of each embedding $\phi_i$ as given by a framing of $L_i$. Then we are interested in Morse functions with $L_i$ in the level set at level $z_i$, with framing coming from a framing in this level set. These can be constructed by starting with a given Morse function, projecting $L_i$ into a level set, and then resolving crossings by stabilizing to increase the genus. More generally the $\phi_i$'s are going to be attaching maps for handles, and we will see in the next section the importance of attaching handles along spheres lying in level sets of a Morse function, with tubular neighborhoods interacting well with the Morse function. 

\begin{remark}
 In the proofs of the following theorems, we will use generic gradient-like vector fields, and in particular the ascending and descending manifolds of critical points, as a tool to organize local modifications of Morse functions, such as cancellation of critical points. Recall that a gradient-like vector field is a vector field which is transverse to level sets and such that, near each critical point, there are local coordinates with respect to which the Morse function takes the usual form $-x_1^2 - \ldots -x_k^2 + x_{k+1}^2 + \ldots + x_m^2 = \mu_k^m(\mathbf{x})$ and the vector field is the usual Euclidean gradient of this function. For a fixed Morse function $g$, a generic gradient-like vector field is one for which the ascending and descending manifolds meet transversely in intermediate level sets. For a generic homotopy $g_t$ between Morse functions, a generic $1$--parameter family of gradient-like vector fields $V_t$ is one for which the $1$--parameter families of ascending and descending manifolds 
intersected with intermediate level sets are transverse in the $1$--parameter sense, and for a generic homotopy $g_{s,t}$ between generic homotopies we have the natural notion of a generic $2$--parameter family of gradient-like vector fields $V_{s,t}$. (We should also require that, at the non-Morse singularities, the vector field is the usual Euclidean gradient for the standard model of the singularity in local coordinates.) It is clear from the fact that the  transversality properties of the ascending and descending manifolds are generic that the associated ``genericity'' properties of the vector fields are actually generic.
\end{remark}

\begin{theorem} \label{T:MorseExistence}
 There exists an indefinite ordered Morse function $g : M \to I$, with critical values not in $\{z_1, \ldots, z_p\}$, which is standard with respect to each of $\phi_1, \ldots, \phi_p$, at heights $z_1, \ldots, z_p$ respectively. Furthermore, the indices of the critical values and the dimensions $l_i$ of the submanifolds $L_i$ are such that all critical values of index $\leq l_i$ are below $z_i$ while all critical values of index $> l_i$ are above $z_i$.
\end{theorem}

In the following proof, we will make essential use of the standard lemma that, for a Morse function $g$ with a generic gradient-like vector field, if there is a single gradient flow line from a critical point $q$ of index $k+1$ down to a critical point $p$ of index $k$, then the two critical points can be cancelled.  More precisely, there exists a generic homotopy $g_t$, with $g_0 = g$, with exactly one death singularity at $g_{1/2}$ involving $p$ and $q$, and no other birth or death singularities. Also note that no other critical values need to move if there is a regular level set between $p$ and $q$ such that the descending manifold for $q$ and the ascending manifold for $p$ avoid all other critical points on their way to this level set. If this is not the case then other critical points may need to move ``out of the way'' to facilitate the crossing. See Figure~\ref{F:CriticalCancellation} for an illustration of these ideas. 
\begin{figure}[ht!]
\labellist
\small\hair 2pt
\pinlabel $q$ [b] at 20 97 
\pinlabel $q$ [b] at 175 97 
\pinlabel $p$ [t] at 39 74
\pinlabel $p$ [t] at 194 74
\pinlabel $a$ [t] at 164 91
\pinlabel $b$ [b] at 201 83
\pinlabel $q$ [r] at 8 40
\pinlabel $p$ [r] at 8 15
\pinlabel $q$ [r] at 161 40
\pinlabel $p$ [r] at 161 15
\pinlabel $a$ [r] at 161 33
\pinlabel $b$ [r] at 161 26
\pinlabel $a$ [t] at 235 68
\pinlabel $b$ [b] at 272 83
\endlabellist
\begin{center}
\includegraphics{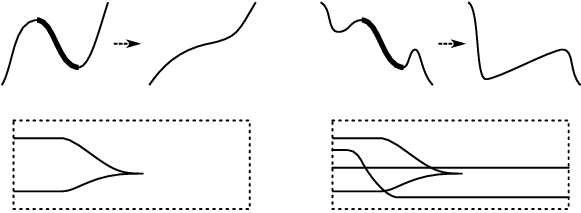}
\caption{Cancellation of critical points $q$ and $p$ in one dimension, with accompanying Cerf graphics. In the first example, no other critical points need to move, but in the second example the critical point $a$ needs to drop below $p$ before $q$ and $p$ can cancel (or raise $b$ above $q$).}
\label{F:CriticalCancellation}
\end{center}
\end{figure}
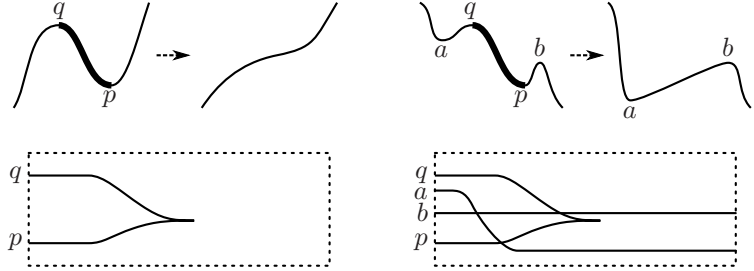

\begin{definition}
 Given a Morse function $g : M \to I$ with a gradient-like vector field and critical points $q$ of index $k+1$ and $p$ of index $k$, with $g(q) > g(p)$, we say that $q$ {\em cancels} $p$  if there is a unique gradient flow line from $q$ to $p$. 
\end{definition}

Then we can summarize the standard cancellation lemma (without proof) as follows:

\begin{lemma} \label{L:MorseCancellation}
 Given $g : M \to I$ with a generic gradient-like vector field, if critical point $q$ cancels critical point $p$, let $C$ be the 
closure of the descending manifold for $q$ inside $g^{-1}[g(p),g(q)]$.
Note that $p \in C$. Then there is a generic homotopy $g_t$ between Morse functions, with $g_0 = g$, which is independent of $t$ outside an arbitrarily small neighborhood of $C$, passes through exactly one death singularity at $g_{1/2}$ in which $q$ and $p$ cancel, and has no other birth or death singularities. If 
$C$ is actually just the descending disk, then no other critical values will change; otherwise all the critical points in $C$ besides $q$ and $p$ will have to move below $p$ before the death occurs.
\end{lemma}

\begin{proof}[Proof of Theorem~\ref{T:MorseExistence}]
 First define $g : M \to I$ on the submanifolds $\phi_i([-\epsilon,\epsilon] \times B^{m-1-l_i} \times L_i)$ by $\phi_i^{-1}$ followed by projection to $[-\epsilon,\epsilon]$ followed by translation by $z_i$. This places $\phi_i(L_i)$ at height $z_i$ as desired.
 
 There is no obstruction to extending this map to a Morse function $g : M \to I$. (Issues of smoothness at the boundary of $\phi_i([-\epsilon,\epsilon] \times B^{m-1-l_i} \times L_i)$ are easily avoided either by use of tubular neighborhood theorems or by slightly shrinking the image of $\phi_i$.) To make $g$ indefinite and ordered, first choose a generic gradient-like vector field, so that we can construct arguments using gradient flow lines and ascending and descending manifolds. Indefiniteness is easily achieved because each critical point of index $0$ (or $m$) must be cancelled by a critical point of index $1$ (or $m-1$), since $M$ is connected. (If, for an index $0$ critical point $p$, there was no such cancelling index $1$ critical point, then there could not be a path from $p$ to $F_0$, and $M$ would not be connected.) The 
 set $C$ for an index $1$ critical point is $1$--dimensional, and will thus miss a neighborhood of the $\phi_i(L_i)$'s by genericity of the gradient-like vector field. Here we use the fact that $\dim(L_i)=l_i < m/2 \leq m-1$, so that each $\phi_i(L_i)$ has positive codimension in the level set $g^{-1}(z_i)$. Thus we can cancel all the index $0$ and $m$ critical points without modifying $g$ near the $\phi_i(L_i)$'s.
 
 To arrange that $g$ is ordered, suppose that $p$ and $q$ are critical points of index $j$ and $k$, respectively, with $g(p) < g(q)$, and with no critical values in $(g(p),g(q))$. If the descending manifold $D_q$ for $q$ and the ascending manifold $A_p$ for $p$ are disjoint then there is a generic homotopy supported in a neighborhood of $D_q \cap g^{-1}[g(p),g(q)]$ which lowers $g(q)$ below $g(p)$ (without creating any new critical points). In this case there is also a generic homotopy supported in a neighborhood of $A_p \cap g^{-1}[g(p),g(q)]$ which raises $g(p)$ above $g(q)$. If $j \geq k$ this disjointness can always be arranged by a generic choice of gradient-like vector field, and thus we can get $g$ to be ordered, using either raising or lowering homotopies each time we need to switch the relative order of two critical points. However, if we are not careful, we may mess up the behavior of $g$ near $\phi_i(L_i)$. To avoid this, we need to make sure that either $D_q$ or $A_p$ misses each $\phi_i(L_i)$, 
which can be done by generically choosing the gradient-like vector field, as long as $\dim(L_i)=l_i < m/2$. (To see this, count dimensions in the level set $F$ containing $\phi_i(L_i)$ and note that we are asking for either $l_i < j$ or $l_i+k <m$; if $l_i \geq j \geq k$ then $l_i + k \leq 2l_i < m$.)

 To arrange that the critical values are ordered nicely with respect to the values $z_i$ as stated, we may need to further raise or lower some critical values, but the same dimension count argument works in this case.
\end{proof}

The following theorem is about ordered, indefinite generic homotopies on $M$; if $F_0$ and $F_1$ are connected, then ordered and indefinite implies fiber-connected. At the end of this section, in Lemma~\ref{L:Connected2Ordered}, we will discuss fiber-connectedness without the ordered assumption.

\begin{theorem} \label{T:MorseDeformationExistence}
 Given two indefinite, ordered Morse functions $g_0, g_1 : M \to I$ there exists an indefinite, ordered generic homotopy $g_t : M \to I$ from $g_0$ to $g_1$. If $m \geq 3$ and both $g_0$ and $g_1$ are standard with respect to each $\phi_i$ at height $z_i$, then we can arrange that, for all $t$, $g_t$ is standard with respect to each $\phi_i$ at height $z_i$. We can also arrange that all the births occur before the critical point crossings and that all the deaths occur after the critical point crossings.
\end{theorem}

The proof of this theorem (and the one to follow about homotopies of homotopies) is in essentially the same spirit as the proof of the preceding theorem; we just need the right cancellation lemmas to get rid of definite critical points over time and we need to count dimensions to see that we can order critical points appropriately and avoid the submanifolds $\phi_i(L_i)$ as we modify the homotopy.

For the cancellation lemmas, we need to articulate conditions under which we can pass through eye, unmerge or swallowtail singularities to simplify the homotopy. When discussing a generic homotopy $g_t : M \to I$, we will frequently use the Cerf graphic to organize our argument; recall that this is the image in $I \times I$ of the critical points of $g_t$ under the map $G : (t,p) \mapsto (t,g_t(p))$. We will use the term ``$k$--fold'' to refer to an arc of index $k$ critical points in $I \times M$. If we label a $k$--fold $P$, then for a fixed time $t$, $P_t$ will refer to the index $k$ critical point on $P$ at time $t$. We will also fix a generic $1$--parameter family of gradient-like vector fields so that we may refer to gradient flow lines for each $g_t$. (Here ``generic'' means that the $1$--parameter families of descending manifolds intersect transversely in level sets, which means that handle slides occur at isolated times, with lower index critical points never sliding over higher index critical 
points.) We will say that a $(k+1)$--fold $Q$ ``cancels'' a $k$--fold $P$ over a time interval $A$ if $Q_t$ cancels $P_t$ for every time $t \in A$. 

In the following three lemmas, suppose that $g_t : M \to I$ is a generic homotopy between Morse functions $g_0$ and $g_1$ with a generic $1$--parameter family of gradient-like vector fields. We state conditions under which we can modify $g_t$ so as to reverse the arrows in Figures~\ref{F:Eye},~\ref{F:Merge} and~\ref{F:Swallowtail}. We leave the proofs to the reader; the basic idea, as with the proof of the standard cancellation lemma (Lemma~\ref{L:MorseCancellation}), is as follows: We first push extraneous critical points down far enough so that we only need to work with a single descending disk and ascending disk. Then we use the uniqueness of a gradient flow line to find coordinates on a neighborhood of the gradient flow line in which the gradient flow line itself lies in one coordinate axis, and in the other coordinates the Morse function has the usual Morse model. Then the cancellation occurs entirely in the coordinate chart containing the gradient flow line.

\begin{lemma}[Unmerge] \label{L:UnmergeCancellation}
 Consider a $k$--fold $P$ and  a $(k+1)$--fold $Q$ over a time interval $[t_0,t_1]$ such that $Q$ cancels $P$ over all of $[t_0,t_1]$. Then for some arbitrarily small $\delta >0$ there is a generic homotopy between homotopies $g_{s,t}$, with $g_{0,t} = g_t$, passing through a single unmerge singularity (see Figure~\ref{F:Merge}) at $s=1/2$ and no other $1$--parameter singularities, independent of $s$ for $t \in [0,t_0] \cup [t_1,1]$, such that, with respect to $g_{1,t}$, the cancelling pair $Q_t$ and $P_t$ die at $t=t_0+\delta$ and are reborn at $t=t_1-\delta$. For each $t \in [t_0,t_1]$, $g_{s,t}$ is independent of $s$ outside an arbitrarily small neighborhood of the descending manifold for $Q_t$. Also note that, with respect to $g_{1,t}$, $Q_t$ still cancels $P_t$ on $[t_0,t_0+\delta)$ and on $(t_1-\delta,t_1]$. Furthermore we can arrange that any other folds that cancelled $P$ on $[t_0,t_0+\delta]$ or $[t_1-\delta,t_1]$ will still cancel $P$ there.
\end{lemma}

\begin{lemma}[Eye death] \label{L:EyeCancellation}
 Consider a $k$--fold $P$ and  a $(k+1)$--fold $Q$ over a time interval $[t_0,t_1]$ such that $Q$ cancels $P$ over all of $(t_0,t_1)$. Also suppose that the critical points $Q_t$ and $P_t$ are born as a cancelling pair at time $t_0$ and die as a cancelling pair at time $t_1$. Then for some small $\delta > 0$ there is a generic homotopy between homotopies $g_{s,t}$, with $g_{0,t} = g_t$, passing through a single eye death singularity (see Figure~\ref{F:Eye}) at $s=1/2$ and no other $2$--parameter singularities, independent of $s$ for $t \in [0,t_0-\delta] \cup [t_1+\delta,1]$, such that, in $g_{1,t}$, the cancelling pair $Q_t$ and $P_t$ have cancelled for all $t \in [t_0,t_1]$. For each $t \in (t_0,t_1)$, $g_{s,t}$ is independent of $s$ outside an arbitrarily small neighborhood of the descending manifold for $Q_t$. For $t \in \{t_0,t_1\}$, $g_{s,t}$ is independent of $s$ outside a neighborhood of the birth/death point.
\end{lemma}

\begin{lemma}[Swallowtail death] \label{L:SwallowtailCancellation}
 Consider a $k$--fold $P$ and two $(k\pm 1)$--folds $Q$ and $R$ over a time interval $(t_0,t_1)$. Suppose furthermore that $Q$ and $P$ are born as a cancelling pair at time $t_0$, that $R$ and $P$ die as a cancelling pair at time $t_1$, and that $Q$ cancels $P$ over $(t_0,(t_0+t_1)/2+\delta)$ while $R$ cancels $P$ over $((t_0+t_1)/2-\delta,t_1)$ for some small $\delta >0$. Then there is a generic homotopy between homotopies, $g_{s,t}$, with $g_{0,t} = g_t$, passing through a single swallowtail death singularity (see Figure~\ref{F:Swallowtail}) at $s=1/2$ and no other $2$--parameter singularities, independent of $s$ for $t \in [0,t_0-\delta] \cup [t_1+\delta,1]$ (again for some small $\delta >0$) such that, in $g_{1,t}$, the $k$--fold $P$ has disappeared and the $(k\pm 1)$--folds $Q$ and $R$ have become the same fold. For each $t \in (t_0,(t_0+t_1)/2-\delta)$, $g_{s,t}$ is independent of $s$ outside an arbitrarily small neighborhood of the descending/ascending (according to whether $\pm=+$ or $\pm =-$) manifold for $Q_t$, while for each $t \in ((t_0+t_1)/2 + \delta,t_1)$ the independence is outside a neighborhood of the descending/ascending manifold for $R_t$, and for $t \in [(t_0+t_1)/2 - \delta,(t_0+t_1)/2+\delta]$ we need a neighborhood of the union of the descending/ascending manifolds for $Q_t$ and $R_t$. For $t \in \{t_0,t_1\}$ the independence is outside a neighborhood of the birth/death points.
\end{lemma}

\begin{proof}[Proof of Theorem~\ref{T:MorseDeformationExistence}]

 As in the proof of Theorem~\ref{T:MorseExistence}, there is no difficulty in finding a generic homotopy $g_t : M \to I$, and if $g_0$ and $g_1$ are standard with respect to each $\phi_i$ at height $z_i$ then we can make $g_t$ independent of $t$ on these neighborhoods. Arranging for births to happen first and deaths last is standard, using connectedness; $g_t$ is modified via a generic homotopy between homotopies which passes through cusp-fold crossings, moving left-cusps (births) further left (earlier), and right-cusps (deaths) further right (later). However it requires some work to make $g_t$ indefinite and then ordered for all $t \in (0,1)$.

 We will cancel the $0$--folds one at a time. Consider a $0$--fold $P$ which is born at time $a$ and dies at time $b$. At each time $t \in [a,b]$, there is some index $1$ critical point $q$ which cancels $P_t$. Thus there is a sequence $a=t_0 < t_1 < \ldots < t_n=b$ and some $\delta>0$, giving a covering of $[a,b]$ by intervals $I_1=[a,t_1+\delta), I_2= (t_1-\delta,t_2+\delta), \ldots , I_n=(t_{n-1}-\delta,b]$, and a sequence of $1$--folds $Q^1, \ldots, Q^n$ such that each $Q^i$ cancels $P$ over the interval $I_i$. We do this so that $Q^1$ is the $1$--fold born with $P$ as a cancelling $0$--$1$ pair at time $t_0=a$ and so that $Q^n$ is the $1$--fold which dies with $P$ as a cancelling pair at time $t_n=b$. Using the above three lemmas we can then cancel $P$ with the $Q^i$'s: First cancel over the non-overlapping parts of the open intervals $I_2, \ldots, I_{n-1}$ using Lemma~\ref{L:UnmergeCancellation}. 
 Then cancel over the overlaps and $I_1$ and $I_n$ using either Lemma~\ref{L:EyeCancellation} or Lemma~\ref{L:SwallowtailCancellation}, depending on whether the two cancelling $1$--folds $Q^i$ and $Q^{i+1}$ in the overlap region $(t_i-\delta,t_i+\delta)$ are the same or different. Going from cancelling on the nonoverlapping regions to the overlapping regions requires the extra clauses in Lemma~\ref{L:UnmergeCancellation} to the effect that whatever cancelled $P$ at the beginning still cancels $P$ wherever it has not been killed with Lemma~\ref{L:UnmergeCancellation}.
 
 The above argument ignored the issue of the submanifolds $\phi_i(L_i)$. If $m \geq 3$ and both $g_0$ and $g_1$ are standard with respect to each $\phi_i$ at height $z_i$, then we want to arrange that, for all $t$, $g_t$ is standard with respect to each $\phi_i$ at height $z_i$. Since, before cancelling the definite folds, we had arranged for this property to hold, we need to arrange that, in each application of Lemmas~\ref{L:UnmergeCancellation},~\ref{L:EyeCancellation} and~\ref{L:SwallowtailCancellation}, we avoid neighborhoods of each $\phi_i(L_i)$. In other words, all the descending manifolds for the cancelling $1$--folds down to the level of the cancelled $0$--folds should avoid $\phi_i(L_i)$. Counting dimensions we see that this can generically be achieved at all but finitely many times $t$, which are distinct from the times $t_0, t_1, \ldots, t_n$ at which we switch from one cancelling $1$--fold to another. If the descending manifold for a $1$--fold $Q^j$ intersects some $\phi_i(L_i)$ at time $t_*$, 
with $t_{j-1} < t_* < t_j$, we break $Q^j$ into two $1$--folds by introducing a $1$--$2$ swallowtail (passing through a swallowtail singularity) at time $t_*$ along $Q^j$. This is illustrated in Figure~\ref{F:SplitQj}; we label the two new $1$--folds $Q^j_-$ and $Q^j_+$ as indicated in the figure, and observe that we can arrange for the descending manifold for $Q^j_-$ to meet $\phi_i(L_i)$ at some time $t_- > t_*$ while the descending manifold for $Q^j_+$ meets $\phi_i(L_i)$ at some time $t_+ < t_*$. Then we break the interval $(t_{j-1}-\delta,t_j+\delta)$ into two overlapping intervals $(t_{j-1}-\delta,t_*+\delta)$ and $(t_*-\delta,t_j+\delta)$ and replace the single cancelling $1$--fold $Q^j$ with $Q^j_-$ over $(t_{j-1}-\delta,t_*+\delta)$ and $Q^j_+$ over $(t_*-\delta,t_j+\delta)$. (We might, of course, need to decrease $\delta$.)
 \begin{figure}[ht!]
\labellist
\small\hair 2pt
\pinlabel $Q_j$ [b] at 52 47 
\pinlabel $t_{j-1}$ [t] at 37 15
\pinlabel $t_{*}$ [t] at 85 15
\pinlabel $t_{j}$ [t] at 133 15
\pinlabel $P$ [bl] at 155 16
\pinlabel $Q_j^-$ [b] at 222 47 
\pinlabel $Q_j^+$ [b] at 285 47 
\pinlabel $t_{j-1}$ [t] at 206 15
\pinlabel $t_+$ [t] at 246 15
\pinlabel $t_{*}$ [t] at 256 15
\pinlabel $t_-$ [tl] at 263 15
\pinlabel $t_{j}$ [t] at 303 15
\pinlabel $P$ [bl] at 325 16
\endlabellist
\begin{center}
\includegraphics{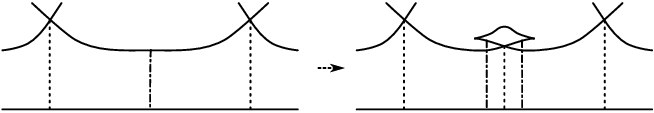}
\caption{Splitting $Q^j$.}
\label{F:SplitQj}
\end{center}
\end{figure}
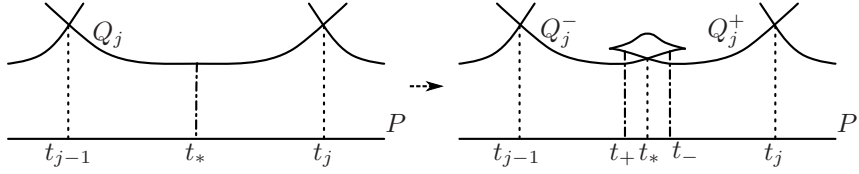

 Note that the above argument required $m \geq 3$ because otherwise the $2$--fold in the $1$--$2$ swallowtail is a definite fold.
 
 \label{P:OrderingHomotopy}
 To arrange that $g_t$ is ordered when $g_0$ and $g_1$ are ordered, we first need to arrange that each birth or death of a cancelling $k$--$(k+1)$ pair occurs above all the other $k$--folds and below all the other $(k+1)$--folds. This is straightforward because such a modification of $g_t$ can be achieved through a generic homotopy supported in a neighborhood of an arc, which can be chosen to be disjoint from the $\phi_i(L_i)$'s. Now the only issue is pulling $k$--folds below $j$--folds when $k < j$ (or pushing $j$--folds above $k$--folds); this can be achieved if we ignore the $\phi_i(L_i)$'s for the same reason that it can be achieved for a fixed Morse function, as in Theorem~\ref{T:MorseExistence}, namely that, for a generic $1$--parameter family of gradient-like vector fields, $k$--handles will not slide over $j$--handles if $k < j$. However, if we want to avoid modifying $g_t$ near each $\phi_i(L_i)$ we need to be more careful, and to do this we count dimensions again.

 Here we need to check that either the $1$--parameter descending disk for the $k$--fold $Q$ or the $1$--parameter ascending disk for the $j$--fold $P$ misses the $l_i$--dimensional submanifold $\phi_i(L_i)$ in the level set $g^{-1}(z_i)$, which is presumed to be between $Q$ and $P$. The level set is $(m-1)$--dimensional, the descending sphere for $Q$ in the level set is $(k-1)$--dimensional, and the ascending  sphere for $P$ is $(m-j-1)$--dimensional. However, because of the parameter $t$, we now want that either $(k-1)+1+l_i < m-1$ or that $(m-j-1)+1+l_i < m-1$, i.e. that $k+l_i < m-1$ or that $l_i < j-1$. If $l_i \geq j-1$, so that $k < j \leq l_i+1$, we have $k \leq l_i$ and thus $k+l_i \leq 2l_i$. Thus we are fine as long as $l_i < (m-1)/2$, but in our initial hypotheses we only assumed that $l_i < m/2$. The only potentially bad case is when $m$ is odd, $k = l_i = (m-1)/2$ and $j = (m+1)/2$. In this case both the ascending sphere for $P$ and the descending sphere for $Q$ will intersect $\phi_i(L_i)$ at 
discrete times. However, now we simply note that, again by genericity, these times will be distinct for $P$ and $Q$, and so at times when the ascending sphere for $P$ intersects $\phi_i(L_i)$ we lower $Q$ below $P$ while at times when the descending sphere for $Q$ intersects $\phi_i(L_i)$ we raise $P$ above $Q$.

\end{proof}

\begin{theorem} \label{T:MorseDefOfDefsExistence}
 Suppose that $m \geq 3$. Given two indefinite, ordered generic homotopies $g_{0,t}, g_{1,t} : M \to I$ between indefinite Morse functions $g_{0,0} = g_{1,0}$ and $g_{0,1} = g_{1,1}$, there exists an indefinite, almost ordered generic homotopy of homotopies $g_{s,t} : M \to I$ from $g_{0,t}$ to $g_{1,t}$ with fixed endpoints. When $m \geq 4$ and $F_0$ and $F_1$ are both connected this guarantees that all level sets of each $g_{s,t}$ are connected. In the case where $m = 3$ and $F_0$ and $F_1$ are both connected, we can do a little extra work to arrange that all level sets of each $g_{s,t}$ are connected, even though the ``almost ordered'' condition is not sufficient to imply this.
\end{theorem}

(Compare Proposition~3.6 in~\cite{HatcherWagoner}, which deals with approximately the same issue, but is about cancelling critical points of arbitrary indices and requires high ambient dimensions. There are many striking similarities between that proof and our proof of Theorem~\ref{T:MorseDefOfDefsExistence}.)

Note that we could also ask that $g_{s,t}$ behave well on neighborhoods
of the $L_i$'s as in the preceding theorems, and presumably there are constraints in terms of the dimensions involved, but we have no need for such a
result in this paper.

Although the statement of the theorem does not say this, we will actually be able to modify a given generic homotopy $g_{s,t}$ from $g_{0,t}$ to $g_{1,t}$ through a generic homotopy $g_{r,s,t}$ with yet one more parameter $r \in [0,1]$, with $g_{0,s,t} = g_{s,t}$, so that $g_{1,s,t}$ satisfies the various conditions we want. In doing so we will pass through higher codimension singularities, and so we could have ``cancellation'' lemmas analogous to Lemmas~\ref{L:MorseCancellation},~\ref{L:EyeCancellation},~\ref{L:UnmergeCancellation} and~\ref{L:SwallowtailCancellation}. In our case they would involve the ``butterfly singularity'' and the ``monkey saddle'' (or ``elliptic umbilic'').  However, since each only occurs once in the proof, we just develop them in the course of the proof. After discovering the necessary homotopies corresponding to these singularities, we then realized that they also play a central role, for similar reasons, in the work of Hatcher and Wagoner~\cite{HatcherWagoner}.

\begin{proof}
 There is always a generic homotopy $g_{s,t} : M \to I$ rel. boundaries, so the first issue is to make it indefinite, that is, to remove all $0$--folds ($m$--folds are treated the same way using $1-g_{s,t}$).
 
 Instead of the traditional Cerf graphic, we consider a $1$--parameter family of Cerf graphics, and in this case the $0$--folds form a $2$--dimensional immersed surface $\Sigma$, as in the example in Figure~\ref{F:Ex2DFolds}.
 \begin{figure}[ht!]
\labellist
\small\hair 2pt
\pinlabel $s$ [tr] at 5 9
\pinlabel $t$ [l] at 165 43
\pinlabel $z=g_{s,t}$ [b] at 56 110
\pinlabel $\Sigma$ at 42 56
\pinlabel $\Sigma$ at 82 77
\pinlabel $*$ at 114 48
\endlabellist
\begin{center}
\includegraphics{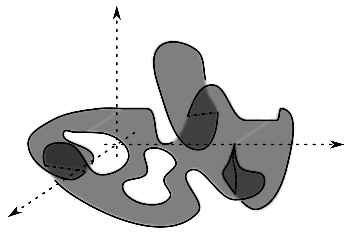}
\caption{Example of a $2$--dimensional surface $\Sigma$ of $0$--folds in a generic $2$--parameter homotopy between Morse functions. Note the swallowtail singularity at the point labelled $*$.}
\label{F:Ex2DFolds}
\end{center}
\end{figure}
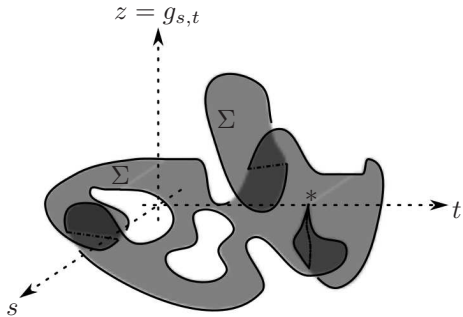
In $\Sigma$, with respect to the $s$ direction, there are merges, unmerges, eyes and swallowtails; apart from the swallowtails, these appear as smooth curves with tangents parallel to the $(t,z)$ plane.

The first step is to cut $\Sigma$ into pieces, each of which is embedded in $I \times I \times I$. We do this by first cutting $\Sigma$ in the $(t,z)$ direction at many fixed $s$ values. Such a cut is done by first making $g_{s,t}$ independent of $s$ in a small $s$--interval $[s_*-\delta, s_*+\delta]$, then applying the technique in the proof of Theorem~\ref{T:MorseDeformationExistence} to modify the homotopy $g_{s_*,t}$ to get rid of definite folds at $s_*$, and then noting that the modification is through a generic homotopy between homotopies, which can then be run forward and backward in the $s$ direction as $s$ ranges from $s_*-\delta$ to $s_*$ to $s_*+\delta$. A typical example of the result is illustrated in Figure~\ref{F:Ex2DFoldsCut}.
 \begin{figure}[ht!]
\labellist
\small\hair 2pt
\pinlabel $s$ [tr] at 5 9
\pinlabel $t$ [l] at 165 43
\pinlabel $z=g_{s,t}$ [b] at 56 110
\endlabellist
\begin{center}
\includegraphics{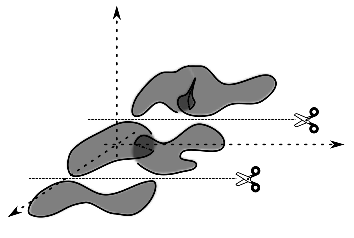}
\caption{After the first cuts, along constant $s$ slices. Note that components are still not embedded, due to the bad swallowtail.}
\label{F:Ex2DFoldsCut}
\end{center}
\end{figure}

The components of $\Sigma$ are not yet embedded because of the possibility that there are births of swallowtails in the middle of $\Sigma$.
However, in~\cite{HatcherWagoner}, page~199 (see also item~1 on page~194), it is shown exactly how such a swallowtail birth can be extended past the $0$--$1$ cusp and onto the surface of $1$--folds by passing through a ``butterfly singularity'', with the result that the swallowtail cuts $\Sigma$ into two parts, which intersect each other but no longer have a self-intersection. The change in the movie of Cerf graphics is shown in Figure~\ref{F:CutBadSwallowtail}, with accompanying graphs of the $1$--dimensional Morse functions. 
\begin{figure}[ht!]
\labellist
\small\hair 2pt
\pinlabel $0$ [t] at 50 227 
\pinlabel $1$ [b] at 53 241
\endlabellist
\begin{center}
\includegraphics{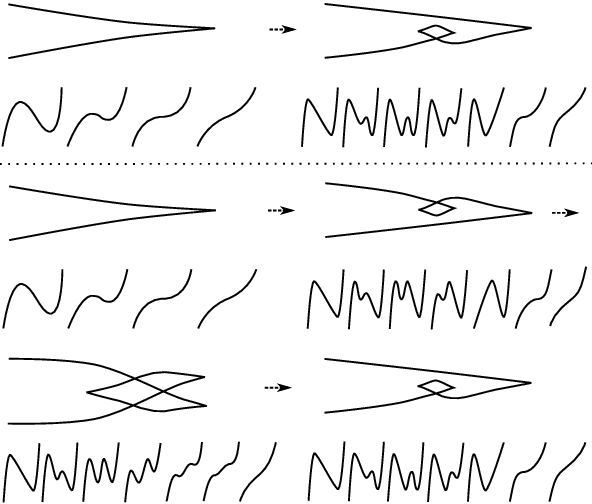}
\caption{Cutting the bad swallowtail onto the surface of $1$--folds using the butterfly singularity. Above the dotted line, the swallowtail occurs in the middle of a $0$--fold, leading to surface of $0$--folds that is not embedded. Below the line, the swallowtail is born in the middle of a $1$--fold, so that the two $0$--folds which intersect are in distinct components of the surface of $0$--folds.}
\label{F:CutBadSwallowtail}
\end{center}
\end{figure}

Now each component $\mathcal{P}$ of $\Sigma$ is embedded in $I \times I \times I$ and we will eliminate these components one by one, in much the same way we eliminated individual $0$--folds in Theorem~\ref{T:MorseDeformationExistence} and individual index $0$ critical points in Theorem~\ref{T:MorseExistence}. Noting that, for each $(s,t)$, the index $0$ critical point $\mathcal{P}_{s,t}$ coming from $\mathcal{P}$ is cancelled by some index $1$ critical point, we can cover $\mathcal{P}$ with open disks $\{\mathcal{P}^i\}$ over each of which we have chosen a particular disk $\mathcal{Q}^i$ of cancelling index $1$ critical points. We also arrange that every vertex in the nerve of this cover (including vertices on $\partial \mathcal{P}$) has valence $3$ and that every edge of the nerve is transverse to constant $s$ slices. Then we can use the ideas in the proof of Theorem~\ref{T:MorseDeformationExistence} to eliminate $\mathcal{P}$ at all points in exactly one or two of the open sets of the cover, and we reduce 
to the case where $\mathcal{P}$ is a union of disjoint triangles each cancelled by three distinct surfaces of $1$--folds. 

Figure~\ref{F:TripleCover} shows a sequence of Cerf graphics, representing the $2$--parameter Cerf graphic where three open sets intersect; on the right we show the covering in parameter space, with the nerve and its trivalent vertex. The labels $a$, $b$ and $c$ indicate the $1$--handles that cancel in each of the three open sets. At each point in parameter space, we adopt the convention that the closest index $1$ critical point to the $0$--fold is the cancelling one. This is not necessarily the case at first, but right before the cancellations this will be true. Next, Figure~\ref{F:FirstCancellation} shows the result of cancelling the $0$--fold with the appropriate $1$--folds away from the overlaps in the cover, with the labels $ab$, $bc$ and $ac$ indicating the pairs of $1$--folds which cancel in each region. Finally, Figure~\ref{F:CerfBelowMonkey} shows how this sequence of Cerf graphics is transformed by cancelling the swallowtails that remain when two of the open sets intersect, leaving the $0$--fold 
uncancelled only in a cusped-triangular neighborhood of the trivalent vertex. (The boxed numbers indicate points in parameter space for reference in the next figure to come.)
\begin{figure}[ht!]
\labellist
\small\hair 2pt
\pinlabel $z$ [b] at 14 163 
\pinlabel $z$ [b] at 14 107
\pinlabel $z$ [b] at 14 51
\pinlabel $t$ [l] at 106 128
\pinlabel $t$ [l] at 106 72
\pinlabel $t$ [l] at 106 15
\pinlabel $t$ [l] at 219 128
\pinlabel $s$ [t] at 6 5
\pinlabel $s$ [t] at 127 28
\pinlabel $0$ [b] at 59 132
\pinlabel $c$ [l] at 99 137
\pinlabel $b$ [l] at 99 148
\pinlabel $a$ [l] at 99 159
\pinlabel $a$ [l] at 99 103
\pinlabel $b$ [l] at 99 91
\pinlabel $c$ [l] at 99 80
\pinlabel $a$ [l] at 99 46
\pinlabel $b$ [l] at 99 35
\pinlabel $c$ [l] at 99 24
\pinlabel $c$ at 197 102
\pinlabel $b$ at 168 47
\pinlabel $a$ at 140 103
\endlabellist
\begin{center}
\includegraphics{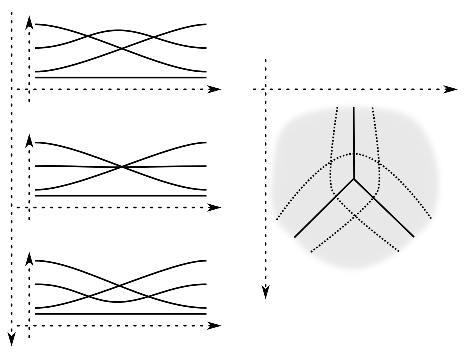}
\caption{Three $1$--folds $a$, $b$ and $c$ cancelling a $0$--fold in two parameters.}
\label{F:TripleCover}
\end{center}
\end{figure}
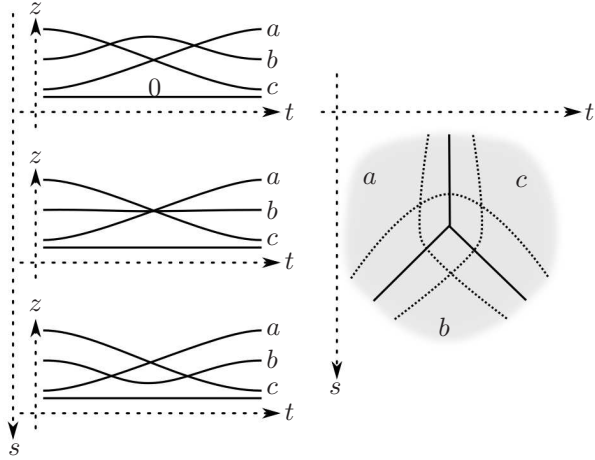
\begin{figure}[ht!]
\labellist
\small\hair 2pt
\pinlabel $0$ [t] at 60 193 
\pinlabel $a$ [l] at 95 214
\pinlabel $c$ [r] at 23 214 
\pinlabel $b$ [l] at 95 198
\pinlabel $a$ [l] at 95 157
\pinlabel $c$ [r] at 23 157 
\pinlabel $b$ [l] at 95 145
\pinlabel $a$ [l] at 95 100
\pinlabel $c$ [r] at 23 100 
\pinlabel $b$ [l] at 95 91
\pinlabel $a$ [l] at 95 43
\pinlabel $c$ [r] at 23 43 
\pinlabel $b$ [l] at 95 34
\pinlabel $b$ [r] at 23 34
\pinlabel $t$ [l] at 106 184
\pinlabel $z$ [b] at 14 218
\pinlabel $s$ [t] at 5 6
\pinlabel $s$ [t] at 127 57
\pinlabel $t$ [l] at 219 157
\pinlabel $ac$ at 176 133
\pinlabel $ab$ at 145 100
\pinlabel $bc$ at 185 92
\endlabellist
\begin{center}
\includegraphics{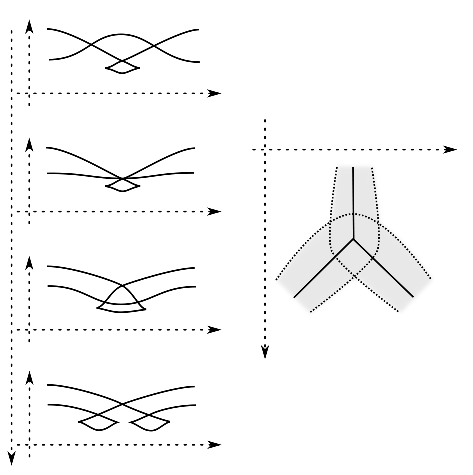}
\caption{After cancelling the $0$--$1$ pairs away from the overlapping regions.}
\label{F:FirstCancellation}
\end{center}
\end{figure}
\begin{figure}[ht!]
\labellist
\small\hair 2pt
\pinlabel {\fbox{$1$}} at 188 232
\pinlabel {\fbox{$2$}} at 206 232
\pinlabel {\fbox{$3$}} at 196 187
\pinlabel {\fbox{$4$}} at 147 174
\pinlabel {\fbox{$5$}} at 249 174
\pinlabel {\fbox{$6$}} at 177 150
\pinlabel {\fbox{$7$}} at 220 150
\pinlabel {\fbox{$8$}} at 171 112
\pinlabel {\fbox{$9$}} at 224 112
\pinlabel {\fbox{$1$}} at  48 335
\pinlabel {\fbox{$2$}} at  65 335
\pinlabel {\fbox{$3$}} at  59 277
\pinlabel {\fbox{$4$}} at  36 220
\pinlabel {\fbox{$5$}} at  80 220
\pinlabel {\fbox{$6$}} at  47 108
\pinlabel {\fbox{$7$}} at  72 108
\pinlabel {\fbox{$8$}} at  42 51 
\pinlabel {\fbox{$9$}} at  76 51 
\pinlabel $c$ [r] at 24 329
\pinlabel $a$ [l] at 95 329
\pinlabel $b$ [l] at 95 313
\pinlabel $c$ [r] at 24 272
\pinlabel $a$ [l] at 95 272
\pinlabel $b$ [l] at 95 255
\pinlabel $c$ [r] at 24 101
\pinlabel $b$ [r] at 24 91
\pinlabel $a$ [l] at 95 101
\pinlabel $b$ [l] at 95 91
\pinlabel $c$ [r] at 24 44 
\pinlabel $b$ [r] at 24 34
\pinlabel $a$ [l] at 95 44 
\pinlabel $b$ [l] at 95 34
\pinlabel $a$ [tr] at 42 29
\pinlabel $c$ [tl] at 77 29

\endlabellist
\begin{center}
\includegraphics{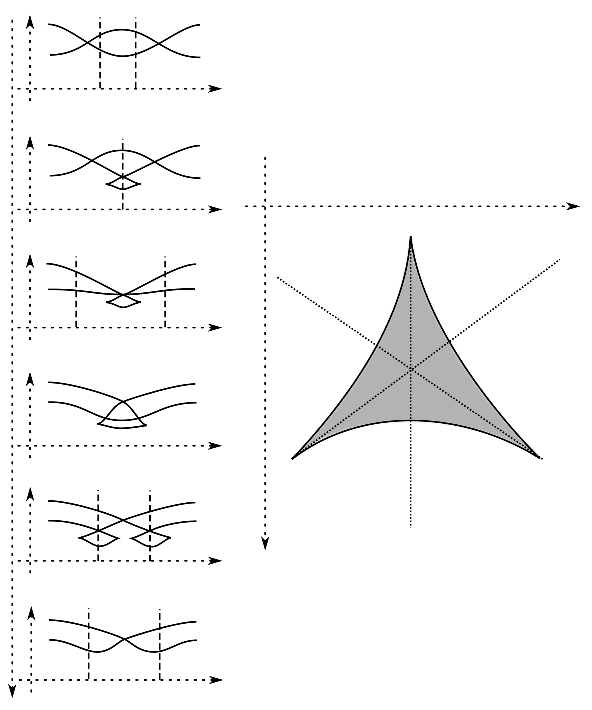}
\caption{The $2$--parameter Cerf graphic after cancelling along the double overlaps but not the triple overlaps. The dotted lines in the figure on the right indicate the points where two $1$--folds intersect. Note that the labelling of the $1$--folds $a$, $b$ and $c$ is no longer consistent because, as one moves across the bottom Cerf graphic, $c$ becomes $b$ and $b$ becomes $a$, and as one moves across the top Cerf graphic, $c$ becomes $a$.}
\label{F:CerfBelowMonkey}
\end{center}
\end{figure}

Because there are three $1$--handles cancelling the $0$--handle across this triangle, we can construct a local model in which we separate out two local coordinates in the domain in which the cancellations occur and keep the other coordinates as a sum of squares independent of the parameters. In other words, locally $g_{s,t}$ is given by $g_{s,t}(x_1,x_2,x_3, \ldots, x_m) = h_{s,t}(x_1,x_2) + x_3^2 + \ldots + x_m^2$. In Figure~\ref{F:BelowMonkeyHandles} we schematically illustrate the handle decomposition corresponding to $h_{s,t}$ at representative points in the $(s,t)$--parameter space, as labelled by boxed numbers in Figure~\ref{F:CerfBelowMonkey}. In~\cite{HatcherWagoner}, page~202, this is shown to be precisely the southern hemisphere of the $S^2$--boundary of a $B^3$ space of deformations of the monkey saddle $h(x_1,x_2) = x_1^3 - 3 x_1 x_2^2$. The south pole is visualized as pushing down in the middle of the monkey saddle to create an index $0$ critical point with three cancelling index $1$ critical 
points, while the equator is a loop of Morse functions involving two index $1$ critical points that slide over each other three times. To eliminate this $0$--fold completely, we simply replace the southern hemisphere of this $S^2$ with the northern hemisphere, which involves pushing the middle of the monkey saddle upwards to create an index $2$ critical point. We replace the triangular $0$--fold surface with a triangular $2$--fold surface, exactly as in~\cite{HatcherWagoner}, page~198, item~(iv) and Lemma~4.1. (This is where it is important that $m \geq 3$, so that index $2$ is indefinite.)
\begin{figure}[ht!]
\labellist
\small\hair 2pt
\pinlabel {\fbox{$1$}} at  91 229
\pinlabel {\fbox{$2$}} at 192 229
\pinlabel {\fbox{$3$}} [b] at 141 202
\pinlabel {\fbox{$4$}} at  41 143
\pinlabel {\fbox{$5$}} at 240 143
\pinlabel {\fbox{$6$}} [tl] at 118 116
\pinlabel {\fbox{$7$}} [tr] at 165 116
\pinlabel {\fbox{$8$}} at  91  59
\pinlabel {\fbox{$9$}} at 192  59
\pinlabel $c$ at 92 240
\pinlabel $b$ at 103 222
\pinlabel $a$ at 191 240
\pinlabel $b$ at 182 222
\pinlabel $c$ at 133 197
\pinlabel $a$ at 150 197
\pinlabel $b$ at 142 177
\pinlabel $c$ at 42 154
\pinlabel $b$ at 52 135
\pinlabel $a$ at 240 154
\pinlabel $b$ at 231 135
\pinlabel $c$ at 98 125
\pinlabel $a$ at 119 122
\pinlabel $b$ at 111 108
\pinlabel $a$ at 187 125
\pinlabel $c$ at 163 122
\pinlabel $b$ at 172 108
\pinlabel $c$ at 81 50
\pinlabel $a$ at 102 50
\pinlabel $c$ at 181 50
\pinlabel $a$ at 202 50
\endlabellist
\begin{center}
\includegraphics{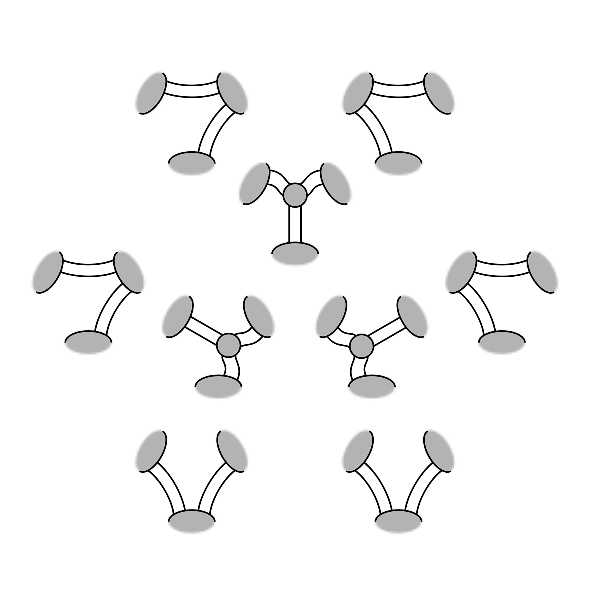}
\caption{Handle decompositions at various points in the Cerf graphic in Figure~\ref{F:CerfBelowMonkey}.}
\label{F:BelowMonkeyHandles}
\end{center}
\end{figure}

Do this to each component of $\Sigma$ and, upside down, do the same to the index $m$ critical points and we have an indefinite generic homotopy between homotopies.

\label{P:AlmostOrdering}
Getting $g_{s,t}$ to be almost ordered when $g_{0,t}$ and $g_{1,t}$ are ordered follows the same argument as for a fixed Morse function or a homotopy between Morse functions, except that we now note that for a generic $2$--parameter family of gradient-like vector fields the descending manifold for an index $j$ critical point may meet the ascending manifold for an index $j+1$ critical point at isolated points $(s,t)$ in the $2$--dimensional parameter space. This is why we can at most ask that, if an index $k$ critical point is above an index $j$ critical point, then $j \leq k+1$. As noted earlier, when $F_0$ and $F_1$ are connected and $m \geq 3$ then ordered implies connected level sets. In fact, for connectedness of level sets we simply need that all index $m-1$ critical points are above all index $1$ critical points, and thus almost ordered implies connected level sets when $m \geq 4$. So, if $m \geq 4$, the proof is complete.

For the remainder of the proof, we assume $m=3$, in which case level sets are $2$--dimensional. 

There are no definite folds, so the only way for a level set to become disconnected is by adding a $2$--handle $H_2$ to a separating circle $C$ in a lower level set. Then this level set would remain disconnected all the way to the top of the Morse function unless a higher $1$--handle $H_1$ is attached to the different components of the level set. In the $0$-- and $1$--parameter cases, we can always arrange that the $1$--handle is added below the $2$--handle in which case the attaching circle $C$ of the $2$--handle  does not separate. But in the $2$--parameter case, the attaching $0$--sphere of the $1$--handle $H_1$ may go over the $2$--handle $H_2$ and not be able to be pulled off. This is illustrated in Figure~\ref{F:2ParameterSlide}, where one foot of the $1$--handle moves around $C_1$ and over the $2$--handle $H_2$, as the parameter runs over a $2$--disk $D$. At the center of $D$, the foot is stuck at the critical point of $H_2$.
 \begin{figure}[ht!]
\labellist
\small\hair 2pt
\pinlabel $C$ [r] at 87 21
\pinlabel $C_1$ [l] at 109 18
\pinlabel {fiber} [l] at 127 7 
\endlabellist
\begin{center}
\includegraphics{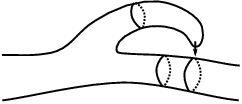}
\caption{A $1$--handle sliding over a $2$--handle in a $2$--parameter family.}
\label{F:2ParameterSlide}
\end{center}
\end{figure}

However, it is important to note that this problem occurs at isolated points, so we can pull the $1$--handles below the $2$--handles everywhere except at small disks exactly like this disk $D$.

To deal with the disconnectedness of the level set over the middle of $D$, above the $2$--handle $H_2$ and below the $1$--handle $H_1$, we need a ``helper'' $1$--handle $H_1'$, so create a $1$--$2$ cancelling pair over the entire disk $D$ in which the helper $1$--handle $H_1'$ is attached to a parallel copy $C'$ of $C$ and over $H_2$, and $H_2'$ is a ``tunnel'' between $H_1$ and $H_1'$. See Figure~\ref{F:2ParameterSlideHelper}. Now perturb $H'_1$ so that the value of the parameter in $D$ at which $H'_1$ hits the critical point of $H_2$ is different from $0 \in D$, say $d \in D$. Now the fibers are connected except for disjoint disk neighborhoods of $0$ and $d$ of radius less than $||d||/3$. However, in the disk around $0 \in D$, the helper $1$--handle can be pulled below $H_2$, so that $H_2$ does not disconnect and then $H_1$ is attached, and lastly $H_2'$ which obviously does not disconnect. By symmetry, the level sets can also be made connected in the disk around $d \in D$.
\begin{figure}[ht!]
\labellist
\tiny\hair 2pt
\pinlabel $C$ [r] at 74 17
\pinlabel $C'$ [l] at 89 15
\pinlabel $H_1$ at 68 33
\pinlabel $H_1'$ at 74 46
\pinlabel $H_2'$ at 83 34
\endlabellist
\begin{center}
\includegraphics[scale=2]{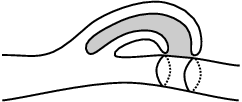}
\caption{The helper $1$--handle $H_1'$, with cancelling $2$--handle $H_2'$.}
\label{F:2ParameterSlideHelper}
\end{center}
\end{figure}
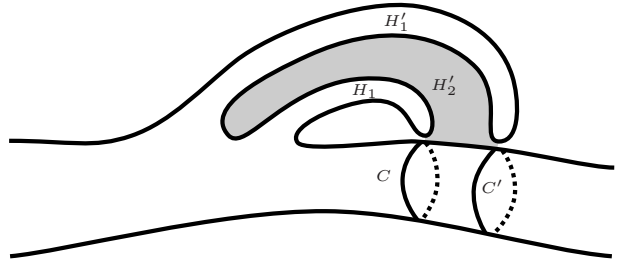

\end{proof}

We end this section with a lemma, of a slightly different flavor, that will be used several times in the following sections to relate connectedness of level sets to the property of being ordered. In particular, Theorem~\ref{T:MorseDeformationExistence} requires, as a hypothesis, that the given Morse functions be ordered. The following lemma is used to prepare for this theorem. We have already stated that ordered $I$--valued Morse functions have connected level sets, and the converse is obviously not true. However, we do have:

\begin{lemma} \label{L:Connected2Ordered}
 Let $g : M \to I$ be an indefinite Morse function which is standard with respect to each $\phi_i$ at height $z_i$, and suppose that all level sets of $g$ are connected. Then there exists an indefinite generic homotopy $g_t : M \to I$ between Morse functions, with $g_0 = g$, such that $g_1$ is ordered, all level sets of $g_t$ are connected for all $t$, and such that each $g_t$ is standard with respect to each $\phi_i$ at height $z_i$, with the critical values of $g_1$ also ordered with respect to the heights $z_i$ as in the statement of Theorem~\ref{T:MorseExistence}. 
\end{lemma}

\begin{proof}
 In the process of ordering the critical points, we just need to check that we do not create disconnected level sets. Disconnected level sets only come from $(m-1)$--handles attached along separating spheres. We will never be moving $(m-1)$--handles below other handles in the ordering process. If the attaching sphere for an $(m-1)$--handle $H$ is nonseparating, it can only be made separating by attaching another $(m-1)$--handle below $H$. Thus we never need to create disconnected level sets while ordering if all level sets were connected to begin with.
\end{proof}

For the sake of completeness, we could also state a parameterized version, in which one homotopes a homotopy with connected level sets to an ordered homotopy of homotopies. This would be used as preparation for applying Theorem~\ref{T:MorseDefOfDefsExistence}. However, we will only need such a lemma at two points (in the proofs of Theorem~\ref{T:SquareUniqueness} and Lemma~\ref{L:HorizCritSwitch}) so we will prove it when we need it.

\section{Theorems about $I^2$--valued Morse $2$--functions on cobordisms
between cobordisms}

Throughout this section let $X$ be an oriented connected $n$--dimensional cobordism from $M_0$ to $M_1$, where each $M_i$ is a nonempty oriented $(n-1)$--dimensional cobordism from $F_{i0}$ to $F_{i1}$, with $F_{00} \cong F_{10}$ and $F_{01} \cong F_{11}$, with each $F_{ij}$ oriented, closed and nonempty. Recall that this means that $\partial X$ is equipped with a fixed identification with $-M_0 \cup (I \times F_{00}) \cup (I \times (-F_{01})) \cup M_1$, with $F_{0j} \subset M_0$ identified with $\{0\} \times F_{0j}$ and $F_{1j} \subset M_1$ identified with $\{1\} \times F_{0j}$ (see Figure~\ref{F:Intro}). Also suppose that we are given indefinite, ordered Morse functions $\zeta_0 : M_0 \to I$ and $\zeta_1 : M_1 \to I$. On $I^2$ we use coordinates $(t,z)$, thinking of $t$ as horizontal and $z$ as vertical. 

Note that we are not at this point assuming that the $M_i$'s or the $F_{ij}$'s are connected. However, the assumption that each $\zeta_i$ is indefinite implies that each component of $M_i$ is a connected cobordism between nonempty components of the $F_{ij}$'s.

Our goal in this section is to prove the following two theorems, an existence theorem and a uniqueness theorem for square Morse $2$--functions:

\begin{theorem} \label{T:SquareExistence}
Suppose that $n \geq 3$. Given any indefinite, ordered Morse function $\tau : X \to I$ which is projection to $I$ on $I \times F_{00}$ and $I \times (-F_{01})$ in $\partial X$, there exists an indefinite square Morse $2$--function $G : X \to I^2$ such that $z \circ G|_{M_0} = \zeta_0$, $z \circ G|_{M_1} = \zeta_1$ and $t \circ G = \tau$. If $n \geq 4$ and each $M_i$ and $F_{ij}$ is connected, then we can arrange that $G$ is fiber-connected.
\end{theorem}

\begin{theorem} \label{T:SquareUniqueness}
 Suppose that $n \geq 4$. Given two indefinite square Morse $2$--functions $G_0, G_1 : X \to I^2$ with $z \circ G_0 |_{M_0} = z \circ G_1 |_{M_0} = \zeta_0$ and $z \circ G_0 |_{M_1} = z \circ G_1 |_{M_1} = \zeta_1$, and such that both $\tau_0 = t \circ G_0$ and $\tau_1=t \circ G_1$ are ordered and indefinite, there exists an indefinite generic homotopy $G_s : X \to I^2$ between $G_0$ and $G_1$ such that $\tau_s = t \circ G_s$ is ordered (and $\tau_s$ restricted to $I \times F_{0i}$ is projection onto $I$, as in Definition~\ref{D:SquareM2F}). If in addition $G_0$ and $G_1$ are fiber-connected then we can arrange that $G_s$ is fiber-connected.
\end{theorem}

Note that in the preceding section we have already proved these two theorems in the special case that $X = [0,1] \times M$, $M_0 = \{0\} \times M$, $M_1 = \{1\} \times M$, and $\tau = \tau_0 = \tau_1$ is projection to $[0,1]$. This is because a generic homotopy $g_t : M \to I$ between Morse functions $g_0, g_1 : M \to I$ gives a square Morse $2$--function $G : X \to I^2$ defined by $G(t,p) = (t,g_t(p))$, with $z \circ G|_{M_i} = g_i$. The key difference between a general square Morse $2$--function and one coming from a generic homotopy between Morse functions is that the ``Cerf graphic'' for a general square Morse $2$--function, i.e. the image $G(Z_G) \subset I^2$ of the critical point set, may have vertical tangencies. These vertical tangencies correspond precisely to critical points of the horizontal Morse function $t \circ G : X \to I$. In the absence of such vertical tangencies, the horizontal Morse function is trivial and hence $X$ is a product. We exploit these ideas repeatedly in the following proofs.

Before working on the proofs, we spend some time understanding neighborhoods of these critical points. In other words, when $t \circ G : X \to I$ has a critical point at $p \in X$, what can we say about $G$ near $p$? We first construct two local models which are illustrated in Figure~\ref{F:Morse2FcnHandles} (recall that $\mu_k^n(\mathbf{x}) = \mu_k^n(x_1, \ldots, x_n) = -x_1^2 - \ldots - x_k^2 + x_{k+1}^2 + \ldots + x_n^2$):
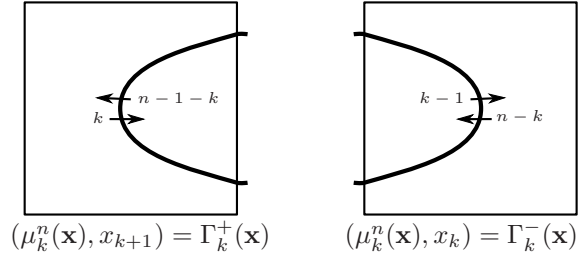
\begin{figure}[ht!]
\labellist
\small\hair 2pt
\pinlabel {$(\mu_k^n(\mathbf{x}), x_k) = \Gamma_k^-(\mathbf{x})$} [t] at 165 1
\pinlabel {$(\mu_k^n(\mathbf{x}), x_{k+1}) = \Gamma_k^+(\mathbf{x})$} [t] at 44 1
\tiny
\pinlabel $k$ [r] at 32 38
\pinlabel $n-1-k$ [l] at 41 45
\pinlabel $n-k$ [l] at 176 38
\pinlabel $k-1$ [r] at 168 45
\endlabellist
\begin{center}
\includegraphics{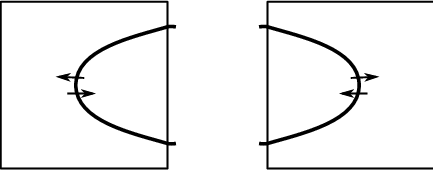}
\vspace*{10pt}
\caption{Local models for Morse critical points of index $k$ realized via Morse $2$--functions. On the left is a ``forward $k$--handle'', on the right is a ``backward $k$--handle''. Again, the arrows indicate the index of the fold in the given direction, not the index of the Morse critical point.}
\label{F:Morse2FcnHandles}
\end{center}
\end{figure}

\begin{enumerate}
 \item We call the following map a {\em forward index $k$ critical point}, or a {\em forward $k$--handle}: $\Gamma_k^+(x_1, \ldots, x_n) = (\mu_k^n(\mathbf{x}), x_{k+1})$. Note that this only makes sense for $0 \leq k \leq n-1$. By construction $t \circ \Gamma_k^+$ has a critical point of index $k$ at $0$. Reparametrizing the range by $(t,z) \mapsto (t-z^2,z)$ transforms $\Gamma_k^+$ to $(x_1, \ldots, x_n) \mapsto (\mu_k^{n-1}(x_1, \ldots, \widehat{x_{k+1}}, \ldots, x_n), x_{k+1})$, showing that $\Gamma_k^+$ really is a (local) Morse $2$--function with a single fold $Z_{\Gamma_k^+}$ along the $x_{k+1}$ axis and with the image $\Gamma_k^+(Z_{\Gamma_k^+})$ of this fold equal to the rightward-opening parabola $t=z^2$. The fold is indefinite when $0 < k < n-1$. 
 \item This map is a {\em backward index $k$ critical point}, or a {\em backward $k$--handle}: $\Gamma_k^-(x_1, \ldots, x_n) = (\mu_k^n(\mathbf{x}), x_{k})$. This only makes sense for $1 \leq k \leq n$. Again, by construction $t \circ \Gamma_k^-$ has a critical point of index $k$ at $0$. In this case, reparametrizing the range by $(t,z) \mapsto (t+z^2,z)$ transforms $\Gamma_k^-$ to $(x_1, \ldots, x_n) \mapsto (\mu_{k-1}^{n-1}(x_1, \ldots, \widehat{x_{k}}, \ldots, x_n), x_{k})$, showing that $\Gamma_k^-$ really is a (local) Morse $2$--function with a single fold $Z_{\Gamma_k^-}$ along the $x_{k}$ axis and with the image $\Gamma_k^-(Z_{\Gamma_k^-})$ of this fold equal to the leftward-opening parabola $t=-z^2$. This fold is indefinite when $1 < k < n$.
\end{enumerate}
Note that we could have defined $\Gamma_k^+$ by $\Gamma_k^+(\mathbf{x}) = (\mu_k^n(\mathbf{x}),\pm x_j)$ for any $j \in \{k+1, \ldots, n\}$ and it would still have all the properties listed, and in fact such a definition is equivalent to the one given up to a change of coordinates in the domain. Similarly we could define $\Gamma_k^-(\mathbf{x}) = (\mu_k^n(\mathbf{x}), \pm x_j)$ for any $j \in \{1, \ldots, k\}$. However, $\Gamma_k^+$ and $\Gamma_k^-$ are not in general related by a change of coordinates, even allowing a change of coordinates in the range, because the indices of the folds are different, as can be seen in Figure~\ref{F:Morse2FcnHandles}. If we turn a forward $k$--handle backwards, i.e. post-compose $\Gamma_k^+$ with the time-reversal $(t,z) \mapsto (-t,z)$, it becomes a backward $(n-k)$--handle.

Here are some further observations about forward $k$--handles, illustrated below in Figure~\ref{F:ForwardkHandle}; the reader can figure out the parallel statements for backward handles:
\begin{enumerate}
 \item When $0 < k \leq n-1$, the descending disk $\{x_{k+1} = \ldots = x_n = 0\}$ has image equal to the horizontal line $\{t \leq 0, z=0\}$. When $k=0$ there is, of course, no descending disk.
 \item When $0 \leq k < n-1$, the ascending disk $\{x_1 = \ldots = x_k = 0\}$ has image equal to the ``interior'' of the parabola $\{ t \geq z^2 \}$. (Of course, when $k=0$ the ascending disk is the whole domain of the function.) When $k=n-1$, the image of the ascending disk is just the parabola $\{ t = z^2 \}$. 
 \item For $0 < k \leq n-1$, the descending $(k-1)$--sphere $\{x_{k+1} = \ldots = x_n = 0, x_1^2 + \ldots + x_k^2 = R^2\}$ has image equal to the point $(-R^2,0)$.
 \item For $0 < k \leq n-1$, the attaching region for the $k$--handle, which we identify as the set $\{\mu_k^n(\mathbf{x})=-R^2, -\epsilon \leq x_{k+1} \leq \epsilon,  x_{k+2}^2 +  \ldots + x_n^2 \leq \epsilon^2\} \cong S^{k-1} \times [-\epsilon,\epsilon] \times B^{n-1-k}$, has image equal to the line segment $\{t = -R^2, -\epsilon \leq z \leq \epsilon\}$, and the map to this line segment is simply projection onto $[-\epsilon,\epsilon]$. This is where we first see the relevance of the conditions in the preceding section regarding constructing Morse functions which are standard with respect to embeddings of $[-\epsilon,\epsilon] \times B^{n-1-k} \times S^{k-1}$.
 \item For $k < n-1$, the ascending $(n-k-1)$--sphere $\{x_1 = \ldots = x_k = 0, x_{k+1}^2 + \ldots + x_n^2 =  R^2\}$ maps to the line $\{t=R^2\}$ via the standard Morse function on a sphere, with image equal to the interval $\{-R \leq z \leq R\}$. For $k=n-1$ the ascending sphere is two points mapping to $(R^2, \pm R)$.
\end{enumerate}
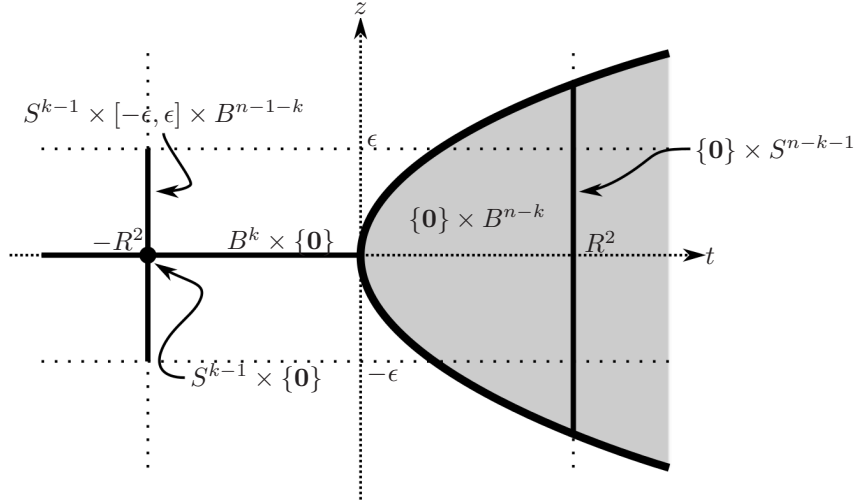
\begin{figure}[ht!]
\labellist
\small\hair 2pt
\pinlabel $z$ [b] at 134 183
\pinlabel $t$ [l] at 262 94
\pinlabel $\epsilon$ [bl] at 135 134
\pinlabel $-\epsilon$ [tl] at 135 54
\pinlabel $-R^2$ [br] at 54 94
\pinlabel $R^2$ [bl] at 217 94
\pinlabel {$\{\mathbf{0}\} \times B^{n-k}$} at 178 108
\pinlabel {$\{\mathbf{0}\} \times S^{n-k-1}$} [l] at 258 134
\pinlabel {$B^k \times \{\mathbf{0}\}$} [b] at 104 92
\pinlabel {$S^{k-1} \times \{\mathbf{0}\}$} [l] at 68 48
\pinlabel {$S^{k-1} \times [-\epsilon,\epsilon] \times B^{n-1-k}$} [b] at 60 140
\endlabellist
\begin{center}
\includegraphics{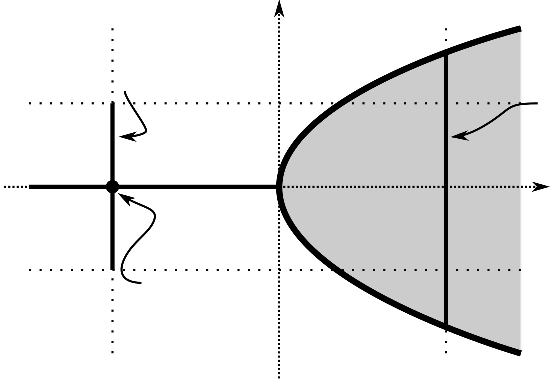}
\caption{Images of ascending and descending disks and spheres, and attaching region, for a forward $k$-handle $B^{k} \times B^{n-k}$. The parabola $t=z^2$ is the image of the fold.}
\label{F:ForwardkHandle}
\end{center}
\end{figure}

First we prove existence using these local models:
\begin{proof}[Proof of Theorem~\ref{T:SquareExistence}]
If $\tau$ has no critical points then we use a gradient flow to identify $X$
with $I \times M_0$ such that $\tau(t,p)=t$, and then we see $\zeta_0$ and
$\zeta_1$ as two indefinite Morse functions on $M_0$. Then our result follows
from Theorem~\ref{T:MorseDeformationExistence}; we get an indefinite generic homotopy $\zeta_t$ and we let $G(t,p) = (t,\zeta_t(p))$.

Thus if we can now prove the theorem in the case where $\tau$ has exactly one
critical point $p \in X$, then we are done. Suppose that $\tau(p) = 1/2$ and
that $p$ has index $k \leq n/2$. (If $k > n/2$ then replace $\tau$ with
$1-\tau$ and switch $M_0$ and $M_1$.) Choose a gradient-like vector field $V$ for $\tau$, and use this to find an embedding $\phi : [-\epsilon,\epsilon] \times B^{n-1-k} \times S^{k-1} \hookrightarrow M_0$ which gives the glueing map for the associated handle with appropriate framing. (We have split the normal direction into a product of
$[-\epsilon,\epsilon]$ and $B^{n-1-k}$ as preparation for the use of Theorem~\ref{T:MorseExistence}. Note that $k-1$ is the dimension referred to as $l_i$ in the statement of that theorem, and that the dimension of $M_0$ is $m=n-1$. We need to verify that $k-1 < (n-1)/2$, which we do have, because $k-1 \leq (n/2)-1 = (n-2)/2 < (n-1)/2$.) Let $T \subset \tau^{-1}(1/2 - \delta)$ be the image of $\phi$. For some small $\delta > 0$ we can then decompose $X$ into a union of four parts, $X = X_0 \cup X_c \cup H \cup X_1$ with the following properties: (Figure~\ref{F:Morse2FcnConstruction} shows where these four parts will sit in $I^2$ and shows what the singular locus $G(Z_G)$ will look like in each part.)
\begin{figure}[ht!]
\labellist
\small\hair 2pt
\pinlabel $X_0$ at 24 145
\pinlabel $X_1$ at 148 45
\pinlabel $X_c$ at 79 36
\pinlabel $X_c$ at 80 139
\pinlabel $X_c$ at 78 69
\pinlabel $H$ at 73 90
\endlabellist
\begin{center}
\includegraphics{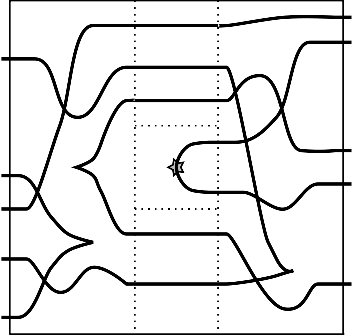}
\caption{Constructing a Morse $2$--function with a single critical point in the
horizontal Morse function. This diagram shows the images of the four parts of $X$, $X_0$, $X_1$, $X_c$ and the handle $H$; the image of the critical point is at the star. Note that the images of $X_0$ and $X_1$ are disjoint, while the image of $X_c$ contains the image of $H$, even though $X_c$ and $H$ intersect in $X$ only along their boundaries.}
\label{F:Morse2FcnConstruction}
\end{center}
\end{figure}
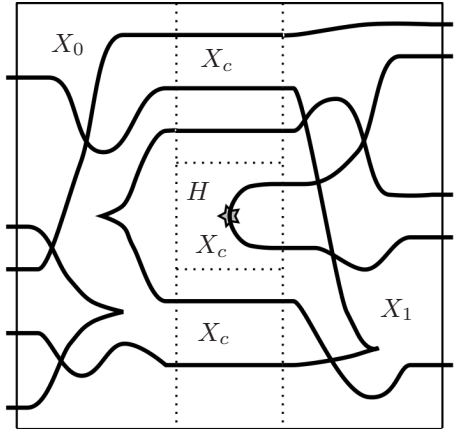
\begin{enumerate}
 \item $X_0 = \tau^{-1}[0,1/2-\delta]$ and is identified, via $V$, with
$[0,1/2-\delta] \times M_0$ in such a way that $\tau|_{X_0}(t,p) = t$.
 \item $X_1 = \tau^{-1}[1/2+\delta,1]$ and is identified, via $V$, with
$[1/2+\delta,1] \times M_1$ in such a way that $\tau|_{X_1}(t,p) = t$.
 \item $X_c \subset \tau^{-1}[1/2-\delta,1/2+\delta]$ and is identified, via
$V$, with $[1/2-\delta,1/2+\delta] \times (M_0 \setminus (T \setminus \partial
T))$, in such a way that $\tau|_{X_c}(t,p) = t$.
 \item $H$ is the $k$--handle, the union of the forward flow lines for $V$
starting at $T$, together with the ascending manifold of $p$ (using $V$), intersected with $\tau^{-1}[1/2-\delta,1/2+\delta]$. On $H$ we have coordinates $(x_1, \ldots, x_n)$ with respect to which $\tau(\mathbf{x}) = 1/2 + \mu_k^n(\mathbf{x})$ and $V = -2x_1 \partial_{x_1} - \ldots - 2x_k \partial_{x_k} + 2x_{k+1} \partial_{x_{k+1}} + \ldots + 2x_n \partial_{x_n}$. We choose these coordinates so that the $x_{k+1}$ direction is the $[-\epsilon,\epsilon]$ direction in the attaching region $[-\epsilon,\epsilon] \times B^{n-1-k} \times S^{k-1}$ (and $(x_{k+2}, \ldots, x_n)$ give the $B^{n-1-k}$ directions while the sphere in the $(x_1, \ldots, x_k)$ coordinates gives the $S^{k-1}$ factor).
\end{enumerate}
In order to construct $G$, we first use Theorem~\ref{T:MorseExistence} to each component of $M_0$ to construct an indefinite, ordered Morse function $\zeta_{1/2-\delta} : M_0 \to I$ which is standard with respect to $\phi$ at height $1/2$, so that $\zeta_{1/2-\delta}$ on the attaching region $T$ is of the form $(t,x,p) \mapsto t+1/2$ (identifying $T$ with $[-\epsilon, \epsilon] \times B^{n-1-k} \times S^{k-1}$ via $\phi$), with critical values of index $\leq k-1$ at heights $< 1/2$ and critical values of index $\geq k$ at heights $> 1/2$. Now use Theorem~\ref{T:MorseDeformationExistence} to construct an indefinite, ordered generic homotopy $\zeta_{t}$ (for $t \in [0,1/2-\delta]$) connecting $\zeta_{0}$ to $\zeta_{1/2-\delta}$. Then we let $G : X_0 \to [0,1/2-\delta] \times I$ be $G(t,p) = (t,\zeta_{t}(p))$, after identifying $X_0$ with $[0,1/2-\delta] \times M_0$ as above. On $H$, at first just let $G(\mathbf{x}) = \Gamma_k^+(\mathbf{x})+(1/2,1/2)$, a forward $k$--handle. This gives the single vertical 
tangency as part of a horizontal parabola seen in the middle in Figure~\ref{F:Morse2FcnConstruction}. This fits together smoothly with the definition of $G$ on $X_0$. Now we postcompose with an isotopy of $[1/2-\delta,1/2+\delta] \times [0,1]$ to make the image of $H$ exactly equal to the square $[1/2-\delta,1/2+\delta] \times [1/2-\epsilon,1/2+\epsilon]$, so that $G$ as defined on $X_0$ and $H$ extends smoothly to $X_c \cong [1/2-\delta,1/2+\delta] \times (M_0 \setminus (T \setminus \partial T))$ via $G(t,p) = (t,\zeta_{1/2-\delta}(p))$. One sees then that these definitions fit together smoothly to define $G$ over $\tau^{-1}[0,1/2+\delta]$, and that $z \circ G$ then defines an indefinite Morse function $\zeta_{1/2+\delta}$ on $\tau^{-1}\{1/2+\delta\}$, which is identified with $M_1$ via $V$. 
Finally, we use Theorem~\ref{T:MorseDeformationExistence} to construct an indefinite generic homotopy $\zeta_{t}$ (for $t \in [1/2+\delta,1]$) connecting $\zeta_{1/2+\delta}$ to $\zeta_1$, and we define $G$ on $X_1 \cong [1/2+\delta,1] \times M_1$ by $G(t,p) = (t,\zeta_{t}(p))$.

If we arranged that $\zeta_{t}$ is ordered for $t \in [1/2+\delta,1]$, we would have completed the proof of the connectedness assertion. This is fine if $\zeta_{1/2+\delta}$ (handed to us by the construction on $X_0 \cup X_c \cup H)$ is ordered.
There are two different cases when $\zeta_{1/2+\delta}$ will not be ordered. The first case is when $k=n/2$. Here we have a critical point of index $n/2$ below a critical point of index $n-1-n/2 = n/2-1$. The relevant indices are indicated on the left in Figure~\ref{F:UnorderedAfterCriticalHandle}. However, as long as the level sets of $\zeta_{1/2+\delta}$ are connected, we can start off the homotopy $\zeta_t$, $t \in [1/2+\delta,1]$, by switching the heights of the two offending critical points as indicated on the right in Figure~\ref{F:UnorderedAfterCriticalHandle}. Because $\zeta_{1/2-\delta}$ has critical values of index $\leq k-1$ below $1/2$ and $\geq k$ above $1/2$, $\zeta_{1/2+\delta}$ is now ordered. The only case in which the level sets might be disconnected in the short space when $\zeta_t$ is not ordered is when $n/2 = n-2$ and $n/2-1 = 1$, i.e. when $n=4$ and $k=2$. In this case we should make sure that, in $\zeta_{1/2-\delta}$, the attaching $S^1$ for the $2$--handle $H$ does not separate the level set in which it lies. A moment's thought about the proof of Theorem~\ref{T:MorseExistence} shows that this is easy to achieve.

The second case when $\zeta_{1/2+\delta}$ will not be ordered is when $n-1-k > k$, in which case the upper of the two new critical points, which has index $n-1-k$, may have critical points of lower index above it. In this case start off the homotopy $\zeta_t$, $t \in [1/2+\delta,1]$, by lifting this critical point above those of lower index, and then proceed as above. The only case when this could conceivably cause any connectedness problems is when $k=1$ and $n-k-1=n-2$, but adding a $1$--handle and its dual $(n-2)$--handle to a fiber that is already connected cannot disconnect the fiber.
\begin{figure}[ht!]
\labellist
\small\hair 2pt
\pinlabel $n/2$ [r] at 34 41
\pinlabel $n/2$ [r] at 158 41
\pinlabel $n/2$ [tl] at 80 17
\pinlabel $n/2-1$ [tl] at 80 65
\pinlabel $n/2$ [tl] at 284 65
\pinlabel $n/2-1$ [tl] at 284 17

\endlabellist
\begin{center}
\includegraphics{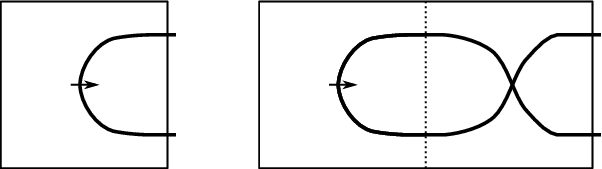}
\caption{After attaching a handle of index $n/2$, the vertical Morse function will not be ordered, as on the left. The arrow labelled $n/2$ indicates that this fold is of index $n/2$ when looked at in the direction of the arrow. The numbers to the right of each box are the index of the critical points of the vertical Morse function there. In the box on the right, we have switched the two critical points to restore order.}
\label{F:UnorderedAfterCriticalHandle}
\end{center}
\end{figure}
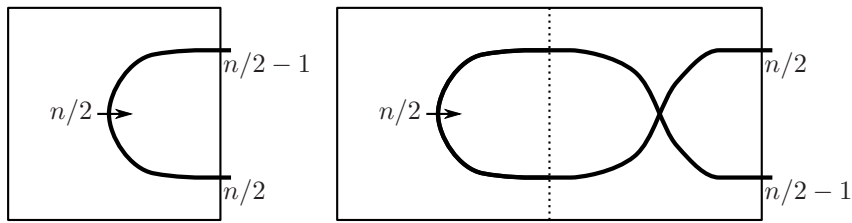
\end{proof}

The rest of this section is devoted to proving uniqueness, i.e. the proof of Theorem~\ref{T:SquareUniqueness}. To do this, we first show that our two models, forward and backward handles, are a complete list of local models, in the following sense:

\begin{lemma} \label{L:LocalSquareCritPtModel}
 Consider a square Morse $2$--function $G : X \to I^2$ and a critical point $p \in X$ of $t \circ G : X \to I$, of index $k$, with $G(p) = (t_p,z_p)$. Suppose we are given standard coordinates $(x_1, \ldots, x_n)$ on a neighborhood $\nu$ of $p$ such that $\tau = t \circ G(x_1, \ldots, x_n) = \mu_k^n(\mathbf{x}) + t_p$. Then there exists an arc $G_s$ of Morse $2$--functions (i.e. $G_s$ is Morse for all $s$) supported inside $\nu$, with $G_0=G$ and $t \circ G_s$ independent of $s$, such that, inside a smaller neighborhood $\nu' \subset \nu$, $G_1(\mathbf{x}) = \Gamma_k^\pm(\mathbf{x}) + (t_p,z_p)$. It will be $\Gamma_k^+$, i.e. a forward $k$--handle, exactly when the point $(t_p,z_p)$ is a local minimum for $t|_{G(Z_G)}$, and it will be $\Gamma_k^-$, a backward $k$--handle, exactly when $(t_p,z_p)$ is a local maximum.

\end{lemma}

\begin{proof}
 Let $\tau = t \circ G$ and $\zeta = z \circ G$, i.e. $G(\mathbf{x}) = (\tau(\mathbf{x}), \zeta(\mathbf{x}))$. We know that $\tau = \mu_k^n(\mathbf{x}) + t_p$. Because the rank of $DG$ at $p=0$ must be $1$, we know that $\zeta$ is nonsingular at $0$ and so, after a small perturbation of $\zeta$ (keeping $G$ Morse) supported in a neighborhood of $0$ we can assume that $\zeta$ is linear of the form $\zeta(x_1, \ldots, x_n) = a_1 x_1 + \ldots + a_n x_n + z_p$. Now note that $-a_1^2 - \ldots -a_k^2 + a_{k+1}^2 + \ldots + a_n^2 \neq 0$, because otherwise we would have a $1$--dimensional subset of the singular locus of $G$ on which $G$ was constant, which does not happen in any of the local models for a Morse $2$--function. Thus there is a linear transformation of $\R^n$ preserving the quadratic form $\mu_k^n$ and taking $(a_1, \ldots, a_n)$ to $(0, \ldots, 0, c, 0, \ldots, 0)$ for some positive $c$, with the $c$ in either the $k$'th or $(k+1)$'st coordinate. Smoothly interpolate from the identity to this linear 
transformation while going in radially towards the origin in our neighborhood, and this creates an ambient isotopy of the domain. Precomposing with this ambient isotopy and postcomposing with a rescaling isotopy gives the arc of Morse $2$--functions $G_s$ with the desired properties. Because we preserved the quadratic form $\mu_k^n$, we didn't change $\tau$ in this isotopy, and thus $t \circ G_s$ is independent of $s$.

\end{proof}

At some point one might wish that a certain forward index $k$ critical point was actually a backward index $k$ critical point, or vice versa. The next lemma addresses this:

\begin{lemma} \label{L:ForwardToBackward}
Suppose that $G : X \to I^2$ is a square Morse $2$--function and that $p \in G$ is a critical point of $t \circ G$ of index $k$ with local coordinates with respect to which $G(\mathbf{x}) = \Gamma_k^\pm(\mathbf{x})+(t_p,z_p)$. If $2 \leq k \leq n-2$ there exists an indefinite generic homotopy $G_s$ of Morse $2$--functions supported inside this coordinate neighborhood, with $G_0 = G$, such that, inside a smaller neighborhood of $p$, we have $G_1(\mathbf{x}) = \Gamma_k^\mp(\mathbf{x}) + (t_p,z_p)$. If all fibers of $G$ are connected then we can arrange that all fibers of $G_s$ are connected for all $s$. We can further arrange that $t \circ G_s$ is independent of $s$.
\end{lemma}

\begin{proof}
Figure~\ref{F:ForwardToBackward} shows how to do this (without the indefinite condition) when $n=2$ and $k=1$. For this we should in principle be able to write down an explicit formula, but the illustration probably explains what is going on better. We know that a forward handle has become a backward handle simply because the vertical tangency in the fold locus has changed from being a rightward-opening parabola to a leftward-opening parabola. The homotopy has passed through a swallowtail singularity. To get the higher dimensional version, and indefiniteness, consider Figure~\ref{F:ForwardToBackward} to be a picture of what is happening in the $(x_k,x_{k+1})$--plane, and keep the homotopy independent of $s$ in the other coordinates $(x_1, \ldots, x_{k-1})$ and $(x_{k+2}, \ldots, x_n)$. It is easy to see that, for $n \geq 4$ and $2 \leq k \leq n-2$, this is an indefinite homotopy and does not disconnect fibers. The way we have drawn it, the critical point moves slightly to the left, but after modifying by 
a small isotopy we can keep $t \circ G_s$ independent of $s$ throughout the homotopy. 
\begin{figure}[ht!]
\begin{center}
\includegraphics{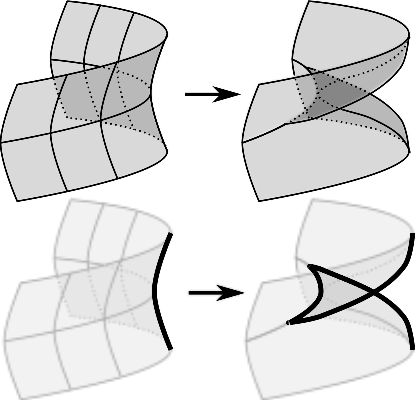}
\caption{Bending a forward handle to a backward handle. The point is to see that the only change to the projection to the horizontal axis in this homotopy is that the critical point moves a little to the left, but that when we look at the map to $2$ dimensions, i.e. projection to the page exactly as in this figure, the fold at the critical point changes from opening to the right to opening to the left. This figure is in fact nothing more than the standard homotopy that introduces a swallowtail.}
\label{F:ForwardToBackward}
\end{center}
\end{figure}
 
\end{proof}

We will want to accompany the lemmas above with a lemma stating that, outside the standard neighborhood of a critical point of $t \circ G$, $G$ can be taken to be ``constant in the $t$ direction''. We state this precisely as follows:

\begin{lemma} \label{L:ConstantOutsideHandle}
 Consider a square Morse $2$--function $G : X \to I^2$ and an index $k$ critical point $p \in X$ of $\tau = t \circ G$, with local coordinates near $p$ with respect to which $G = \Gamma_k^+(\mathbf{x})+(t_p,z_p)$. In these local coordinates, let $V$ be the standard gradient vector field for $\tau = t \circ G = \mu_k^n(\mathbf{x}) + t_p$ and, for any small $\epsilon > 0$, let $H_\epsilon$ be the $k$--handle neighborhood of $p$ obtained by flowing forward along $V$ starting from the attaching region $\{(x_1, \ldots, x_n) \mid \tau(x_1, \ldots, x_n) = -\epsilon + t_p, x_{k+1}^2 \leq \epsilon^2, x_{k+2}^2+ \ldots + x_n^2 \leq \epsilon^2 \} \cong S^{k-1} \times [-\epsilon,\epsilon] \times B^{n-k-1}$, together with the ascending disk, and stopping at $\tau = t_p+\epsilon$. Note that $H_\epsilon \subset \tau^{-1}[t_p-\epsilon,t_p+\epsilon]$; let $H_\epsilon^c$ be the closure of $\tau^{-1}[t_p-\epsilon,t_p+\epsilon] \setminus H_\epsilon$ and $M_{t_p-\epsilon}^c = H_\epsilon^c \cap \tau^{-1}(t_p-\epsilon)$. Using $V$, 
$H_\epsilon^c$ is naturally identified with $[t_p-\epsilon,t_p+\epsilon] \times M_{t_p-\epsilon}^c$. Then there exists an $\epsilon > 0$ and an arc of Morse $2$--functions $G_s$, with $G_0 = G$, which is independent of $s$ inside $H_{\epsilon}$ and independent of $s$ outside $\tau^{-1}(t_p-2\epsilon,t_p+2\epsilon)$, such that, on $H_\epsilon^c$, identified with $[t_p-\epsilon,t_p+\epsilon] \times M_{t_p-\epsilon}^c$, $G_1$ is of the form $(t,x) \mapsto (t,g(x))$ for a fixed Morse function $g = \zeta|_{\tau^{-1}(t_p-\epsilon)}$. In addition, we can arrange that there are no critical values of $g|_{M_{t_p-\epsilon}^c}$ in $[z_p-\epsilon,z_p+\epsilon]$.
\end{lemma}

\begin{proof}
 Given the standard model on the handle, the complement of the handle in $\tau^{-1}([t_p-2\epsilon,t_p+2\epsilon])$ is a product cobordism and $G$ on this product is identified with an arc of Morse functions $g_t$. A standard homotopy can make $g_t$ independent of $t$ for $t \in [t_p-\epsilon,t_p+\epsilon]$.
\end{proof}

We now present the proof of uniqueness for indefinite square Morse $2$--functions as a sequence of steps forward-referencing two further lemmas which will be stated and proved afterwards. 

\begin{proof}[Proof of Theorem~\ref{T:SquareUniqueness}]
 We are given two indefinite square Morse $2$--functions $G_0, G_1 : X \to I^2$ which agree on $M_0$ and $M_1$, i.e. $z \circ G_0 |_{M_i} = \zeta_i = z \circ G_1 |_{M_i}$, for $i = 0,1$. Then we need to construct an indefinite generic homotopy $G_s : X \to I^2$ between $G_0$ and $G_1$, and we need to address the issue of connected fibers. The steps are as follows:
 \begin{enumerate}
  \item Let $\tau_0 = t \circ G_0$ and let $\tau_1 = t \circ G_1$. These are indefinite $I$--valued Morse functions, and we are given that they are ordered. Let $\tau_s : X \to I$ be an indefinite, ordered generic homotopy from $\tau_0$ to $\tau_1$ such that all the births of cancelling pairs of critical points occur for $s \in [0,1/4]$ and all the deaths occur for $s \in [3/4,1]$, and such that $\tau_s$ is independent of $s$ for all $s \in [1/4,3/4]$. We then construct the desired generic homotopy $G_s$ for $s \in [0,1/4]$ and for $s \in [3/4,1]$ such that $t \circ G_s = \tau_s$. (We need to slightly modify $\tau_s$ to achieve this.) This step is carried out in Lemma~\ref{L:SquareDefFromIDef}. The key outcome of this step is that $t \circ G_{1/4} = t \circ G_{3/4}$, so that when we construct $G_s$ for $s \in [1/4,3/4]$, we can leave $t \circ G_s$ independent of $s$ and work on $z \circ G_s$.
  \item Now we need to connect $G_{1/4}$ to $G_{3/4}$. Let $\tau = t \circ G_{1/4} = t \circ G_{3/4}$. Our next step is to extend the homotopy $G_s$ to $s \in [1/4,1/2]$, keeping $t \circ G_s = \tau$ for all $s \in [1/4,1/2]$ so that, for some $\epsilon > 0$ and for each critical value $t_*$ of $\tau$, $G_{1/2}$ and $G_{3/4}$ agree on $\tau^{-1}([t_*-\epsilon,t_*+\epsilon]) = G_{1/2}^{-1}([t_*-\epsilon,t_*+\epsilon] \times I) = G_{3/4}^{-1}([t_*-\epsilon,t_*+\epsilon] \times I)$. This step is carried out in Lemma~\ref{L:SquareDefNearCritPoints}.
  \item Finally, we extend $G_s$ to $s \in [1/2,3/4]$ to connect $G_{1/2}$ to $G_{3/4}$ as follows: The parts of $X$ where $G_{1/2}$ and $G_{3/4}$ do not yet agree are of the form $X_* = \tau^{-1}([t_*+\epsilon,t_*'-\epsilon])$, for two consecutive critical values $t_* < t_*'$ of $\tau$. But then $X_*$ can be identified with a product $[t_*+\epsilon,t_*'-\epsilon] \times M_*$, where $M_* = \tau^{-1}(t_*+\epsilon)$. Furthermore, with this identification, for $s=1/2$ and $s=3/4$, we see that $G_s|_{X_*}$ is of the form $(t,p) \mapsto (t,g_{s,t}(p))$, precisely because $t \circ G_{1/2} = t \circ G_{3/4} = \tau$. Thus we can use Theorem~\ref{T:MorseDefOfDefsExistence} from the preceding section to construct a homotopy $g_{s,t}$ from $g_{1/2,t}$ to $g_{3/4,t}$, and then define $G_s$ for $s \in [1/2,3/4]$ and $p \in X_* = [t_*+\epsilon,t_*'-\epsilon] \times M_*$ by $G_s(t,p) = (t,g_{s,t}(p))$. (Here we actually need a parameterized version of Lemma~\ref{L:Connected2Ordered}, which is discussed in a similar context in the proof of Lemma~\ref{L:HorizCritSwitch}.) Finally, since $G_{1/2}$ and $G_{3/4}$ already agree on $\tau^{-1}([t_*-\epsilon,t_*+\epsilon])$ for critical points $t_*$,
 we can define $G_s = G_{1/2} = G_{3/4}$ on these subsets, for all $s \in [1/2,3/4]$, and we are done.
 \end{enumerate}
 
 In the above steps we did not address the issue of keeping fibers connected when $G_0$ and $G_1$ are fiber-connected. We have already arranged for $\tau_s$ to be ordered for all $s$. In this case, Lemma~\ref{L:SquareDefFromIDef} also states that $G_s$ will have all fibers connected for all $s \in [0,1/4]$ and for all $s \in [3/4,1]$. Lemma~\ref{L:SquareDefNearCritPoints} then explicitly asserts that, if the fibers of $G_{1/4}$ and $G_{3/4}$ are connected in each $\tau^{-1}([t_*-\epsilon,t_*+\epsilon])$, for $t_*$ a critical value of $\tau$, then we can keep the fibers of $G_s$ connected there when we construct $G_s$ for $s \in [1/4,1/2]$. Finally, when we use Theorem~\ref{T:MorseDefOfDefsExistence} to construct $G_s$ for $s \in [1/2,3/4]$, we should first use Lemma~\ref{L:Connected2Ordered} to get each of the Morse function $g_{1/2,t_*+\epsilon} = g_{3/4,t_*+\epsilon}$ and $g_{1/2,t_*'-\epsilon} = g_{3/4,t_*'-\epsilon}$ to be ordered.

\end{proof}

We now state and prove the two lemmas referenced in the proof above.

\begin{lemma} \label{L:SquareDefFromIDef}
 Given any indefinite Morse $2$--function $G : X \to I^2$ and an indefinite, ordered generic homotopy $\tau_s : X \to I$ between Morse functions, with $t \circ G = \tau_0$ and with no deaths of cancelling pairs of critical points, there exists an indefinite, ordered generic homotopy $\tau'_s : X \to I$, with $\tau'_0 = \tau_0$ and $\tau'_1 = \tau_1$, which is connected to $\tau_s$ by an arc of generic homotopies, and there exists an indefinite generic homotopy of Morse $2$--functions $G_s : X \to I^2$ with $G_0 = G$ and with $t \circ G_s = \tau'_s$. If $G$ is fiber-connected then we can arrange that $G_s$ is fiber-connected for all $s$.
\end{lemma}

\begin{proof}[Proof of Lemma~\ref{L:SquareDefFromIDef}]
We will show how to construct generic indefinite homotopies $G_s$ such that $t \circ G_s$ is a generic homotopy between Morse functions which realizes either (1) a desired birth of a cancelling pair or (2) a desired crossing of two critical points. The given $\tau_s$ will then tell us where the births should be and which critical points should cross when. Making the births occur at these points and the right critical points cross at the right time, and pre- and post-composing $G$ with ambient isotopies, we can construct $G_s$ so that $\tau'_s = t \circ G_s$ is connected to $\tau_s$ by an arc of generic homotopies. (To see this, first note that an arc of Morse functions can always be realized by pre- and post-composing with ambient isotopies. This is because we can post-compose with an isotopy so that the critical values are constant, then pre-compose so that the critical points are constant, then pre-compose again to arrange that the homotopy is constant on neighborhoods of the critical points, and finally 
integrate a time-like vector field to get the full isotopy. Then note that births and critical point crossings can be localized by using bump functions to keep given homotopies stationary for short time periods outside standard neighborhoods.)

We deal with these two moves as follows:

\begin{enumerate}
 \item The easiest way to arrange a birth is to arrange that, inside the ball in which the birth should occur, $G$ has the form $G(x_1, \ldots, x_n) = (-x_1^2 - \ldots - x_k^2 +x_{k+1} + x_{k+2}^2 + \ldots + x_n^2, x_{k+1})$. This is a fold which is index $k$ looked at from left to right and the image of the fold set is the line $z=t$. We can arrange this via, for example, an eye birth as illustrated in Figure~\ref{F:SquareDefFromTauBirth}; there are two cases, one which is indefinite for $1 \leq k \leq n-3$ and one which is indefinite for $2 \leq k \leq n-2$. Now let $f_s(x)$ be a function which equals $x^3 - s x$ for $x$ in a neighborhood of $0$, has no critical points outside that neighborhood for any $s \in [-\epsilon,\epsilon]$, and is linear in $x$ and independent of $s$ outside a slightly larger neighborhood. Finally let $G_s(x_1, \ldots, x_n) = (-x_1^2 - \ldots - x_k^2 + f_s(x_{k+1}) + x_{k+2}^2 + \ldots + x_n^2, x_{k+1})$, the result of which is also illustrated in Figure~\ref{F:SquareDefFromTauBirth}.
\begin{figure}[ht!]
\labellist
\small\hair 2pt
\pinlabel $n-2-k$ [t] at 40 118
\pinlabel $n-1-k$ [b] at 40 140
\pinlabel $k$ [b] at 164 154
\pinlabel $k+1$ [b] at 164 114
\pinlabel $k$ [br] at 145 136
\pinlabel $n-1-k$ [t] at 40 30
\pinlabel $n-k$ [b] at 40 52
\pinlabel $k$ [b] at 164 18
\pinlabel $k-1$ [b] at 164 58
\pinlabel $k$ [br] at 181 28
\endlabellist
\begin{center}
\includegraphics{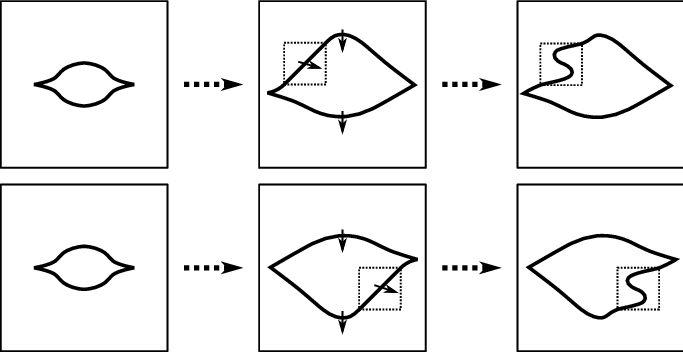}
\caption{Two ways to realize a birth of a cancelling $k$--$(k+1)$ pair in $\tau_s$ via a homotopy of $G$; the birth occurs inside the dotted square.}
\label{F:SquareDefFromTauBirth}
\end{center}
\end{figure}
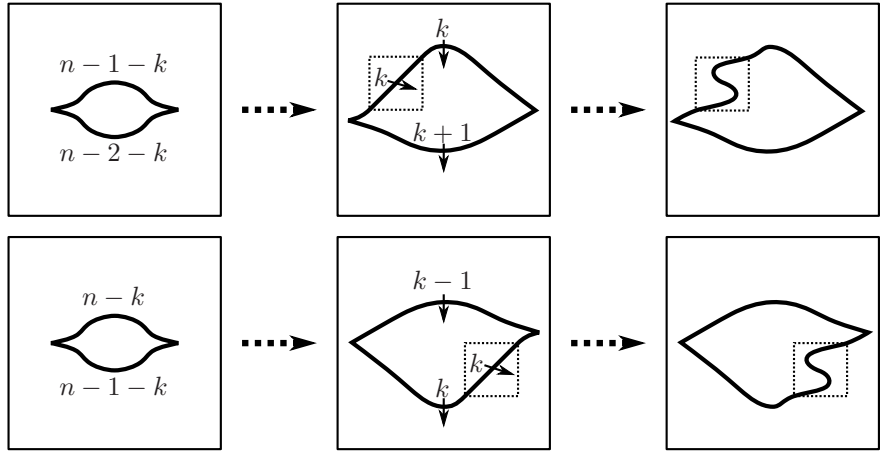
As long as this eye birth is indefinite, it will not disconnect fibers.

\item 
Since $\tau_s$ is ordered, we only need to switch critical values of the same index. Lemma~\ref{L:HorizCritSwitch} below deals with more general critical values switches, but applies in particular to this case. (We have separated that lemma from the proof of this lemma because we will need the more general case in the following section.) 

\end{enumerate}

\end{proof}

\begin{lemma} \label{L:HorizCritSwitch}
 Given an indefinite Morse $2$--function $G \co X \to I^2$, let $\tau = t \circ G$. Consider two critical points $p$ and $q$ of $\tau$, with $p$ having index $k \leq n/2$, with $q$ having index $l \geq k$, with $\tau(q) < \tau(p)$ and with no other critical values between $\tau(q)$ and $\tau(p)$. Fix a generic gradient-like vector field for $\tau$. Then there exists an indefinite generic homotopy $G_s$ with $G_0 = G$ such that $\tau_s = t \circ G_s$ agrees with $\tau$ outside a neighborhood of the descending manifold for $p$ and, inside this neighborhood, no new critical points are born and $\tau_s(p)$ decreases monotonically from $\tau(p)$ to $\tau(q)-\epsilon$ for some small $\epsilon > 0$. Furthermore, if $G$ is fiber-connected then we can arrange that $G_s$ is fiber-connected for all $s$.
\end{lemma}

\begin{proof}
 
Suppose that $\tau(p) = 2/3$ and $\tau(q)=1/3$ and that there are no other critical values
between $1/3$ and $2/3$. 
Using Lemmas~\ref{L:LocalSquareCritPtModel} and~\ref{L:ForwardToBackward}, we arrange for local coordinates near $p$ with respect to which $G$ has the form $G = \Gamma_k^+(\mathbf{x})+(2/3,z_p)$, a forward $k$--handle, and we arrange for a disjoint coordinate system near $q$ with respect to which $G = \Gamma_l^\pm(\mathbf{x})+(1/3,z_q)$, a forward or backward $l$--handle. We further use Lemma~\ref{L:ConstantOutsideHandle} to arrange that $G$ is constant in the $t$ direction outside a neighborhood of $q$ for $t \in [1/3-\epsilon,1/3+\epsilon]$. First we consider the case that $q$ is a a backward $l$--handle. 

Since there are no critical points between $p$ and $q$, $\tau^{-1}([1/3+\epsilon,2/3-\epsilon])$ is diffeomorphic to $[1/3+\epsilon,2/3-\epsilon] \times M$ for an $(n-1)$--manifold $M$ (each component of which is a connected cobordism between non-empty manifolds), and via this diffemorphism $G$ has the form $(t,p) \mapsto (t,g_t(p))$ where $g_t$ is an indefinite generic homotopy between Morse functions on $M$. Furthermore, because of the local models we have found for $G$ near $p$ and $q$, we have an embedding $\phi_q$ of a neighborhood of the ascending sphere for $q$ and an embedding $\phi_p$ of a neighborhood of the descending sphere for $p$ in $M$ such that $g_{1/3+\epsilon}$ is standard with respect to $\phi_q$ at height $z_q$ while $g_{2/3-\epsilon}$ is standard with respect to $\phi_p$ at height $z_p$. 

We will now sequentially improve $g_t$ in preparation for lowering $p$ past $q$. This is illustrated in Figure~\ref{F:CrossingHorizCritPts}, where ``lowering'' $p$ really means moving $p$ to the left. At each stage we produce an improved $g_t$, which must be connected to the preceding $g_t$ by a homotopy $g_{s,t}$. The homotopy $g_{s,t}$ is produced by appealing to Theorem~\ref{T:MorseDefOfDefsExistence}. The one hitch here, as pointed out helpfully by our anonymous referee, is that Theorem~\ref{T:MorseDefOfDefsExistence} requires ordered Morse homotopies as input, and at certain stages our two Morse homotopies may have connected level sets without necessarily being ordered. For this, we need a parameterized version of Lemma~\ref{L:Connected2Ordered}. Now we have a given homotopy $g_t$ which has connected level sets and must be made ordered, respecting standardization with respect to $\phi_p$ and/or $\phi_q$ (see below). Here we simply pull the arcs of critical points down or up, as appropriate, using exactly the same argument as in last two paragraphs of the proof of Theorem~\ref{T:MorseDeformationExistence}.

Now here are the sequence of improvements of $g_t$: First we use Lemma~\ref{L:Connected2Ordered}, Lemma~\ref{L:ConstantOutsideHandle}, Theorem~\ref{T:MorseExistence} and Theorem~\ref{T:MorseDeformationExistence} to arrange that $g_t$ is ordered and standard with respect to $\phi_p$ at height $z_p$ for $t$ near $1/3$. Similarly we arrange for $g_t$ to be ordered near $2/3$, keeping it standard with respect to $\phi_p$ at height $z_p$ near $2/3$. (Note the importance of the $t$--independence of $G$ outside a neighborhood of $q$, as given by Lemma~\ref{L:ConstantOutsideHandle}, so that we can smoothly connect the behavior of $g_{1/3+\epsilon}$ to the behavior of $g_{1/3-\epsilon}$, and similarly for the $g_{2/3-\epsilon}$ and $g_{2/3+\epsilon}$.) Then we use Theorem~\ref{T:MorseDeformationExistence} to arrange that $g_t$ is standard with respect to $\phi_p$ at height $z_p$ (and ordered) for all intermediate values of $t$. The argument from the preceding paragraph connects these improved $g_t$'s by the appropriate homotopies of homotopies of Morse functions, which turn into homotopies of Morse $2$--functions. Finally, having arrange the standardness of $g_t$ for the point $p$ on the entire interval $[1/3-\epsilon, 2/3+\epsilon]$, we can easily lower $p$ past $q$.
\begin{figure}%[ht!]
\labellist
\small\hair 2pt
\pinlabel $p$ [r] at 104 185
\pinlabel $q$ [l] at 48 145
\pinlabel $z_p$ [l] at 153 185
\pinlabel $z_q$ [l] at 153 145
\pinlabel $\frac{1}{3}$ [t] at 49 121
\pinlabel $\frac{2}{3}$ [t] at 105 121
\pinlabel $t$ [t] at 76 121
\pinlabel $z$ [r] at 1 167
\pinlabel $s_1$ [b] at 76 208
\pinlabel $s_2$ [b] at 257 208
\pinlabel $s_3$ [b] at 76 89
\pinlabel $s_4$ [b] at 257 89

\endlabellist
\begin{center}
\includegraphics{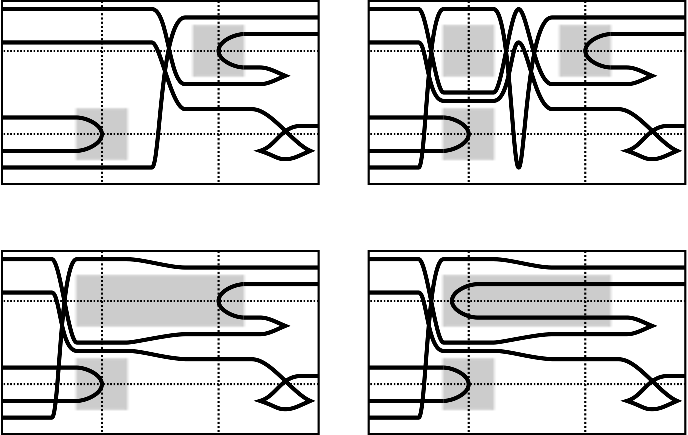}
\caption{Passing two critical points past each other; the homotopy is illustrated at four successive values of $s$. Shaded squares are regions where $g_t$ is standard with respect to $\phi_p$ or $\phi_q$. To pass from $s_1$ to $s_2$, we modify $g_t$ in a neighborhood of $t=1/3$ so that the framed attaching sphere for $p$ lies in a level set at $t=1/3$. Going from $s_2$ to $s_3$, we modify $g_t$ between $t=1/3$ and $t=2/3$ so that this attaching sphere lies in a level set on the whole interval $[1/3,2/3]$.}
\label{F:CrossingHorizCritPts}
\end{center}
\end{figure}
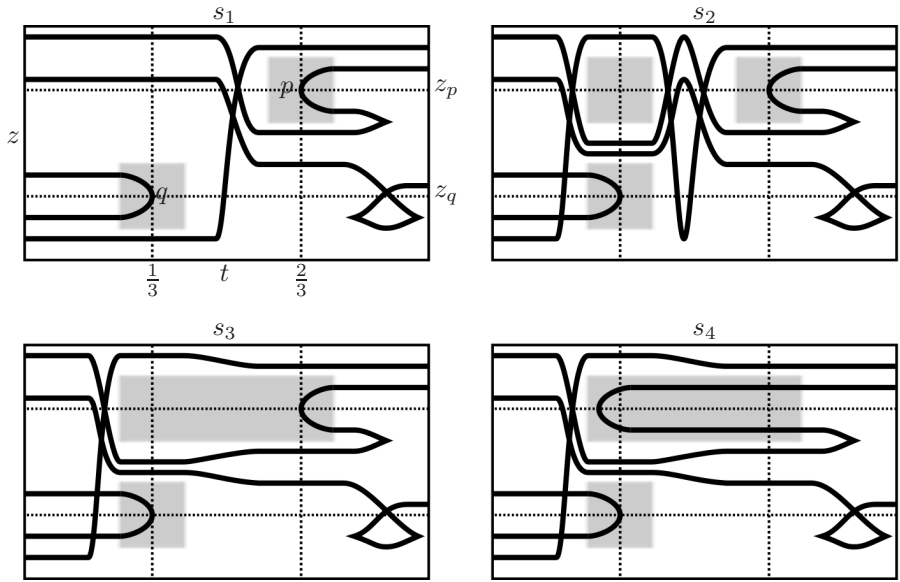

In the case that $q$ is a forward $l$--handle, we use the same argument as above but look first at $M' = \tau^{-1}(1/3-\epsilon)$, with a vertical Morse function $g_{1/3-\epsilon} \co M' \to I$ which is standard with respect to the descending manifold for $q$ at height $z_q$. We modify this to be standard with respect to the descending manifolds for both $p$ and $q$, using the $t$--independence of $G$ outside a neighborhood of $q$ for $t \in [1/3-\epsilon,1/3+\epsilon]$, so that again we get $g_{1/3+\epsilon} \co M \to I$ to be standard with respect to the descending manifold for $p$ and proceed as above.

Regarding fiber-connectedness, note that the only potential problem arises when the attaching sphere for $p$ has codimension $1$ in the fiber, and then we need to make sure that it never separates the fibers as we lower $p$ past $q$. Thus we only need to worry when $k=n-2$, but since $k \leq n/2$ and $n \geq 4$, the only case of concern is $n=4$ and $k=2$. (To a $4$--manifold topologist this is, of course, the most interesting case.) Thus, when we apply Theorems~\ref{T:MorseExistence} and ~\ref{T:MorseDeformationExistence} and Lemma~\ref{L:Connected2Ordered} to arrange that $g_t$ is standard with respect to $\phi_p$ at height $z_p$ for all values of $t$ between $1/3$ and $2/3$, we actually need to do a little more: We need to ensure that the descending sphere for $p$ remains nonseparating in its fiber for all $t$. Thus we need a slight improvement on Theorem~\ref{T:MorseDeformationExistence} which says that, in this special case $m=n-1=3$ and $l_i=k-1=1$, we can maintain the nonseparating property throughout a 
generic homotopy between Morse functions. The easiest way to do this is to arrange that a dual circle to the attaching circle also lies in the fiber throughout the homotopy, i.e. that the union of two circles intersecting transversely at one point stays in the fiber. Since the proof of the relevant part of Theorem~\ref{T:MorseDeformationExistence} simply involves counting dimensions and appealing to transversality, this slight improvement is straightforward.

\end{proof}

The other lemma needed in the proof of Theorem~\ref{T:SquareUniqueness} is:

\begin{lemma} \label{L:SquareDefNearCritPoints}
 Given two indefinite Morse $2$--functions $G, G' : X \to I^2$ such that $z
\circ G |_{M_0} = z \circ G' |_{M_0} = \zeta_0$, $z \circ G |_{M_1} = z \circ
G' |_{M_1} = \zeta_1$ and $t \circ G = t \circ G' = \tau : X \to I$, there
exist an $\epsilon > 0$ and a generic indefinite homotopy of Morse $2$--functions
$G_s : X \to I^2$ such that:
\begin{enumerate}
 \item $G_0 = G$,
 \item $t \circ G_s = \tau$ for all $s$, and
 \item for each critical value $t_*$ of $\tau$, letting $X_* =
\tau^{-1}[t_*-\epsilon,t_*+\epsilon]$, we have $G_1|_{X_*} = G'|_{X_*}$.
\end{enumerate}
If $G$ and $G'$ have all fibers connected then we can arrange for $G_s$ to have all fibers connected as well. 
\end{lemma}

\begin{proof}
Again, we use Lemma~\ref{L:LocalSquareCritPtModel} to standardize $G =
(\tau,\zeta)$ and $G' = (\tau,\zeta')$ near each critical point of $\tau$, so
that $\zeta$ and $\zeta'$ are equal in a neighborhood of each critical point. Then use Lemma~\ref{L:ConstantOutsideHandle} to make $G$ and $G'$ ``constant in the $t$ direction'' inside each $\tau^{-1}[t_*-\epsilon,t_*+\epsilon]$ but away from the critical point. Thus if we homotope $\zeta|_{\tau^{-1}(t_*-\epsilon)}$ to $\zeta'|_{\tau^{-1}(t_*-\epsilon)}$ without changing $\zeta$ on the attaching region for the handle associated to this critical point, this homotopy can be spread out over $[t_*-\epsilon,t_*+\epsilon]$ to give the desired homotopy of $G$. The homotopy from $\zeta|_{\tau^{-1}(t_*-\epsilon)}$ to $\zeta'|_{\tau^{-1}(t_*-\epsilon)}$ is given by Lemma~\ref{L:Connected2Ordered} followed by Theorem~\ref{T:MorseDeformationExistence}.
\end{proof}

\section{The main results} \label{S:MainProofs}

In this section we prove the theorems stated in the introduction. We will apply Thom-Pontrjagin type arguments to reduce to the case of maps to disks, so first we show quickly how the results of the preceding two sections immediately give our main theorems in the case when we are mapping to $B^1$ or $B^2$.

\begin{proof}[Proof of Theorem~\ref{T:1Existence} for maps to $B^1$]
 This is exactly Theorem~\ref{T:MorseExistence}, with ordered implying fiber-connected when $m > 2$.
\end{proof}

\begin{proof}[Proof of Theorem~\ref{T:1Uniqueness} for maps to $B^1$]
 Given two indefinite Morse functions $g_0,g_1 \co M^m \to B^1$, we need to construct an indefinite generic homotopy $g_t$ connecting them, and if $m \geq 3$ and $g_0$ and $g_1$ are fiber-connected we need to arrange that $g_t$ is fiber-connected. Apply Lemma~\ref{L:Connected2Ordered} to homotope $g_0$ and $g_1$ to be ordered, without destroying fiber-connectedness if they are given as fiber-connected. Then apply Theorem~\ref{T:MorseDeformationExistence}.
\end{proof}

\begin{proof}[Proof of Theorem~\ref{T:Existence} for maps to $B^2$]
 Here we are given an indefinite Morse function $g \co \partial X^n \to S^1$ and we need to construct an indefinite Morse $2$--function $G \co X \to B^2$ with $G|_{\partial X} = g$. When $n \geq 4$ and $g$ is fiber-connected we need to arrange that $G$ is fiber-connected. Arbitrarily identify $S^1$ with $\partial I^2$ so that $g$ has no critical values in $I \times \{0,1\} \subset \partial  I^2$, and then use this identification to realize $X$ as a cobordism from $M_0 = g^{-1}(\{0\} \times I)$ to $M_1 = g^{-1}(\{1\} \times I)$. Apply Lemma~\ref{L:Connected2Ordered} to begin the construction of $G$ on a collar neighborhood of $\partial X$ so as to reduce to the case where $g$ is ordered on $M_0$ and $M_1$, without destroying fiber-connectedness if we had it to begin with. This gives us $\zeta_0$ and $\zeta_1$ as in the preceding section, in preparation for Theorem~\ref{T:SquareExistence}. Also apply Theorem~\ref{T:MorseExistence} to $X$ to produce the desired indefinite, ordered Morse function $\tau$. Finally 
apply Theorem~\ref{T:SquareExistence} to produce $G$.
\end{proof}

\begin{proof}[Proof of Theorem~\ref{T:Uniqueness} for maps to $B^2$]
Now we are given two indefinite Morse $2$--functions $G_0, G_1 \co X^n \to B^2$ which agree on $\partial X$, and, assuming that $n \geq 4$, we need to construct an indefinite generic homotopy $G_s$ connecting them. When $G_0$ and $G_1$ are fiber-connected we need to arrange that $G_s$ is fiber-connected. In order to reduce to Theorem~\ref{T:SquareUniqueness}, we first proceed as in the previous proof, identifying $B^2$ with $I^2$ and applying Lemma~\ref{L:Connected2Ordered} to get $\zeta_0$ and $\zeta_1$ ordered as in Theorem~\ref{T:SquareUniqueness}. Now we need to arrange that $\tau_0 = t \circ G_0$ and $\tau_1 = t \circ G_1$ are ordered. We do this with Lemma~\ref{L:HorizCritSwitch}, switching critical values so as to order them by index, making sure to always move critical points of index $\leq n/2$ to the left and critical points of index $> n/2$ to the right. Now we can apply Theorem~\ref{T:SquareUniqueness}.
\end{proof}

Now we can consider more interesting topology in our target spaces. For existence of indefinite, fiber-connected $S^1$--valued Morse functions, for example, we need to arrange for one fiber to be connected, and then we can cut open along that fiber and reduce to the $B^1$--valued case. These proofs are essentially a refinement of the Thom-Pontrjagin construction, in which maps to $S^n$ and homotopies between them are constructed and characterized by specifying the fiber over a point. However, we need to be more careful than in the basic Thom-Pontrjagin construction because of the fact that we need to arrange for connected fibers, and because we want to construct homotopies which connect fibers without introducing extraneous components along the way. 

In the traditional Thom-Pontrjagin construction for maps from $S^n$ to $S^k$, a framed codimension--$k$ submanifold of $S^n$ determines a map by sending the submanifold to the north pole and the framed normal $k$--disk bundle to $S^k$, with boundary and the complement of the normal disk bundles all mapping to the south pole. A framed cobordism in $I \times S^n$ between two such framed submanifolds determines a homotopy between their maps in the same way. 

In our setting, the framed cobordisms are easy to see but we require more control on the associated homotopies than we would get by applying the standard Thom-Pontrjagin arguments. In particular, if we want to construct a homotopy which connects two components of a disconnected fiber, we will need to choose an arc connecting the components and then construct a homotopy with support in a neighborhood of the arc, such that, during the homotopy, extraneous components are not introduced.

We begin with the existence and uniqueness results for indefinite, fiber-connected, Morse functions. As mentioned earlier, we already have these results when mapping to $I$, so we focus now on maps to $S^1$. 

\begin{proposition} \label{P:Arcs}
 Given a Morse function $g: M^m \to S^1$ which is surjective on $\pi_1$, with $M$ connected, and given a regular value $q \in S^1$, if $g^{-1}(q)$ is disconnected then we can find two components of $g^{-1}(q)$ connected by an arc $a \subset M$ which intersects $g^{-1}(q)$ only at its endpoints and which is disjoint from either $g^{-1}[q-\epsilon,q)$ or $g^{-1}(q,q+\epsilon]$, for some small $\epsilon>0$.
\end{proposition}

\begin{proof}
 Begin with an arc connecting any two components of $g^{-1}(q)$, meeting $g^{-1}(q)$ transversely, intersecting $g^{-1}(q)$ with opposite signs at its endpoints, and projecting to a homotopically trivial loop in $S^1$. (Here we use $\pi_1$--surjectivity.) If this arc hits $g^{-1}(q)$ anywhere in its interior, either take an innermost arc intersecting distinct components with opposite signs at its endpoints, or, if an innermost arc starts and ends in the same component, shortcut in the obvious way, reducing the number of intersections.
\end{proof}

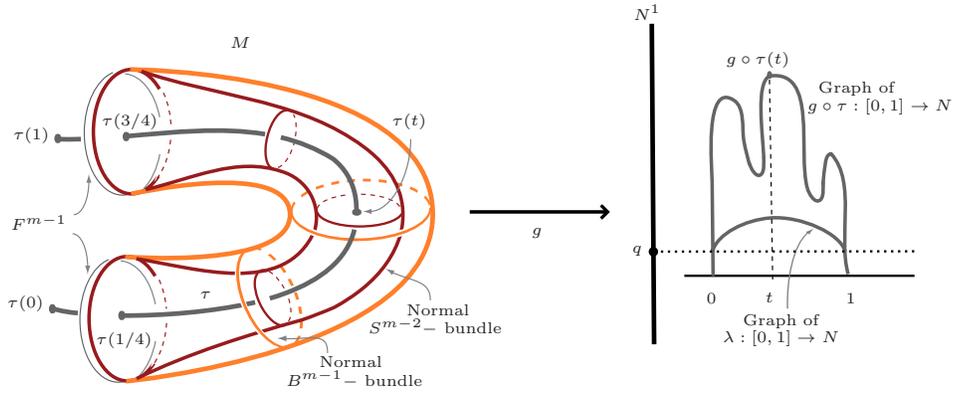
\begin{figure}
\labellist
%\small
\tiny
\hair 2pt
%%%%%%%%%%%%%%%%% M and N
\pinlabel $M$ at 33 61
\pinlabel $N^1$ [b] at 103 64
%%%%%%%%%%%%%%%%%%% taus
\pinlabel $\tau (1)$ [r] at 1 45
\pinlabel $\tau (3/4)$ [b] at 14 45
\pinlabel $\tau (0)$ [r] at 0 16
\pinlabel $\tau (1/4)$ [t] at 13 12
\pinlabel $\tau$ [b] at 27 16
\pinlabel $\tau (t)$ [b] at 62 45
%%%%%%%%%%%%%%%%
\pinlabel $g$ [t] at 84 30
\pinlabel $F^{m-1}$ [r] at 4 30
%%%%%%%%%%%%%%%%%% long labels
% \pinlabel $\textnormal{Normal bundle}$ [t] at 33 0
% \pinlabel $\textnormal{to sphere bundle}$ [t] at 33 -3
%%%%%%
\pinlabel $\textnormal{Normal}$ at 52 6
\pinlabel $\textnormal{ $B^{m-1}-$ bundle}$ at 52 3
%%%%%%
\pinlabel $\textnormal{Normal}$ at 67 15
\pinlabel $\textnormal{$S^{m-2}-$ bundle }$ at 67 12
%%%%%%%%%%%%%%%%%%%%%%%%%%%% on the graph
\pinlabel $0$ at 114 17
\pinlabel $1$ at 138 17
\pinlabel $t$ at 124 17
\pinlabel $\textnormal{Graph of}$ at 139 53
\pinlabel $g\circ\tau:[0,1]\rightarrow N$ at 143 50
\pinlabel $\textnormal{Graph of}$ at 126 13
\pinlabel $\lambda:[0,1]\rightarrow N$ at 126 10
\pinlabel $g\circ\tau(t)$ at 122 58
\pinlabel $q$ [r] at 103 25
\endlabellist
\centering
\includegraphics[scale=2.2]{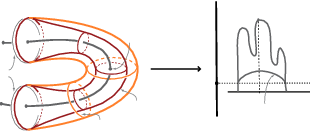}
\vspace*{2pt}
\caption{Connecting components of $g^{-1}(q)$.}
\label{F:Tubing}
\end{figure}
We will use such arcs to construct homotopies that connect components of $g^{-1}(q)$. For the basic idea,
referring to Figure~\ref{F:Tubing}, consider a given Morse function $g : M^m \to N^1$, with regular value $q$ and suppose the fiber $F=g^{-1}(q)$ is disconnected. Choose an arc $\tau : [0,1] \to M$ transverse to $F$, intersecting $F$ only at $\tau(1/4)$ and $\tau(3/4)$, connecting two components of $F$. (If $\tau$ intersects $F$ at more points, choose a shorter arc.) Note that $\tau$ has a normal $B^{m-1}$--bundle with boundary $S^{m-2}$--bundle. We wish to construct a homotopy $g_t$ starting at $g_0 = g$, supported inside this $B^{m-1}$--bundle, such that, watching the level sets $g_t^{-1}(q)$ as $t$ ranges from $0$ to $1$, we see two fingers stick out from $F$ following $\tau$, meeting at $\tau(1/2)$, and merging to achieve a surgery on the $0$--sphere $\{\tau(1/4),\tau(3/4)\}$. In other words, we connect the two components of $F$ with a tube along $\tau$, without introducing any extraneous components during the homotopy.

To visualize the homotopy $g_t$, choose coordinates for the normal $B^{m-1}$--bundle to the arc $\tau([0,1])$ so that $g$ is projection on the $t$ coordinate, that is, each normal $B^{m-1}$ maps to a point of $N$. Then $g_t$ should equal $g_0$ on the boundary $S^{m-2}$ of each normal $B^{m-1}$ and should push the interior of this $B^{m-1}$ down $N$ until the center of $B^{m-1}$ has gone past $q$. If the center is below $q$, then a normal $S^{m-2}$ will map to $q$, and these spheres will make up the cylinder $I \times S^{m-2}$ which connects the two components of the fiber.

In order not to create extra components during the homotopy $g_t$, the interiors of the normal $B^{m-1}$'s need to be pushed across $q$ in turn, starting monotonically at $t=1/4$ and $t=3/4$, and finally at $t=1/2$. To organize this, choose a parabolic arc $\lambda$ as drawn in Figure~\ref{F:Tubing} and first push all the centers of the normal $B^{m-1}$'s, i.e. $\tau(t)$, down to $\lambda$, and then push the centers down as though $\lambda$ was being translated down.

The details are as follows:

\begin{lemma} \label{L:PathHomotopyToMorseHomotopy}
Suppose that $g : M^m \to N^1$ is a Morse function and that $\tau : I \to M$ is a smooth embedded path avoiding the critical points of $g$ such that $\gamma = g \circ \tau$ is homotopically trivial (rel. endpoints) in $N$. Let $q \in N$ be a regular value for both $g : M \to N$ and $\gamma : [0,1] \to N$, with $\gamma^{-1}(q) = \{1/4,3/4\}$. Suppose that $\Gamma : I \times I \to N$ is a homotopy, with $\Gamma(0,x) = \gamma(x)$, such that $q$ is regular for $\Gamma$ and such that $\Gamma^{-1}(q)$ is an arc $a \subset I \times I$ from $(0,1/4)$ to $(0,3/4)$ and such that projection onto the first factor of $I \times I$ restricts to $a$ as a Morse function with a single index $1$ critical point at $(1/2,1/2)$. Then, inside any neighborhood of $\tau(I)$ there is an embedding $\overline{\tau} : B^{m-1} \times I \to M$, with $\overline{\tau}(\mathbf{0},x) = \tau(x)$ such that, letting $\overline{\gamma} = g \circ \overline{\tau} : B^{m-1} \times I \to N$, we have a homotopy $\overline{\Gamma} : I \times B^{m-1} \times I \to N$ satisfying the following properties:
\begin{enumerate}
 \item For any $(0,\mathbf{y},x) \in I \times B^{m-1} \times I$, $\overline{\Gamma}(0,\mathbf{y},x) = \overline{\gamma}(\mathbf{y},x)$.
 \item For any $(t,\mathbf{0},x) \in I \times B^{m-1} \times I$, $\overline{\Gamma}(t,\mathbf{0},x) = \Gamma(t,x)$.
 \item For any $(t,\mathbf{y},x) \in I \times ((S^{m-2} \times I) \cup (B^{m-1} \times \{0,1\}))$, $\overline{\Gamma}(t,\mathbf{y},x) = \overline{\Gamma}(0,\mathbf{y},x) = \overline{\gamma}(\mathbf{y},x)$.
 \item The point $q \in N$ is a regular value for $\overline{\Gamma}$.
 \item $\overline{\Gamma}^{-1}(q)$ is an $m$--dimensional submanifold of $I \times B^{m-1} \times I$ on which the projection onto the first factor of $I \times B^{m-1} \times I$ has a single Morse critical point of index $1$ at $(1/2,\mathbf{0},1/2)$. This is illustrated in Figure~\ref{F:PathHomotopyToMorseHomotopy}.
 \item This submanifold $\overline{\Gamma}^{-1}(q)$ intersects the boundary of $I \times B^{m-1} \times I$ as  
\[ (\{0\} \times B^{m-1} \times I) \cap \overline{\Gamma}^{-1}(q) = \{0\} \times B^{m-1} \times \{1/4,3/4\} \]
 and
\[  (I \times S^{m-2} \times I) \cap \overline{\Gamma}^{-1}(q) =  I \times S^{m-2} \times \{1/4,3/4\}. \]
 \item In particular, this implies that $(\{1\} \times B^{m-1} \times I) \cap \overline{\Gamma}^{-1}(q)$ is properly embedded in $\{1\} \times B^{m-1} \times I$ and diffeomorphic to $S^{m-2} \times [1/4,3/4]$, only meeting the boundary of $\{1\} \times B^{m-1} \times I$ at $S^{m-2} \times \{1/4,3/4\}$.
\end{enumerate}
\end{lemma}
\begin{figure}[ht!]
\labellist
\small\hair 2pt
\pinlabel $B^{m-1}$ [l] at 179 25
\pinlabel $I$ [t] at 81 3
\pinlabel $I$ [r] at 3 61
\pinlabel $\overline{\Gamma
}^{-1}(q)$ [t] at 98 60
\endlabellist
\begin{center}
\includegraphics{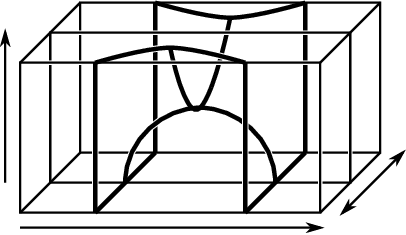}
\caption{The model for $\overline{\Gamma}^{-1}(q)$ in the statement of Lemma~\ref{L:PathHomotopyToMorseHomotopy}.}
\label{F:PathHomotopyToMorseHomotopy}
\end{center}
\end{figure}

\begin{proof}
Let $\overline{\tau}$ be a parametrization of a small neighborhood of $\tau([0,1])$ such that $\overline{\gamma}^{-1}(q) = B^{m-1} \times \{1/4,3/4\}$. Let $H$ be a model saddle hypersurface in $I \times B^{m-1} \times I$ satisfying the behavior given in the statement of the lemma for $\overline{\Gamma}^{-1}(q)$. Then, the constraints given for $\overline{\Gamma}$, where we ask that $\overline{\Gamma}^{-1}(q) = H$, completely determine $\overline{\Gamma}$ on the following subset $C$ of $I \times B^{m-1} \times I$:
\[ C = (\{0\} \times B^{m-1} \times I) \cup (I \times S^{m-2} \times I) \cup H \cup (I \times \{\mathbf{0}\} \times I) \cup (I \times B^{m-1} \times \{0,1\}). \]
There is a natural extension of this prescribed behavior of $\overline{\Gamma}$ on $C$ to a smooth function on a regular neighborhood $\nu$ of $C$ so that $q$ becomes a regular value and such that $\overline{\Gamma}(\partial \nu)$ avoids $q$. Then note that $I \times B^{m-1} \times I$ deformation retracts onto $C$, so there is a continuous extenstion of $\overline{\Gamma}$ to all of $I \times B^{m-1} \times I$, which can be smoothed and made generic without changing the behaviour on $\nu$, and without introducing any components of $\overline{\Gamma}^{-1}(q)$ outside $\nu$.

\end{proof}

\begin{proof}[Existence: Proof of Theorem~\ref{T:1Existence}]
We have already proved this for the case of maps to $B^1$. Now we consider the case of maps to $S^1$. We are given a homotopically nontrivial map $g' : M \to S^1$ and we want to produce an indefinite Morse function $g : M \to S^1$ which is homotopic to $g'$. To simplify notation, we will simply refer to all our maps to $S^1$ as $g$, modifying $g$ in stages by homotopies. So we begin with $g=g'$. Because $g$ is homotopically nontrivial, we can lift to a finite cover of $S^1$ so that the induced map on $\pi_1$ is surjective, and thus reduce to the case $g_*(\pi_1(M))=\pi_1(S^1)$. Fiber-connectedness is not preserved under post-composition with a covering space map, but we are not asking for fiber-connectedness unless $g$ is $\pi_1$--surjective to begin with.

So now assume $g$ is $\pi_1$--surjective and Morse. Identify $S^1$ with $\R / \Z$ so that $0$ is a regular value. We want to homotope $g$ to arrange that $g^{-1}(0)$ is connected. (The new $g$ should still be Morse and $0$ should still be regular.) Then our theorem will reduce to the cobordism case, Theorem~\ref{T:MorseExistence}, by cutting $M$ open along $g^{-1}(0)$. 

Choose some $\epsilon>0$ such that there are no critical values of $g$ in $[-\epsilon,0]$.
Choose two components of $g^{-1}(0)$ as in Proposition~\ref{P:Arcs}, and, extending the arc produced in the proposition, choose a path $\tau : [0,1] \to M$ connecting them in $M$, but starting and ending in $g^{-1}(-\epsilon)$ and passing through the two components at $\tau(1/4)$ and $\tau(3/4)$, respectively, such that $\gamma = g \circ \tau$ is homotopic to $0 \in \pi_1(S^1,-\epsilon)$. We can then choose a homotopy $\Gamma : I \times I \to N$ as in the hypotheses of Lemma~\ref{L:PathHomotopyToMorseHomotopy}, with $0$ being the regular value $q$ in the lemma. Then the embedding $\overline{\tau}$ and the homotopy $\overline{\Gamma}$ gives us a homotopy $g_t$ defined as the identity outside the image of $\overline{\tau}$ and as $g_t(p) = \overline{\Gamma}(t,\overline{\tau}^{-1}(p))$ for $p$ in the image of $\overline{\tau}$. The effect on $g^{-1}(0)$ is to replace two $B^{m-1}$ neighborhoods in $g^{-1}(0)$ of the two points $\tau(1/4)$ and $\tau(3/4)$ with a tube diffeomorphic to $S^{m-2} \times I$, and thus 
the two components get connected. Repeating, we connect all the components of $g^{-1}(0)$.

Note that the final homotopy is a concatenation of homotopies each supported in a neighborhood of an arc in $M$. Thus we could just as well have performed all the homotopies at the same time, as long as we can choose all the relevant arcs to be disjoint from the beginning, which we can do if $m \geq 3$. This point of view is more useful in the next proof.

\end{proof}

\begin{proof}[Uniqueness: Proof of Theorem~\ref{T:1Uniqueness}]
 Again, we have already proved this in the case of homotopies between maps to $B^1$, so now we consider the case of homotopies between maps to $S^1$. We are given two homotopic indefinite Morse functions $g_0, g_1 : M \to S^1$ and we need to produce an indefinite homotopy $g_t$ between them; when $g_0$ and $g_1$ are fiber-connected then $g_t$ should also be fiber connected for all $t$. Throughout the rest of this proof suppose we are trying to prove both indefiniteness and fiber-connectedness. Assuming that $g_0$ and $g_1$ are indefinite and fiber-connected implies that both maps are $\pi_1$--surjective. The proof without asking for fiber-connectedness follows by lifting to an appropriate cover.

Again, identify $S^1$ with $\R / \Z$ and assume that $0$ and $1/2$ are regular values of both $g_0$ and $g_1$. We first choose any generic homotopy $g_t$, and we will show how to modify $g_t$ so as to arrange the following connectedness of level sets property: For a sequence of values $0=t_0 < t_1 < \ldots < t_{2n}=1$, we have that $g_t^{-1}(0)$ is connected and regular on $[t_0,t_1] \cup [t_2,t_3] \cup \ldots \cup [t_{2n-2},t_{2n-1}]$ and that $g_t^{-1}(1/2)$ is connected and regular on $[t_1,t_2] \cup [t_3,t_4] \cup \ldots \cup [t_{2n-1},t_{2n}]$.

As in the proof of Theorem~\ref{T:1Existence} we can make either of the level sets $g_t^{-1}(0)$ or $g_t^{-1}(1/2)$ connected for a fixed $t$, as long as $0$ or $1/2$ is a regular value. The construction depends on a choice of arcs in $M$ based at the level set in question, and then produces a homotopy supported in a neighborhood of those arcs and loops. Even if one of the level sets contains a single Morse critical point, we can keep the arcs and loops away from that critical point, either connecting the entire singular level set if the critical point is indefinite or connecting all the components other than the single point component when the critical point is a minimum or maximum. 

Now note that the arcs and the guiding homotopies in $S^1$ which are used for a fixed value $t=t_0$ will work for all $t$ in some short interval $[t_0-\epsilon,t_0+\epsilon]$, modulo modifying the arcs by small ambient isotopies near their endpoints. Also note that, in the middle of one of these homotopies, we do not introduce new components of the level set in question. This is because, in Lemma~\ref{L:PathHomotopyToMorseHomotopy}, $\overline{\Gamma}^{-1}(q)$ has a single critical point of index $1$ with respect to projection onto the first factor of $I \times B^{m-1} \times I$, so for each $t$, $(\{t\} \times B^{m-1} \times I) \cap \overline{\Gamma}^{-1}(q)$ either has two components (for $t < 1/2$) or one component (for $t \geq 1/2$), but never more.

Thus, for a fixed level set $g_t^{-1}(z_0)$, as long as $z_0$ is regular or a Morse critical value for each $t$, we can cover $[0,1]$ by a finite collection of such short 
intervals, using disjoint arcs and loops where the intervals overlap (recall that $m \geq 3$, where $m$ is the dimension of $M$), and use bump functions to patch together the homotopies over the whole of $[0,1]$. The result will be that $g_t^{-1}(z_0)$ is connected for all $t$ outside a short interval around each time $t_*$ at which $g_t^{-1}(z_0)$ contains a definite critical point. In these short intervals, $g_t^{-1}(z_0)$ will be connected, say, for all $t< t_*$, will be the disjoint union of a connected manifold and a point for $t=t_*$, will have an isolated sphere component for $t$ slightly larger than $t_*$, and then that component will get connected back to the rest of the level set immediately thereafter. (The phenomenon described in the preceding sentence could also occur with time reversed, of course.)

Do this once for $g_t^{-1}(0)$. Now do this for $g_t^{-1}(1/2)$, noting that generically $1/2$ and $0$ will never be critical values at the same time, and also noting that the arcs used to connect components of $g_t^{-1}(1/2)$ can be made to avoid $g_t^{-1}(0)$ because $g_t^{-1}(0)$ is now connected or has a single isolated sphere component that dead-ends immediately above or below $0$. Because these arcs are disjoint from $g_t^{-1}(0)$, the homotopy used to connect components of $g_t^{-1}(1/2)$ does not destroy the connectedness properties of $g_t^{-1}(0)$. We have also assumed here that non-Morse singularities, i.e. births and deaths, never occur at $0$ or $1/2$.

\textbf{The zig-zag argument:} We have now achieved the connectedness property advertised: For $0=t_0 < t_1 < \ldots < t_{2n}=1$, we have $g_t^{-1}(0)$ is connected and regular on $[t_0,t_1] \cup [t_2,t_3] \cup \ldots \cup [t_{2n-2},t_{2n-1}]$ and $g_t^{-1}(1/2)$ is connected and regular on $[t_1,t_2] \cup [t_3,t_4] \cup \ldots \cup [t_{2n-1},t_{2n}]$, as in Figure~\ref{F:ZigZag}.
If $g_0$ and $g_1$ agreed on, say, $g^{-1}_0(0) = g^{-1}_1(0)$, then we could cut open $M$ along $L$ and obtain two indefinite Morse functions to $I$, thus reducing, with the help of Lemma~\ref{L:Connected2Ordered}, to the already proven Theorem~\ref{T:MorseDeformationExistence}. The goal now is to arrange for this to be true for $g_{t_i}$ and $g_{t_{i+1}}$ for $0 = t_0 < t_1 < \ldots < t_{2n} = 1$.

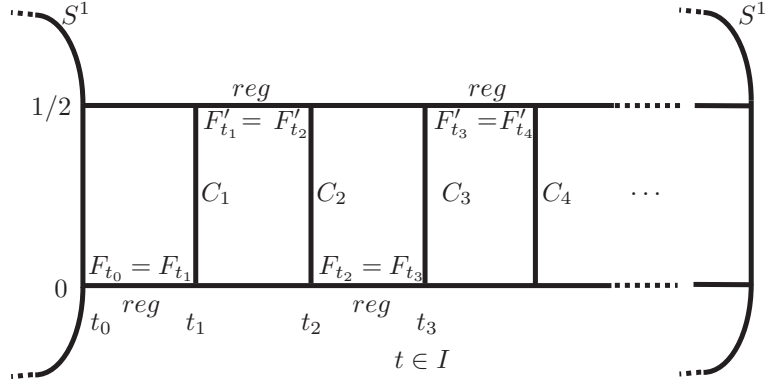
\begin{figure}
\labellist
\small\hair 2pt
\pinlabel $t_0$ at 18 11
\pinlabel $t_1$ at 37 11
\pinlabel $t_2$ at 60 11
\pinlabel $t_3$ at 83 11
%%%%%%%%%%%%%%%% The C's
\pinlabel $C_1$ at 41 37
\pinlabel $C_2$ at 64 37
\pinlabel $C_3$ at 89 37
\pinlabel $C_4$ at 109 37
%%%%%%%%%%%%%%%%%%%%%% The fibers
\pinlabel $F_{t_0}$ at 19 22
\pinlabel $F_{t_1}$ at 33 22
\pinlabel $F'_{t_1}$ at 42 51
\pinlabel $F'_{t_2}$ at 56 51
\pinlabel $F_{t_2}$ at 65 22
\pinlabel $F_{t_3}$ at 79 22
\pinlabel $F'_{t_3}$ at 88 51
\pinlabel $F'_{t_4}$ at 101 51
%%%%%%%%%%%%%%%%%% Equal sign
\pinlabel $=$ at 26 22
\pinlabel $=$ at 48 51
\pinlabel $=$ at 72 22
\pinlabel $=$ at 95 51
%%%%%%%%%%%%%%%%%% reg
\pinlabel $reg$ at 26 14
\pinlabel $reg$ at 48 57
\pinlabel $reg$ at 72 14
\pinlabel $reg$ at 95 57
%%%%%%%%%%%%%%% dots
\pinlabel $\dots$ at 127 37
%%%%%%%%%%%%%%%%%%%%%% the interval
\pinlabel $t\in I$ at 82 4
%%%%%%%%%%%%%%%%%%%%%%%%%% the circle
\pinlabel $S^1$ at 13 73
\pinlabel $S^1$ at 148 73
%%%%%%%%%%%%%%%%%%%%%% the zero and the half
\pinlabel $0$ at 10 18
\pinlabel $1/2$ at 8 54
\endlabellist
\centering
\includegraphics[scale=1.90]{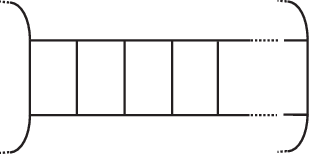}
\caption{The zig-zag argument.}
\label{F:ZigZag}
\end{figure}

Let $F_{t_i} = g_{t_i}^{-1}(0)$ and let $F_{t_i}' = g_{t_i}^{-1}(1/2)$. Note that $F_{t_0}$ is isotopic to $F_{t_1}$ because $0$ is regular for each $g_t$, $t \in [t_0,t_1]$. Similarly, $F_{t_{2i}}$ is isotopic to $F_{t_{2i+1}}$. We can assume they are, in fact, equal. The same works for $F'_{t_{2i+1}}$ and $F'_{t_{2i+2}}$.

The vertical cobordism $C_i = g_{t_i}^{-1}[0,1/2]$ between $F_{t_i}$ and $F'_{t_i}$ is connected because the top, $g_{t_i}^{-1}(1/2)$, is connected, as is the bottom, and because $g_{t_i} : M \to S^1$ is surjective on $\pi_1$. (For, given $p,q \in C_i$, there is an arc joining them in $M$ which intersects $F_{t_i}$ and $F'_{t_i}$ algebraically $0$ each, and now connectedness of $F_{t_i}$ and $F'_{t_i}$ allows the arc to be replaced by an arc in $C_i$.)

Now for each $i=1, \ldots, 2n-1$ we construct, by Theorem~\ref{T:MorseExistence}, intermediate functions $h_{t_i}: M \to S^1$ homotopic to $g_{t_i}$ such that $h_{t_i}^{-1}[0,1/2] = C_i$ and $h_{t_i}$ is indefinite and fiber-connected.

Then $g_0$ and $h_{t_1}$ agree on the level set $F_{t_0} = g_0^{-1}(0)$ and are homotopic rel. this level set and thus are joined by an indefinite fiber-connected homotopy (using Lemma~\ref{L:Connected2Ordered} and Theorem~\ref{T:MorseDeformationExistence}). Also $h_{t_1}$ and $h_{t_2}$ agree on the level set $F'_{t_2} = h_{t_2}^{-1}(1/2)$ so they are also joined by an indefinite fiber-connected homotopy. Repeat this up and down zig-zag construction up to $h_{t_{2n-1}}$, and then finally join $h_{t_{2n-1}}$ to $g_{t_{2n}}=g_1$. This ends the proof.
\end{proof}

We have proved Theorem~\ref{T:Existence} when $\Sigma^2$ is $B^2$, and we could
prove the case $\Sigma = S^2$ by homotoping $G$ so that a fiber $F$ is
connected and then removing a $B^2$-bundle neighborhood of $F$ to
reduce to the case $\Sigma = B^2$.  However when $\Sigma$ is closed
and not $S^2$, there can be problems keeping $G$ an epimorphism on
$\pi_1$ after removing $F$.  To resolve this issue, we need some
constructions which use the following lemma. This is essentially the Morse $2$--function version of Lemma~\ref{L:PathHomotopyToMorseHomotopy}. However, when written as a direct analog of Lemma~\ref{L:PathHomotopyToMorseHomotopy}, the statement of this lemma becomes unwieldy.

\begin{lemma} \label{L:ConnectingM} 
Let $G:X \to \Sigma$ be a Morse $2$-function. Let $a$ be a non-separating simple closed
curve or properly embedded arc in $\Sigma$ which meets all folds
transversely so that $G^{-1}(a) = M^{n-1}$ is a smooth manifold.  Let
$\beta : [-1,2] \to X$ be a smoothly embedded arc in $X$ which
intersects $M$ transversely in two points, $\beta(1)$ and $\beta(0)$,
of opposite sign.  Let the arc $b:[-1,2] \to \Sigma$ be defined as $b
= G \circ \beta$. Now suppose that $b|_{[0,1]}$ is homotopic by $h_s, s \in [0,1]$ to an
arc joining $G(\beta(1))$ and $G(\beta(0))$, and call this arc $a'$.

Then there exists a homotopy $G_s, s \in [0,2]$ satisfying:

\begin{enumerate}

\item $G_0 = G$,

\item $G_s = G$ outside of an arbitrarily thin tubular neighborhood of $\beta ([-1,2])$,

\item If the homotopy $h_s$ of $b$ never takes $b$ across $a$ except
  at time $s=1$, then $G_2^{-1}(a)$ equals $M$ surgered along the
  $0$-sphere, $\beta(1) \cup \beta(0)$, thus connecting the components
  containing $\beta(1)$ and $\beta(0)$ by a tube $[0,1] \times
  S^{n-2}$.

\item If $h_s$ does take $b$ across $a$, then $G_2^{-1}(a)$ is the result of $0$--surgery on $M$ as described above together with the possible addition of new closed components.

\end{enumerate}

\end{lemma}

\begin{proof} 
The argument is a version of the well known Thom-Pontrjagin
method for calculating $\pi_n (S^2)$ by simplifying the preimage of the
north pole of $S^2$ by first connecting its components.

First it is easy to alter $h_s$ so that the end $h_1$ of the homotopy 
taking $b([0,1])$ to $a'$ is a diffeomorphism, 

$M$ in $X$ has a trivial normal line bundle which locally is
compatible with the normal lines to $a$ in $\Sigma$, meaning that the
lines map to lines.  Then, using this product structure around
$\beta(1)$ and $\beta(0)$ and $a \subset \Sigma$, it is easy to extend
the homotopy $h_s$ to a homotopy of the full arc $b: [-1,2] \to \Sigma$, with the homotopy time parameter $s$ extended from $[0,1]$ to $[0,2]$, so that points
under $h_s, s \in [0,2]$ follow the curved lines beginning at $b$ and ending 
up on a parallel copy of $a'$ called $a''$ as drawn in Figure~\ref{F:ConnectingM}.

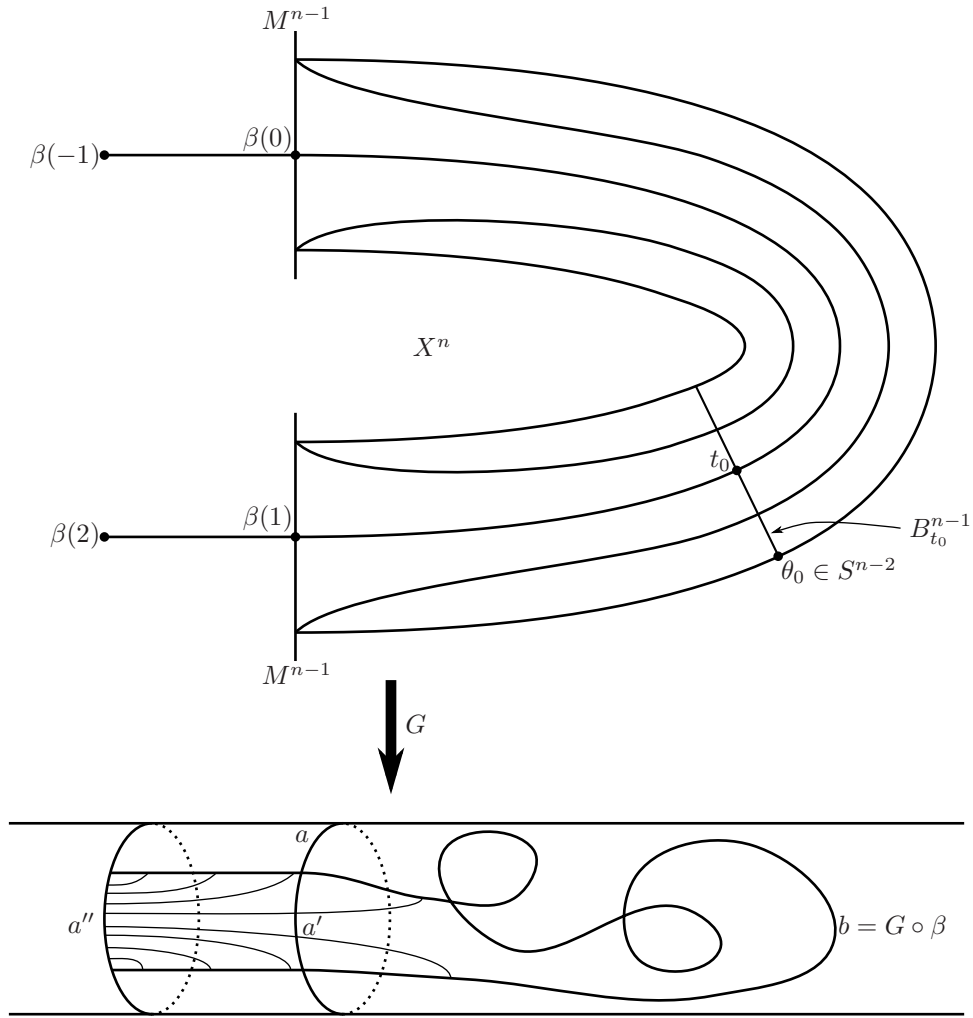
\begin{figure}[ht!]
\labellist
\small\hair 2pt
\pinlabel $a''$ [r] at 35 35
\pinlabel $a'$ [l] at 109 34
\pinlabel $a$ [br] at 115 64
\pinlabel {$b = G \circ \beta$} [l] at 311 33
\pinlabel $G$ [l] at 148 110
\pinlabel $X^n$ at 160 252
\pinlabel $\beta(-1)$ [r] at 37 324
\pinlabel $\beta(0)$ [br] at 109 324
\pinlabel $\beta(2)$ [r] at 37 181
\pinlabel $\beta(1)$ [br] at 109 181
\pinlabel $M^{n-1}$ [t] at 109 135
\pinlabel $M^{n-1}$ [b] at 109 370
\pinlabel $t_0$ [br] at 274 204
\pinlabel {$\theta_0 \in S^{n-2}$} [tl] at 290 174
\pinlabel $B_{t_0}^{n-1}$ [l] at 338 183
%\pinlabel 
\endlabellist
\begin{center}
\includegraphics{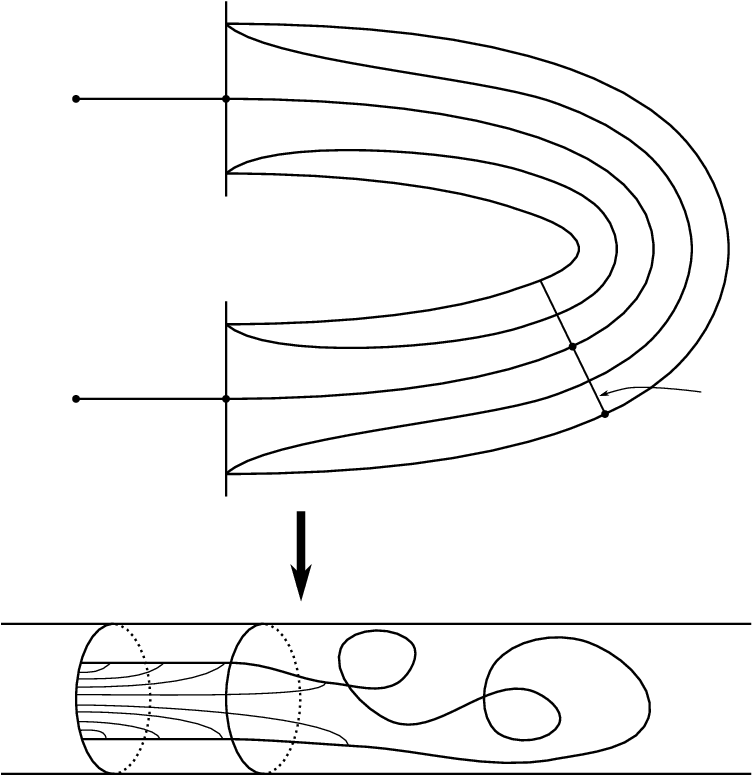}
\caption{Connecting components of $M$.} \label{F:ConnectingM}
\end{center}
\end{figure}

The arc $\beta$ has a tubular neighborhood, $\beta \times B^{n-1}$,
and each $B^{n-1}$ has {\it polar} coordinates $(t,r,\theta)$ where $t
\in [-1,2], r \in [0,1], \theta \in S^{n-2}$ (see Figure~\ref{F:ConnectingM}).  A ray
$(t_0, r, \theta_0), r \in [0,1]$ determined by the pair $(t_0,
\theta_0)$ is mapped by $G$ to a path $\rho=\rho_{t_0, \theta_0} \subset
\Sigma$. The endpoint of $\rho$ at $b$ is moved by the homotopy $h_s,
s \in [0,2]$, along a path ending on $a''$; extend $\rho$ to this
longer path $\bar{\rho}$.

The homotopy of $G$ now maps the ray $(t_0, r, \theta_0)$ to $\rho$
and as time progresses stretches the ray over more and more of
$\bar{\rho}$ until it is onto.  It is easy to see that this homotopy
is constant on $\beta$'s normal sphere bundle, and it is also constant
on the {\it top} and {\it bottom} of this cylinder because these points in
$G^{-1}(a'')$ are also not moved.  At the end $G_2$ is a Morse $2$-function on a
neighborhood of $M$ (because we chose endpoints of $b$ disjoint from folds and then $M$ and nearby copies change only by surgery on small neighborhoods of the end points of $\beta$) but may need to be perturbed to make it a Morse
$2$-function again elsewhere.

The difference between the last two parts of the Lemma is fairly
evident. The surgery statement follows because neighborhoods of
$\beta(1)$ and $\beta(0)$ are pushed off of $a'$ and replaced by a
horn shaped cylinder which is mapped to $a'$, as can be seen in Figure~\ref{F:ConnectingM}.  And if the original homotopy $h_s$ does move points across $a'$, then
closed components can be added to $M$.

\end{proof}

\begin{proof}[Existence: Proof of Theorem~\ref{T:Existence}]
 Here we are given a compact, connected, oriented $n$--manifold $X$
 and a $2$--manifold $\Sigma$ and an indefinite, surjective Morse
 function $g : \partial X \to \partial \Sigma$ which extends to a map
 $G' : X \to \Sigma$. We wish to homotope $G'$ rel. boundary to an
 indefinite Morse $2$--function $G : X \to \Sigma$, and perhaps also
 arrange fiber-connectedness. Again, we will drop the primes and
 simply refer to all our maps to $\Sigma$ as $G$, constructing or
 modifying $G$ in stages.

The base case is when $\Sigma$ is the disk $B^2$, which we have
already addressed.  We postpone the case $\Sigma=S^2$ briefly until Remark~\ref{R:ConnectingFibers} below, and now we reduce all other cases to the case $\Sigma=B^2$. The given map $G$ can be homotoped to a Morse $2$--function; we
need to make it indefinite if $[G_*(\pi_1(X)):\pi_1(\Sigma)]<\infty$
and fiber-connected if $G_*(\pi_1(X)) = \pi_1(\Sigma)$. If
$[G_*(\pi_1(X)):\pi_1(\Sigma)]<\infty$ we can lift to a finite cover
of $\Sigma$ and reduce to the case that $G$ is $\pi_1$--surjective.

First suppose that $\Sigma$ and $X$ are closed, in which case there is
no boundary condition $g$. We first arrange that a nice basis for
$\pi_1(\Sigma,\sigma_0)$ lifts to $X$ (we only know that the basis is
homotopic to one that lifts). We do this as follows:

Describe $\Sigma$ as a $0$-handle, $2g$ $1$-handles and a $2$-handle in the standard way, with the $1$--handles coming in dual pairs.
Let $a_1, a_2,\ldots, a_{2g}$ be the cores of the $1$-handles and
$\bar{a}_1, \ldots, \bar{a}_{2g}$ be the extensions of these arcs to smooth 
loops by ``coning'' their end points to the core, $\sigma_0$,  of the
$0$-handle in a smooth way.  We can assume that $G^{-1}(\bar{a}_i) = M^{n-1}$ is a manifold.

Focus on $a_1$ and $\bar{a}_1$ and drop the subscript for simplicity.
$\bar{a}$ is homotopic to a loop $\bar{b}$ which lifts to $X$, and the
homotopy fixes $\sigma_0$ and can be made to fix the arc,
$\bar{a} -a$, also.  This gives an arc $b$ with $\partial a = \partial
b$.  Let $\beta$ be the lift of $b$ to $X$.

We can assume that $\beta$ meets $M=G^{-1}(\bar{a})$ transversely at
$\partial \beta$ and that these two intersection points have opposite
signs.  Now apply Lemma~\ref{L:ConnectingM} to homotope $G$ so that $a$, hence
$\bar{a}$, has a lift $\alpha$ to $M \subset X$.  Do the same process for each
$a_i$, noting that the arcs $\beta_i$ need not intersect, nor their
thin neighborhoods.  Also note that all the lifts $\alpha_i$ can contain the basepoint of $x_0 \in X$. 

Thus we have homotoped $G$ so that each $a_i$ and $\bar{a_i}$ have
lifts which we call $\alpha_i$ and $\bar{\alpha_i}$.  Now we want to
apply Theorem 1.1 to an $M_i = G^{-1}(\bar{a_i})$, where we know $G$
is an epimorphism on $\pi_1$, but $M=M_i$ may not yet be connected.

We will make $M=M_1$ connected by again using Lemma~\ref{L:ConnectingM}. It suffices to show how to homotope $G$ so as to connect two points, $p$ and $q$ in $M =
G^{-1}(\bar{a})$, where $p$ belongs to the component of $X$ containing $x_0$, without introducing any new components in the process.  

For this we need some notation.  The two dual curves $\bar{a_1}, \bar{a_2}$, define a punctured torus $T_0$ and its one-point compactification $T$, and there is a projection from $X$ to $\Sigma$ to $T$ to $\bar{a_i}, i = 1,2$, and we name the composition $p_i, i=1,2$.

In $X$, $p$ and $q$ are connected by an arc $\beta$
which intersects $M$ transversely in some points including $p$ and
$q$. We form a loop $\gamma$ in $\Sigma$ by joining the endpoints of $G(\beta)$
by the subarc $a'$ of $a$ which does not contain the basepoint $\sigma_0$.  We want to arrange that $\gamma$ is homotopically trivial in $\Sigma$.

First, it may be that $p_1 (\gamma)$ is not homotopically trivial, so we connect sum $\beta$ in $M$ with multiples of the lift of $\bar{a_1}$, $\alpha_1$, so that it is now homotopically trivial.  Recall that $p$ belongs to the component containing $x_0 \in \alpha_1$  so the connect sum is taken near $x_0$.  And we push the copies of $\alpha_1$ slightly to one side of $M$ so as to avoid unnecessary intersections.

Next consider whether the new $\gamma$ projected by $p_2$ to $\bar{a_2}$ is homotopically trivial.  If not, we connect sum parallel copies of $\alpha_2$ (which lie in $M_2$) to arrange triviality.  Furthermore we choose these parallel copies so that all their projections to $\Sigma$ all lie on the same side of the base point $\sigma_0$. 

Continue in this way with the other $\alpha_i$'s so as to arrange that $\gamma$ is homotopically trivial in $\Sigma$. Note that the other $\alpha_i$'s do not intersect $M_1 = M$.

Next we want to arrange that $\gamma$ not intersect $\bar{a_1}$ except along $a'$.  For this, consider the universal cover of $\Sigma$ in which we see ``parallel'' copies of the lift of $\bar{a_1}$, and also a copy of the lift of $\gamma$ which contains a given lift of the base point, see Figure~\ref{F:UniversalCover}.  Look for the last (rightmost in Figure~\ref{F:UniversalCover}) lift of $\bar{a_1}$ which intersects the lift of $\gamma$.  Pick a subarc $\lambda$ of the lift of $\gamma$ which intersects this lift of $\bar{a_1}$ in its two endpoints, necessarily of opposite sign.  If that subarc $\lambda$ has endpoints in the same component of $M$, then connect them in $M$, changing $\beta$ and $\gamma$ by removing these two points of intersection.  Proceed until we get two points in different components.  Then $\lambda$, projected back to $\Sigma$, is extended to a contractible loop, which we again call $\gamma$, by adding a segment of $a_1$. The subarc $\lambda$ corresponds to a subarc of $\beta$ which intersects $M$ only in its endpoints, with opposite sign. Now apply Lemma~\ref{L:ConnectingM} using this subarc of $\beta$ to connect two distinct components of $M$, without introducing any new components because the loop $\gamma$ now lies to one side of $a_1$. These  steps are iterated to make $M$ connected.

At this point we pause in the proof for a useful remark, needed for the case $\Sigma=S^2$ and for the proof of Theorem~\ref{T:Uniqueness} below.

\begin{remark} \label{R:ConnectingFibers} Notice that if the two points $b(0)$ and $b(1)$ which belong to $M$, actually belong to a single fiber of $G$, then the construction we have just described connects the two components of that fiber.  And if the hypothesis of part 3 of the Lemma holds, namely that the homotopy $h_t$ does not move $b$ across $a$ except at time $t=1$, then no components have been added to the fiber.  Thus an iteration of these steps can be used to make a single fiber connected.

\end{remark}

When $\Sigma=S^2$, this remark shows us how to homotope $G$ so that a regular fiber is connected, and then removing the fiber cross disk reduces to the case $\Sigma=B^2$.

\begin{figure}[ht!]
\labellist
\small\hair 2pt
\pinlabel $\tilde{\sigma_0}$ [r] at 62 38
\pinlabel $\gamma$ [b] at 99 160
\pinlabel $\tilde{\bar{a_1}}$ [r] at 53 168
\pinlabel $\lambda$ [l] at 150 115
\endlabellist
\begin{center}
\includegraphics{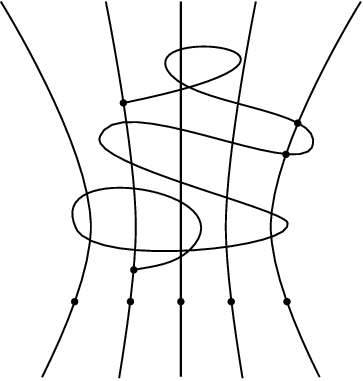}
\caption{Shortening $\gamma$ in the universal cover of $\Sigma$ (the hyperbolic plane, or $\mathbb{R}^2$ if $\Sigma=T^2$).} \label{F:UniversalCover}
\end{center}
\end{figure}
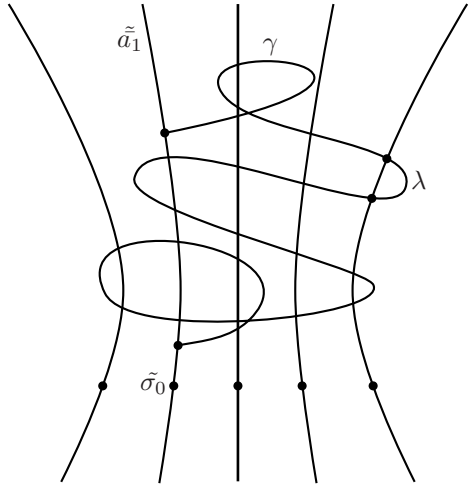

We now return to the general case $\Sigma \neq S^2$. Now $G|_M: M \to \bar{a} = \bar{a}_1 = S^1$ is $\pi_1$--surjective and
$M$ is connected, so by Theorem 1.3, $G|_M$ can be homotoped to be
indefinite with connected fibers.  This homotopy extends to $X$, moving
points only in a thin product neighborhood, $M \times [-1,1]$ in $X$.
Now we remove $M \times (-1,1)$ from $X$ and $\bar{a} \times (-1,1)$
from $\Sigma$ and reduce to the case $\partial X \neq \emptyset$.

Note that after removing $M \times (-1,1)$, $G$ is still surjective on $\pi_1$, because we still have the lifts, $\alpha_i$, of the remaining generators $\bar{a_i}, i\geq 2$ of the fundamental group of $\Sigma \setminus \bar{a_1}$.

The proof of the case in which $\partial \Sigma \neq \emptyset$, assuming only that $\pi_1$ is onto and
that $G|_{\partial X}$ is indefinite with connected fibers, proceeds
nearly identically to the case just done.  Again we describe $\Sigma$
as a $0$-handle and some $1$-handles, and arrange that the loops
defined by the cores of the $1$-handles lift to loops in $X$, in order
to make sure that $G$ remains $\pi_1$--surjective during the remaining steps.

We now apply the technique we just used to make $M$ connected to make the inverse image of a nonseparating arc from boundary to boundary connected, e.g. a co-core of a $1$-handle. The universal cover argument follows in the same way, and we apply Theorem 1.3 in the interval-valued Morse function case rather than the circle-valued case. Then we cut $\Sigma$ along this arc and $X$ along its preimage and proceed inductively. After cutting, $G$ is still $\pi_1$--surjective because the remaining generators of $\pi_1$ still have their lifts.

\end{proof}

\begin{proof}[Uniqueness: Proof of Theorem~\ref{T:Uniqueness}]
Now we are given a compact, connected, oriented $n$--manifold $X$ and a $2$--manifold $\Sigma$ and two indefinite Morse $2$--functions $G_0,G_1 : X \to \Sigma$ which agree on $\partial X$ and are homotopic rel $\partial$. We need to construct an indefinite generic homotopy $G_s$ between them, which is fiber-connected when $G_0$ and $G_1$ are fiber-connected. Begin with an arbitrary generic homotopy $G_s$ and we will modify this in stages, referring to it as $G_s$ before and after each modification. We also reduce to the fiber-connected case by lifting to a cover, as in preceding proofs.

As in the preceding proof, the base case is when $\Sigma = B^2$, which we have addressed. We also postpone the case $\Sigma=S^2$ and now reduce by cutting along closed curves and/or arcs to the case $\Sigma=B^2$, using induction on $-\chi(\Sigma)$.

Choose a nonseparating simple closed curve (if $\Sigma$ is closed) or properly embedded arc (if $\partial \Sigma \neq \emptyset$) $a$ which is transverse to the folds of both $G_0$ and $G_1$, and let $a'$ be a parallel copy of $a$ with the same transversality property. Figure~\ref{F:ZigZag2} illustrates the cases where $a$ and $a'$ are arcs, cobounding a rectangle $R$ with two sides in $\partial \Sigma$; for the closed case, $R$ would be a cylinder. We will parallel the zig-zag argument in the proof of the $S^1$--valued case of Theorem~\ref{T:1Uniqueness}, with $a$ and $a'$ playing the role that $0$ and $1/2$ played in that proof. We use Lemma~\ref{L:ConnectingM} to arrange the following property: For a sequence $0=s_0 < s_1 < \ldots < s_{2n}=1$, $G_s^{-1}(a)$ is connected and $G_s$ is $\pi_1$--surjective on the complement of $G_s^{-1}(a)$, for all $s \in [s_0,s_1] \cup [s_2,s_3] \cup \ldots \cup [s_{2n-2},s_{2n-1}]$, while the same is true for $a'$ on the other intervals $[s_1,s_2] \cup [s_3,s_4] \cup [s_{2n-1},s_{2n}]$. Again, this works because all the homotopies involved are supported in neighborhoods of arcs which can be taken to be disjoint. By further subdivisions if necessary, we can also arrange that, on the $s$--intervals where $G_s^{-1}(a)$ is connected, the restriction of $G_s$ to $G_s^{-1}(a)$ is a Morse function to $a$, and similarly for $a'$. Recall that, in this paper, Morse functions have distinct critical points mapped to distinct critical values. By Theorem~\ref{T:1Existence} we can also arrange that this Morse function $G_s : G_s^{-1}(a) \to a$ is indefinite with connected fibers.

\textbf{The zig-zag argument:} 
Now, in Figure~\ref{F:ZigZag2}, we set up the usual zig-zag argument where the vertical interval $[0,1/2]$ in Figure~\ref{F:ZigZag} is replaced by $R$, with $0$ replaced by $a$ and $1/2$ replaced by $a'$.
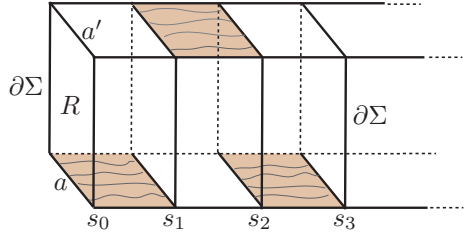
\begin{figure}
\labellist
\small\hair 2pt
\pinlabel $s_0$ at 16 -4
\pinlabel $s_1$ at 40 -4
\pinlabel $s_2$ at 67 -4
\pinlabel $s_3$ at 94 -4
\pinlabel $\partial \Sigma$ at 102 30
\pinlabel $\partial \Sigma$ at -7 39
\pinlabel $R$ at 7 33
\pinlabel $a$ at 4 8
\pinlabel $a'$ at 14 57
\endlabellist
\centering
\includegraphics[scale=1.20]{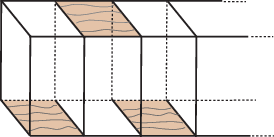}
\caption{The zig-zag argument for Morse $2$--functions}
\label{F:ZigZag2}
\end{figure}

Since the restriction of $G_s$ to $a$ is a Morse function for all $s \in [s_0,s_1]$, it follows that $G_{s_0}^{-1}(a)$ is isotopic, and hence can be taken to be equal, to $G_{s_1}^{-1}(a)$, and in fact $G_s$ can be taken to be independent of $s$ on $G_{s}^{-1}(a)$ for all $s \in [s_0,s_1]$. Similarly, $G_s$ is constant along $G_{s_1}^{-1}(a') = G_{s_2}^{-1}(a')$ for all $s \in [s_1,s_2]$, etc. Note that $G_{s_1}^{-1}(R)$ is a connected cobordism between the connected submanifolds $G_{s_1}^{-1}(a)$ and $G_{s_1}^{-1}(a')$, on each of which $G_{s_1}$ is indefinite and fiber-connected. Thus we can construct an intermediate indefinite, fiber-connected Morse $2$--function $H_{s_1}$ on $X$ which agrees with $G_{s_1}$ over $a$ and $a'$, and is homotopic rel. boundary to $G_{s_1}$ over both $R$ and the closure of the complement of $R$. Then $G_0$ agrees with $H_{s_1}$ over $a$, so there exists an indefinite fiber-connected homotopy between $G_0$ and $H_{s_1}$ (because $-\chi(\Sigma \setminus a) < -\chi(\Sigma)$ and the inductive hypothesis holds). Similarly we construct an appropriate $H_{s_2}$, so that $H_{s_1}$ and $H_{s_2}$ agree over $a'$, so they are homotopic via an indefinite fiber-connected homotopy. This argument can be iterated over each interval $[s_{2i},s_{2i+2}]$, to finish the proof when $\Sigma \neq S^2$.

When $\Sigma = S^2$, there are no $\pi_1$--surjectivity issues, and we can adapt the above zig-zag argument to use Remark~\ref{R:ConnectingFibers} to zig-zag back and forth between two fibers $G_s^{-1}(p)$ and $G_s^{-1}(q)$ as opposed to inverse images of arcs or closed curves $G_s^{-1}(a)$ and $G_s^{-1}(a')$. The rectangle $R$ is replaced by an arc between $p$ and $q$.

\end{proof}

%\bibliographystyle{plain}
%\bibliography{../Main}

\end{document}